\documentclass[12pt,oneside]{amsart}
\usepackage{fouriernc} % Fourier fonts instead of Computer Modern

\usepackage{mathtools}
\usepackage{fancyhdr}
\usepackage{amsthm}
\usepackage{tikz-cd}
\usepackage{framed}
\usepackage{comment}
\usepackage{pdfpages}
\usepackage{amssymb,amsmath}
\usepackage{xcolor}
\usepackage{adjustbox}
\usepackage{makecell}
\usepackage[shortlabels]{enumitem}
\usetikzlibrary{calc,positioning,shapes.geometric,shapes.symbols,shapes.misc,arrows}
\usetikzlibrary{arrows.meta,decorations.markings}
\usepackage{spectralsequences}
\usepackage{todonotes}
\usepackage[htt]{hyphenat}

\usetikzlibrary{external}
\pdfoutput=1

\AtBeginEnvironment{tikzcd}{\tikzexternaldisable}
\AtEndEnvironment{tikzcd}{\tikzexternalenable}

\usepackage[pagebackref]{hyperref}

\hypersetup{
bookmarksnumbered=true,%
colorlinks=true,%
linkcolor=blue,%
citecolor=blue,%
filecolor=blue,%
menucolor=blue,%
urlcolor=blue,%
pdfnewwindow=true,%
pdfstartview=FitBH}

\usepackage[nameinlink,capitalise,noabbrev]{cleveref}

\usepackage[shortalphabetic]{amsrefs}

\tikzcdset{scale cd/.style={every label/.append style={scale=#1},
	cells={nodes={scale=#1}}}}

\theoremstyle{definition} % no italics
\newtheorem{thm}{Theorem}[section] % reset theorem numbering for each chapter
\newtheorem{defn}[thm]{Definition} % definition numbers depend on theorem numbers
\newtheorem{exmp}[thm]{Example}

\newtheorem{prop}[thm]{Proposition}
\newtheorem{lem}[thm]{Lemma}

\newtheorem{fig}[thm]{Figure}
\newtheorem{cor}[thm]{Corollary}

\newcommand{\Z}{\mathbb{Z}}
\newcommand{\Q}{\mathbb{Q}}

\newcommand{\R}{\mathbb{R}}

\newcommand{\F}{\mathbb{F}_2}

\newcommand{\ulF}{\underline{\F}}
\newcommand{\ulZ}{\underline{\Z}}
\newcommand{\ulf}{{\underline{\widehat{f}}}}
\newcommand{\ulg}{\underline{g}}
\newcommand{\ul}[1]{\underline{#1}}
\newcommand{\ulM}{\underline{\mathbb{M}}}

\newcommand{\Si}[1]{\Sigma^{#1}}

\newcommand{\abs}[1]{ \left\lvert#1\right\rvert} % absolute value: single vertical bars
\newcommand{\im}{\text{im }}

\renewcommand\qedsymbol{$\blacksquare$}
\renewenvironment{proof}{{\bfseries Proof.}}{\vspace{-12pt}\begin{flushright}
\qedsymbol
\end{flushright}}
\tikzset{
downtri/.style={
draw,
isosceles triangle,
isosceles triangle apex angle=60,
shape border rotate=-90
},
}

\newcommand{\K}{K}
\newcommand{\resC}{black}
\newcommand{\trC}{black}
\newcommand{\artr}[5]{ % __ to __, xshift, yshift, name
\draw[->,\trC,bend right=2ex] (#1) to node[right,xshift={#3pt},yshift={#4pt}] {\scalebox{0.7}{#5}} (#2);
}
\newcommand{\arres}[5]{ % __ to __, xshift, yshift, name
\draw[->,\resC,bend right=2ex] (#1) to node[above,xshift={#3pt},yshift={#4pt}] {\scalebox{0.7}{#5}} (#2);
}
\newcommand{\res}[2]{ % 1=from, 2=to
\textbf{r}^{#1}_{#2}
}
\newcommand{\tr}[2]{ % 1=to, 2=from
\textbf{t}^{#1}_{#2}
}

\newcommand{\MackC}[6]
{\begin{tikzcd}[ampersand replacement=\&, bend right=2.5ex, column sep=scriptsize]
#1 \arrow[dd, color=#3, "#5"'] \\ \\
#2 \arrow[uu, color=#4, "#6"']
\end{tikzcd}}

\newcommand{\MackKall}[5]
{\begin{tikzcd}[ampersand replacement=\&, bend right=2.5ex, column sep=scriptsize]
	\& #1 \ar[dl, color=#2] \ar[d, color=#2] \ar[dr,color=#2] \& \\
	#1 \ar[dr, color=#3] \ar[ur, color=#4] \& #1 \ar[d, color=#3] \ar[u, color=#4] \& #1 \ar[ul, color=#4]  \ar[dl, bend right, color=#3] \\
	\& #1 \ar[ul, color=#5] \ar[u, color=#5] \ar[ur, color=#5] \&
\end{tikzcd}}

\newcommand{\MackKupperR}[6]
{\begin{tikzcd}[ampersand replacement=\&, column sep=scriptsize]
	\& #1 \ar[dl, color=#3, "#4"'] \ar[d, color=#3, "#5"] \ar[dr,color=#3, "#6"] \& \\
	#2 \& #2 \& #2 \\
	\& 0 \&
\end{tikzcd}}

\newcommand{\MackKgn}[1]
{\begin{tikzcd}[ampersand replacement=\&, bend right=2.5ex, column sep=scriptsize]
	\& #1 \& \\
	0 \& 0 \& 0 \\
	\& 0 \&
\end{tikzcd}}

\def\bdiamond{
\begin{tikzpicture}[x=1.2ex,y=1.2ex]
	\fill (-0.65,0) -- (0,0.65) -- (0.65,0) -- (0,-0.65) -- cycle;
\end{tikzpicture}
}

\def\gcirc{
\begin{tikzpicture}[x=1.2ex,y=1.2ex]
    \draw (0,-0.65) circle (5pt);
    \node at (0,-0.65) {\(n\)};
\end{tikzpicture}
}
\def\trap{
\begin{tikzpicture}
  \node[trapezium, fill=none, draw, inner sep=2.5pt,scale=1] at (0,0) {};
\end{tikzpicture}
}
\def\pent{
\begin{tikzpicture}
  \node[regular polygon, fill=none, draw, regular polygon sides=5, 
minimum width=0pt, inner sep = 0.5ex,] at (0,0) {};
\end{tikzpicture}
}
\def\btrap{
\begin{tikzpicture}
  \node[trapezium, fill=black, draw, inner sep=2.5pt,scale=1] at (0,0) {};
\end{tikzpicture}
}
\def\bpent{
\begin{tikzpicture}
  \node[regular polygon, fill=black, draw, regular polygon sides=5, 
minimum width=0pt, inner sep = 0.5ex,] at (0,0) {};
\end{tikzpicture}
}

\crefname{equation}{}{}

\newcounter{themyfigure}

\makeindex

\renewcommand{\arraystretch}{1.5}

\reversemarginpar

\begin{document}
\title{Klein four 2-slices and the slices of \(\Si{\pm n}H\ulZ\)}
\author{Carissa Slone}
\email{c.slone@uky.edu}
\address{Department of Mathematics, The University of Kentucky, Lexington, KY 40506-0027, USA}

\date{\today}

\begin{abstract} 
We determine a characterization of all 2-slices of equivariant spectra over the Klein four-group \(C_2\times C_2\). We then describe all slices of integral suspensions of the equivariant Eilenberg-MacLane spectrum \(H\ulZ\) for the constant Mackey functor over \(C_2\times C_2\).
\end{abstract}

\maketitle
\setcounter{tocdepth}{2}
\tableofcontents

%%%%%%%%%%%%%%%%%%%%%%%%%%%%%%%%%%%%%
\section{Introduction} \label{sec:intro}
%%%%%%%%%%%%%%%%%%%%%%%%%%%%%%%%%%%%%

The slice filtration, a filtration of genuine \(G\)-spectra, was developed by Hill, Hopkins, and Ravenel in their solution to the Kervaire invariant-one problem \cite{HHR} and is a generalization of Dugger's filtration \cite{D}. It was modeled after Voevodsky's motivic slice filtration \cite{V} and is an equivariant generalization of the Postnikov tower. Rather than decomposing a \(G\)-spectrum into Eilenberg-Mac Lane specra, as does the Postnikov tower, instead the slice filtration decomposes a \(G\)-spectrum into ``\(n\)-slices''.

There is a complete characterization of all \(n\)-slices where \(-1\leq n\leq 1\), listed in \cref{charac:slice:01-1}. This, combined with \cref{rho:susp:commutes}, characterizes all slices in degrees congruent to \(-1\), 0, or 1, modulo the order of \(G\). For \(G=C_2\times C_2\), we are then only missing the \((4k+2)\)-slices. In \cref{sec:2slice} we finish this characterization with the following result.

\textit{\cref{2slice:characterization}: Suppose the only nontrivial homotopy Mackey functors of a \((C_2\times C_2)\)-spectrum \(X\) are \(\ul\pi_1(X)\) and \(\ul\pi_2(X)\) where certain maps in each Mackey functor are injective. Then \(X\) is a 2-slice. Conversely, if \(X\) is a 2-slice, then its only nontrivial homotopy Mackey functors are \(\ul\pi_1(X)\) and \(\ul\pi_2(X)\) where \(\Si{2} H\ul\pi_2(X)\) is a 2-slice and \(\Si{1} H\ul\pi_1(X) \in [2,4]\).}

Much work has been done computing the slices of certain \(RO(G)\)-graded suspensions of \(H_G\ulZ\) including \(G=C_{p^n}\) by \cite{HHR2} and \cite{Y}, and \mbox{\(G=D_{2p}\)} by \cite{Z2}. \cite{GY} computes the slices of \(\Si{n} H_\K \ulF\) where \(\K=C_2\times C_2\). Most of these slices are \(RO(\K)\)-graded suspensions of \(H\ul\pi_{-i}(\Si{-k\rho_K} H_K\ulF)\) for \(i\) in the range \([k+3,4k]\). We primarily focus on the slices of \(\Si{\pm n} H_\K\ulZ\). Although we have cofiber sequences relating \(\Si{n} H_\K\ulZ\) to \(\Si{n} H_\K\ulF\), we can only recover some information about the former from the latter. 

As for the slices of \(\Si{\pm n} H_K\ulZ\), the main result can be summarized as follows:

\textit{Main Result: For \(n<0\), all nontrivial slices of \(\Si{n} H_\K\ulZ\) are given by:
\begin{align*}
    P^{i}_{i}(\Si{n} H_\K\ulZ) \simeq \Si{-V} H_\K\ul M
\end{align*}
where \(i\) is in the range \([4n,n]\). 
For \(0\leq n\leq 5\), \(\Si{n} H\ulZ\) is an \(n\)-slice. Finally, for \(n>5\),
\begin{align*}
    P^i_i(\Si{n} H_\K\ulZ) \simeq \Si{V} H_\K\ul M
\end{align*}
where \(i\) is in the range \([n,4(n-4)]\). 
The representations \(V\) and Mackey functors \(\ul M\) are given in \cref{-nslice:K:Z}, \cref{-4kslices:K:Z}, \cref{-4k-2slices:K:Z}, \cref{slices:K:n:Z}, \cref{slices:K:4k:Z}, and \cref{slices:4k+2:K:Z}.}

The paper is organized as follows. In \cref{sec:background}, we review the slice filtration and relevant dualities. The story for \(\K\) must restrict to the corresponding results for \(C_2\), so we review these results in \cref{sec:C2}. \cref{sec:K:Mackey} provides us with the main Mackey functors for \(\K\) and some pertinent results for \cref{sec:2slice}, in which we characterize all 2-slices over \(\K\). We provide some slice towers in \cref{sec:cotowers:K:Z} and describe the slices of \(\Si{-n} H_\K\ulZ\) in
\cref{sec:-nslice:K:Z}. In \cref{sec:slices:n:K:Z}, we use Brown-Comenetz duality and the slices of \(\Si{-n} H\ulZ\) to obtain the slices of \(\Si{n} H_\K\ulZ\). We then compute the homotopy Mackey functors of the slices of \(\Si{\pm n} H_\K\ulZ\) in \cref{sec:htpycomp}. Finally, we provide some examples of the slice spectral sequence for \(\Si{\pm n} H_\K\ulZ\) in \cref{sec:SSS}.

The author is grateful for the guidance of Bert Guillou and some helpful conversations with Vigleik Angelveit. Figures \ref{fig:SSS:-1..-5:HZ}, \ref{fig:SSS:-7..-9:HZ}, \ref{fig:SSS:6..8:HZ}, \ref{fig:SSS:10..11:HZ}, \ref{fig:SSS:12:HZ}, and \ref{fig:SSS:14:HZ} were created using Hood Chatham's \texttt{spectralsequences} package.

%%%%%%%%%%%%%%%%%%%%%%%%%%%%%%%%%%%%%
\section{Background} \label{sec:background}
%%%%%%%%%%%%%%%%%%%%%%%%%%%%%%%%%%%%%

In this section we give background for the slice filtration as well as Brown-Comenetz and Anderson duality, and \(\K\)-representations. Here, except for \cref{subsec:signrep}, \(G\) is any finite group.

\subsection{The Slice Filtration} \label{slice_filt_intro}
\phantom{a}

We start with a brief review of the equivariant slice filtration. For more details see \cite[Section 4]{HHR}.

\begin{defn} \label{def:tau_n}
Let \(\text{Sp}^G\) be the category of genuine \(G\)-spectra. Let \(\tau^G_{\geq n} \subseteq \text{Sp}^G\) be the localizing subcategory generated by \(G\)-spectra of the form \(\Si{\infty}_G G_+ \wedge_H S^{k\rho_H}\) where \(H\leq G\), \(\rho_H\) is the regular representation of \(H\), and \(k\abs{H} \geq n\). We write \(X\geq n\) to mean that \(X\in \tau^G_{\geq n}\).
\end{defn}

\begin{defn} \label{def:lessthan}
We say that \(X<n\) if 
\begin{align*}
    [S^{k\rho_H+r}, X]^H = 0
\end{align*}
for all \(r\geq 0\) and all subgroups \(H\leq G\) such that \(k\abs{H}\geq n\).
\end{defn}

\begin{thm}{\cite[Corollary 2.9, Theorem 2.10]{HY}} \label{thm:atleastn}
Let \(n\geq 0\). Then \(X\geq n\) if and only if 
\begin{align*}
    \pi_k(X^H) = 0 \quad \text{for } \quad  k<\frac{n}{\abs{H}}.
\end{align*}
\end{thm}

\begin{prop} \label{charac:slice:01-1}
\phantom{a}    
\begin{enumerate}
    \item \cite[Proposition 4.50]{HHR} \(X\) is a 0-slice if and only if \(X\simeq H\ul M\) for some \(\ul M\in \text{Mack}(G)\).
    \item \cite[Proposition 4.50]{HHR} \(X\) is a 1-slice if and only if \(X\simeq \Si{1} H\ul M\) for some \(\ul M\in \text{Mack}(G)\) with injective restrictions.
    \item \cite[Theorem 6-4]{U} \(\Si{-1} H\ul M\) is a \((-n)\)-slice iff \(\ul M\) has surjective transfers for \(\abs{H}\geq n\) and \(M(G/H) = 0\) for \(\abs{H}<n\).
\end{enumerate}
\end{prop}

It is important to note that \cite{HHR} uses the original slice filtration whereas we employ the regular slice filtration from \cite{U}. Except for an indexing difference of one, the results are the same.

\begin{prop}{\cite[Corollary 4.25]{HHR}} \label{rho:susp:commutes}
For any \(k\in\Z\),
\begin{align*}
    P^{k+\abs{G}}_{k+\abs{G}} (\Si{\rho} X) &\simeq \Si{\rho} P^k_k(X).
\end{align*}
\end{prop}

That is, suspension by the regular representation commutes with the slice filtration.

Given some surjection of groups \(\phi_N:G\rightarrow G/N\) where \(N\unlhd G\), there is a geometric pullback functor \(\phi^*_N : \text{Sp}^{G/N} \rightarrow \text{Sp}^G\) \cite[Definition 4.1]{H}.

\begin{prop}{\cite[Conjecture 4.11]{H}, \cite[Corollary 4-5]{U}} \label{pullback:slice}
Let \(N\unlhd G\). If the \((G/N)\)-spectrum \(X\) is a \(k\)-slice over \(G/N\), then \(\phi^*_N X\) is a \(k[G:N]\)-slice over \(G\).
\end{prop}

\subsection{Brown-Comenetz and Anderson Duality} \label{BCandA:duality}
\phantom{A}

As in \cite{HS}, we write \(I_{\Q/\Z}\) to indicate the representing \(G\)-spectrum of the cohomology theory \(X\mapsto \text{Hom}(\pi_{-\ast}^G X, \Q/\Z)\). The Brown-Comenetz dual of \(X\) is then defined to be the function \(G\)-spectrum \(F(X, I_{\Q/\Z})\). Similarly, \(I_{\Q}\) represents \(X\mapsto \text{Hom}(\pi_{-\ast} X, \Q)\) and \(I_\Q X = F(X, I_{\Q})\). Finally, the Anderson dual of \(X\) is \(I_{\Z} X = F(X,I_\Z)\), where \(I_\Z\) is the fiber of the natural map \(I_{\Q} \rightarrow I_{\Q/\Z}\).

\begin{exmp}{\cite[Section 3A]{GM}, \cite{HS}}
For a torsion Eilenberg-MacLane spectrum \(H\ul M\),
\begin{align*}
    I_\Z H\ul M &\simeq \Si{-1} H\ul M^*
\end{align*}
and
\begin{align*}
    I_{\Q/\Z} H\ul M &\simeq H\ul M^*.
\end{align*}
One should note that \cite{HS} deals with non-equivariant spectra and \cite[Section 3A]{GM} refers specifically to \(H\ul M\) as an \(\F\)-torsion spectrum. This example, however, follows easily from their work and \cite[Corollary I.7.3]{U2}. See \cite[Section 3A, Section 3B]{GM} for a more detailed discussion of equivariant Anderson duality.
\end{exmp}

\begin{prop}(\cite[Theorem I.7.7, Theorem I.7.8]{U2}) \label{duality:prop}
For a spectrum \(X\),
\begin{align*}
    X\geq n &\Leftrightarrow I_{\Q/\Z} X\leq -n.
\end{align*}
In particular,
\begin{align*}
    P^n_k I_{\Q/\Z} X &\cong I_{\Q/\Z} P^{-k}_{-n} X.
\end{align*}
\end{prop}

That is, the Brown-Comenetz dualization functor dualizes slice status.

\subsection{\texorpdfstring{\(\K\)}{K4}-Representations} \label{subsec:signrep}
\phantom{a}

Recall that \(\sigma\) is the one-dimensional sign representation of \(C_2\). The sign representations of \(\K\) are then the pullbacks \(p_1^*\sigma\), \(m^*\sigma\), and \(p_2^*\sigma\) where \(p_1\), \(m\), and \(p_2\) are
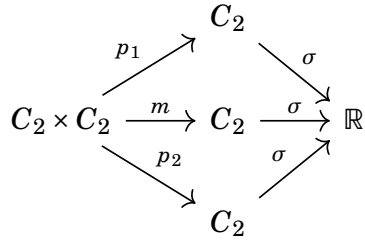
\begin{equation*}
    \begin{tikzcd}[ampersand replacement=\&]
    \& C_2 \arrow[dr, "\sigma"] \\
    C_2\times C_2 \arrow[ur, "p_1"] \arrow[r, "m"] \arrow[dr, "p_2"] \& C_2 \arrow[r, "\sigma"] \& \R \\
    \& C_2 \arrow[ur, "\sigma"] \&
    \end{tikzcd}
\end{equation*}

As in \cite{GY}, we will write 
\begin{align*}
    \alpha=p_1^*\sigma, \qquad \beta=p_2^*\sigma, \qquad \text{and } \qquad \gamma=m^*\sigma.
\end{align*}
The regular representation of \(\K\) is then \(\rho_{\K} = 1 + \alpha + \beta + \gamma\).

%%%%%%%%%%%%%%%%%%%%%%%%%%%%%%%%%%%%%
\section{Review of \texorpdfstring{\(C_2\)}{C2}} \label{sec:C2}
%%%%%%%%%%%%%%%%%%%%%%%%%%%%%%%%%%%%%

The Lewis diagram for a \(C_2\)-Mackey functor takes the form
\begin{equation*}
\begin{tikzcd}
    \ul M(C_2) \arrow[dd, color = \resC, bend right = 20] \\ \\
    \ul M(e) \arrow[uu, color = \trC, bend right = 20]
\end{tikzcd}
\end{equation*}
where \(\ul M(e)\) has a \(C_2/e\) action. Arrows going down the diagram are called restrictions and arrows going up are called transfers.

For the sake of clarity in large diagrams, in a general Mackey functor \(\ul M\) we will henceforward denote \(\ul M(H)\) by \(M_H\).

\begin{prop} \label{C2:2slice}
A \(C_2\)-spectrum \(Y\) is a 2-slice over \(C_2\) if and only if the only nontrivial homotopy Mackey functors are \(\ul\pi_1(Y)\) and \(\ul\pi_2(Y)\), where
\begin{enumerate}
    \item The restriction map \(\res{C_2}{e}: \pi_2^{C_2}(Y) \rightarrow \pi_2^e(Y)\) is injective \label{2slice:C2:M}
    \item \(\pi_1^e(Y) = 0\). \label{2slice:C2:N}
\end{enumerate}
That is, both \(\Si{2} H\ul\pi_2(Y)\) and \(\Si{1} \ul\pi_1(Y)\) are 2-slices.
\end{prop}
\begin{proof}
Let \(\ul M=\ul\pi_2(Y)\) and \(\ul N=\ul\pi_1(Y)\) and suppose that they are the only nontrivial homotopy Mackey functors of \(Y\). We show that \(\Si{2} H\ul M\) and \(\Si{1} H\ul N\) are 2-slices if and only if \cref{2slice:C2:M} and \cref{2slice:C2:N} hold.

First, consider \(\Si{2} H\ul M\). We immediately see that this spectrum is at least two. As for \(\Si{2} H\ul M\leq 2\), the cofiber sequence
\begin{align*}
    (S^{-\sigma} \rightarrow S^0 \rightarrow C_2/e_+) \wedge \Si{2} H\ul M
\end{align*}
provides the homotopy
\begin{center}
\begin{tikzpicture}
    \node at (-4.25,1) {\(\ul\pi_2 (\Si{-\sigma} H\ul M) \quad =\)};
    \node (lMC) at (-1.5,2) {\(\ker(\res{C_2}{e})\)};
    \node (lMe) at (-1.5,0) {\(0\)};
    \node at (0,1) {and};
    \node at (2.5,1) {$\ul\pi_1 (\Si{-\sigma} H\ul M) \quad =$};
    \node (MC) at (5,2) {$M_e / \im \res{C_2}{e}$};
    \node (Me) at (5,0) {$\Z_\sigma \otimes M_e$};
    \arres{MC}{Me}{0}{0}{}
    \artr{Me}{MC}{0}{0}{}
\end{tikzpicture}
\end{center}

Then the cofiber sequences 
\begin{align*}
	(S^{-k\sigma} \rightarrow S^{-(k-1)\sigma} \rightarrow C_2/e_{+}\wedge S^{-(k-1)\sigma})\wedge \Si{2} H\ul M
\end{align*}
reveal that \(\ul\pi_2^{C_2}(\Si{-k\sigma} \Si{2} H\ul M) = \ker(\res{C_2}{e})\) for all \(k\geq 2\).

By definition, \(\Si{2} H\ul M\leq 2\) if and only if
\(\ul\pi_2^{C_2}(\Si{-k\sigma} \Si{2} H\ul M) = 0\) for all \(k\geq 2\). Consequently, \(\Si{2} H\ul M\) is a 2-slice if and only if \(\res{C_2}{e}\) is injective .

As for \(\Si{1} H\ul N\), by \cref{thm:atleastn}, we find that \(Y\geq 2\) if and only if \(N_{e} = 0\). But then \(\Si{1} H\ul N\) is the pullback \(\phi^*_\K \Si{1} HN_{e}\) and consequently, a 2-slice by \cref{pullback:slice}.

Conversely, assume \(Y\) is a 2-slice. By \cref{thm:atleastn}, we know \(\ul\pi_1^e(Y) = 0\) and \(\ul \pi_k(Y) = \ul 0\) for \(k\leq 0\). For \(k\geq 3\), the slice status of \(Y\) and \cref{def:lessthan} dictate that \(\ul \pi_k(Y) = \ul 0\). Consequently, the only nontrivial homotopy Mackey functors of \(Y\) and \(\ul\pi_1(Y)=\ul N\) and \(\ul\pi_2(Y)=\ul M\). Thus, we have a fiber sequence \(\Si{2} H\ul M\rightarrow Y\rightarrow \Si{1} H\ul N\).

If either \(\Si{2} H\ul M\) or \(\Si{1} H\ul N\) is trivial, the result follows from above. So assume that both are nontrivial. Because \(N_e = 0\), \(\Si{1} H\ul N\) is a 2-slice. In particular, \(\Si{1} H\ul N\leq 2\), and since \(Y\leq 2\), we have that \(\Si{2} H\ul M\leq 2\). Consequently, as \(\Si{2} H\ul M\geq 2\), we have that \(\Si{2} H\ul M\) is a 2-slice.
\end{proof}

From \cite{GY}, we will see the \(C_2\)-Mackey functors in \cref{tab-oldC2Mackey}.

\begin{table}[h] 
\caption[Familiar {$C_2$}-Mackey functors.]{Familiar $C_2$-Mackey functors.}
\label{tab-oldC2Mackey}
\begin{center}
	{\renewcommand{\arraystretch}{1.5}
		\begin{tabular}{|c|c|c|}
			\hline
			  $\ulZ$ 
			& $\ulZ^*$
			& $\ulf$
			\\ \hline
			  $\MackC{\Z}{\Z}{blue}{orange}{1}{2}$
			& $\MackC{\Z}{\Z}{orange}{blue}{2}{1}$
			& $\MackC{0}{\Z_\sigma}{white}{white}{}{}$
			\\
			\hline
			  $\ulF$
			& $\ulF^*$  
			& $\ul f$
			\\ \hline
			  $\MackC{\F}{\F}{\resC}{white}{1}{}$
            & $\MackC{\F}{\F}{white}{\resC}{}{1}$
            & $\MackC{0}{\F}{white}{white}{}{}$
			\\
			\hline
			$\ulg$ && \\
			$\MackC{\F}{0}{white}{white}{}{}$ && \\ \hline 
	\end{tabular} }
\end{center}
\end{table}

\begin{prop}[\cite{GY}] \label{C2:equiv}
There are equivalences
\begin{enumerate}
    \item \(\Si{4} H_{C_2}\ulZ \simeq \Si{2\rho} H_{C_2}\ulZ^*\)
    \item \(\Si{2} H_{C_2}\ul f \simeq \Si{\rho} H_{C_2}\ulF^*\)
    \item \(\Si{1} H_{C_2}\ulg \simeq \Si{\rho} H_{C_2}\ulg\) \label{C2:g:2nslice}
\end{enumerate}
\end{prop}

Note, in particular, that \cref{C2:g:2nslice} makes \(\Si{n} H_{C_2}\ulg\) a \(2n\)-slice for any \(n\in \Z\).

\begin{prop} \label{ksigma+r:homotopy:C2:Z}
For \(k,r\geq 0\),
\begin{align*}
	\ul \pi_i (\Si{k\sigma + r} H_{C_2}\ulZ) &= \left\lbrace \begin{array}{ll}
		\ulZ & i=k+r, \text{ } k \text{ even} \\
		\ulf & i=k+r, \text{ } k \text{ odd} \\
		\ulg & i\in[r,k+r-1], \text{ } i\equiv r\pmod 2
	\end{array}\right.
\end{align*}
\end{prop}
\begin{proof}
We calculate \(\ul \pi_i (\Si{k\sigma} H_{C_2}\ulZ)\) and then shift the degrees by \(r\). The result follows by induction on \(j\geq 1\) using the cofiber sequence
\begin{align*}
    \Si{(j-1)\sigma+2} H\ulZ \simeq \Si{j\sigma} H\ulZ^* \rightarrow \Si{j\sigma} H\ulZ \rightarrow \Si{j\sigma} H\ulg\simeq H\ulg.
\end{align*}
\end{proof}

%%%%%%%%%%%%%%%%%%%%%%%%%%%%%%%%%%%%%
\subsection{Slices of \texorpdfstring{\(\Si{\pm n} H_{C_2}\ulZ\)}{HZ}} \label{towers:C2:Z:V} \hspace{2in}
%\subsection{\for{toc}{Slices of $\Sigma^{\pm n} H Z$}\except{toc}{Slices of $\Sigma^{\pm n} H_{C_2} \ulZ$}}
%%%%%%%%%%%%%%%%%%%%%%%%%%%%%%%%%%%%%

Because the slices for \(\Si{\pm n} H_{\K}\ulZ\) over \(\K\) restrict to the corresponding slices over \(C_2\), we must know these slices over \(C_2\). They are as follows.

\begin{prop} \label{C2:Z:slice:0-6}
\(\Si{n} H_{C_2}\ulZ\) is an \(n\)-slice for \(0\leq n\leq 6\).
\end{prop}
\begin{proof}
For \(0\leq n\leq 2\), this follows from \cref{charac:slice:01-1} and \cref{C2:2slice}. By the same results, \(\Si{n} H\ulZ^*\) is an \(n\)-slice for \(0\leq n\leq 2\). Furthermore, \(\Si{-1} H\ulZ^*\) is a \((-1)\)-slice by \cref{charac:slice:01-1}.
Then, by \cref{C2:2slice} and \cref{rho:susp:commutes}, 
\(\Si{n} H\ulZ \simeq \Si{n-4+2\rho} H\ulZ^*\) is a \(n\)-slice for \(3\leq n\leq 6\).
\end{proof}

\begin{prop} \label{tower:C2:Z:n}
Let \(n\geq 7\) and take \(r\equiv n\pmod 4\) with \(3\leq r\leq 6\). The slice tower of \(\Si{n} H_{C_2}\ulZ\) is
\begin{equation*}
\begin{tikzcd}
    \arrow[r] P^{2n-6}_{2n-6} = \Si{n-3} H\ulg & 
    \arrow[d] \Si{n} H\ulZ \\
    \arrow[r] P^{2n-10}_{2n-10} = \Si{n-5} H\ulg & 
    \arrow[d] \Si{2\rho + (n-4)} H\ulZ \\
    & \arrow[d] \vdots \\
    \arrow[r] P^{n+r-2}_{n+r-2} = \Si{\frac{n+r}{2}-1} H\ulg & \arrow[d] \Si{\frac{n-r-4}{2}\rho+(r+4)} H\ulZ \\
    & P^n_n = \Si{\frac{n-r}{2}\rho + r} H\ulZ
\end{tikzcd}	
\end{equation*}
\end{prop}
\begin{proof}
The exact sequence \(\ulZ^* \rightarrow \ulZ \rightarrow \ulg\) and the homotopy equivalence \(\Si{n} H\ulZ \simeq \Si{n-4+2\rho} H\ulZ^*\) provide the fiber sequence 
\begin{align*}
    \Si{n-3} H\ulg \rightarrow \Si{n} H\ulZ \rightarrow \Si{2\rho + (n-4)} H\ulZ.
\end{align*}
We then augment this sequence with its appropriate \(2\rho\) suspensions until \(\Si{\frac{n-r}{2}\rho + r} H\ulZ\), a slice, is reached.
\end{proof}

\begin{cor} \label{2kslice:C2:HZ}
Let \(n\geq 7\) and take \(r\equiv n\pmod 4\) with \(3\leq r\leq 6\). The \((2k)\)-slices of \(\Si{n} H_{C_2}\ulZ\) are
\begin{align*}
	P^{2k}_{2k}( \Si{n} H_{C_2}\ulZ ) &\simeq \Si{k} H_{C_2}\ulg
\end{align*}
for \(k\equiv n+1\pmod 2\) and \(k\in \left[\frac{n+r}{2}-1,\ldots, n-3\right]\).
\end{cor}

\begin{prop} \label{tower:C2:Z:-n}
Let \(n\geq 1\) and take \(r\equiv n\pmod 4\) with \(1\leq r\leq 4\). The slice cotower of \(\Si{-n} H_{C_2}\ulZ\) is
\begin{equation*}
\begin{tikzcd}
    \arrow[d] \Si{-\frac{n-r}{2}\rho-r} H\ulZ^* = P^{-n}_{-n} & \\
    \arrow[d] \arrow[r] \Si{-\frac{n-r}{2}\rho -r} H\ulZ & \Si{-\frac{n+r}{2}} H\ulg = P^{-n-r}_{-n-r} \\
    \arrow[d] \vdots & \\
    \arrow[d] \arrow[r] \Si{-2\rho -n+4} H\ulZ & \Si{-n+2} H\ulg = P^{-2n+4}_{-2n+4} \\
    \arrow[r] \Si{-n} H\ulZ & \Si{-n} H\ulg = P^{-2n}_{-2n}
\end{tikzcd}		
\end{equation*}
\end{prop}
\begin{proof}
The exact sequence \(\ulZ^* \rightarrow \ulZ \rightarrow \ulg\) and the homotopy equivalence \(\Si{-n} H\ulZ \simeq \Si{2\rho-n-4} H\ulZ\) provide the cofiber sequence
\begin{align*}
    \Si{-2\rho-n+4} H\ulZ \rightarrow \Si{-n} H\ulZ \rightarrow \Si{-n} H\ulg.
\end{align*}
We then augment this sequence with its appropriate \(-2\rho\) suspensions until \(\Si{-\frac{n-r}{2}\rho-r} H\ulZ\), a slice, is reached.
\end{proof}

\begin{cor} \label{-2kslice:C2:HZ}
Let \(n\geq 1\) and take \(r\equiv n\pmod 4\) with \(1\leq r\leq 4\). The \((-2k)\)-slices of \(\Si{-n} H_{C_2}\ulZ\) are
\begin{align*}
    P^{-2k}_{-2k}( \Si{-n} H_{C_2}\ulZ ) &\simeq \Si{-k} H_{C_2}\ulg
\end{align*}
for \(k\equiv n\pmod 2\) and \(k\in \left[\frac{n+r}{2},n\right]\).
\end{cor}

%%%%%%%%%%%%%%%%%%%%%%%%%%%%%%%%%%%%%
\section{\texorpdfstring{\(\K\)}{K}-Mackey Functors \label{sec:K:Mackey} \hspace{2in}}
%%%%%%%%%%%%%%%%%%%%%%%%%%%%%%%%%%%%%

Recall that \(\K=C_2\times C_2\) and let \(L\), \(D\), and \(R\) be the left, diagonal, and right subgroups of \(\K\), respectively. The Lewis diagram for a \(\K\)-Mackey functor takes the form

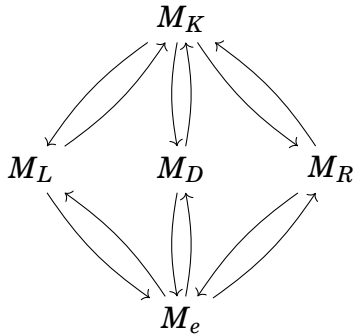
\begin{center}
\begin{tikzpicture}
	\node (MK) at (0,4) {$M_\K$};
	\node (ML) at (-2,2) {$M_L$};
	\node (MD) at (0,2) {$M_D$};
	\node (MR) at (2,2) {$M_R$};	
	\node (Me) at (0,0) {$M_e$};
	% maps
	\arres{MK}{ML}{0}{0}{}
	\arres{MK}{MD}{0}{0}{}
	\arres{MK}{MR}{0}{0}{}
	\artr{ML}{MK}{0}{0}{}
	\artr{MD}{MK}{0}{0}{}
	\artr{MR}{MK}{0}{0}{}
	\arres{ML}{Me}{0}{0}{}
	\arres{MD}{Me}{0}{0}{}
	\arres{MR}{Me}{0}{0}{}
	\artr{Me}{ML}{0}{0}{}
	\artr{Me}{MD}{0}{0}{}
	\artr{Me}{MR}{0}{0}{}
\end{tikzpicture}
\end{center}

with a Weyl group action \(W_G(H) \circlearrowright M_H\) at each level. If we do not indicate the Weyl group actions for a specific Mackey functor, then they are trivial.

A number of Mackey functors from \cite{GY} will make their appearance.

\begin{table}[h] 
\caption[Familiar {$\K$}-Mackey functors.]{Familiar $\K$-Mackey functors.}
\label{tab-oldKMackey}
\begin{center}
	{\renewcommand{\arraystretch}{1.5}
		\begin{tabular}{|c|c|}
			\hline
			  $\ul {mg}$ 
			& $\ul {m}$
			\\ \hline
			  $\MackKupperR{\F^2}{\F}{\resC}{p_1}{\nabla}{p_2}$
			& $\MackKupperR{\F}{\F}{\resC}{1}{1}{1}$
			\\
			\hline
			  $\phi_{LDR}^*\ulF$
			& $\ulg^n$
			\\ \hline
			  $\MackKupperR{\F^3}{\F}{\resC}{p_1}{p_2}{p_3}$
			& $\MackKgn{\F^n}$
			\\
			\hline
	\end{tabular} }
\end{center}
\end{table}

We will also see the duals of the Mackey functors in \cref{tab-oldKMackey}.

\begin{prop}(\cite[Proposition 4.8]{GY}) \label{m:mg:equiv}
There are equivalences 
\begin{align}
    \Si{-\rho} H_\K\ul m &\simeq \Si{-2} H_\K\ul{mg}^* \\
    \Si{-2} H_\K\ul m^* &\simeq \Si{-\rho} H_\K\ul{mg}
\end{align}
\end{prop}

We will also see the new Mackey functors in \cref{tab-newKMackey}.

\begin{table}[h] 
\caption[{New $\K$}-Mackey functors.]{New $\K$-Mackey functors.}
\label{tab-newKMackey}
\begin{center}
	{\renewcommand{\arraystretch}{1.5}
		\begin{tabular}{|c|c|}
			\hline 
			  $\ulZ$
			& $\ulZ^*$
			\\ \hline
         $\MackKall{\Z}{blue}{blue}{orange}{orange}$
			& $\MackKall{\Z}{orange}{orange}{blue}{blue}$
			\\ \hline
			  $\ulZ(2,1)$
			& $\ulZ(2,1)^*$
			\\ \hline
			  $\MackKall{\Z}{orange}{blue}{blue}{orange}$
			& $\MackKall{\Z}{blue}{orange}{orange}{blue}$
			\\ \hline
			$\ulM$ & \\ \hline 
			${\begin{tikzcd}[ampersand replacement=\&, bend right=2.5ex, column sep=scriptsize]
	            \& \Z/4 \ar[dl, color=blue] \ar[d, color=blue] \ar[dr,color=blue] \& \\
	            \Z/2 \ar[ur, color=orange] \& \Z/2 \ar[u, color=orange] \& \Z/2 \ar[ul, color=orange] \\
	            \& 0 \&
                \end{tikzcd}}$
                & \\ \hline 
	\end{tabular} }
\end{center}
\end{table}

In \cref{tab-newKMackey}, \(\ulZ^*\) is the dual to the constant Mackey functor \(\ulZ\) and \(\ulZ(2,1)^*\) is the dual to \(\ulZ(2,1)\), so named for consistency with \cite{Y}. The Mackey functor \(\ulM\) results from the injection \(\ulZ^*\hookrightarrow \ulZ \twoheadrightarrow \ulM\). In each Mackey functor, the blue arrows are multiplication by one and the orange arrows are multiplication by two.

\begin{prop} \label{equivalence:ZZ*}
We have the equivalence
\begin{align*}
    \Si{-\rho} H_\K\ulZ &\simeq \Si{-4} H_\K\ulZ^*.
\end{align*}
\end{prop}
\begin{proof}
This follows from \cref{homotopy:-rho:M}.
\end{proof}

\subsection{Induced Mackey Functors} \label{Induction_Formula}
\phantom{a}

We will now give an explicit description of \(\K\)-Mackey functors induced from the cyclic subgroups.

We recall the following standard result.

\begin{lem}[Shearing isomorphism] \label{shearingiso}
Let \(M\) be a \(\Z[G]\)-module and \(H\) a subgroup of	\(G\). Then we have an isomorphism of \(\Z[G]\)-modules,
\begin{center}
	\begin{tikzpicture}
		\node (L) at (-1.5,0) {$\Z[G]\otimes_{H} M$};
		\node (M) at (0,0) {$\cong$};
		\node (R) at (1.5,0) {$\Z[G/H]\otimes M$};
		\node (loop1) at (-.8,0) {$\phantom{A}$};
		\node (loop2) at (2.3,0) {$\phantom{A}$};
		%
		%240,300
		\draw[->,out=240,in=300,loop, looseness=4] (loop1) to 
		node[below] {\scalebox{0.6}{$H$}} (loop1);
		\draw[->,out=240,in=300,loop, looseness=4] (loop2) to node[below] {\scalebox{0.6}{$G$}} (loop2);
	\end{tikzpicture}		
\end{center}
where \(G\) acts on the left of \(\Z[G]\) and diagonally on \(\Z[G/H]\otimes M\) via the map \(g\otimes m \mapsto gH \otimes g\cdot m\).
\end{lem}

Let \(\ul N\) be a \(C_2\)-Mackey functor. Take \(h\in\K\) and set \(\Delta_h:N_e \rightarrow \Z[\K] \otimes_{\Z[L]} N_e\) and \(\nabla_h:\Z[\K] \otimes_{\Z[L]} N_e \rightarrow N_e\) to be \(n\mapsto (e+h)\otimes n\) and \(e\otimes n,h\otimes n\mapsto n\), respectively.

Then \(\uparrow^\K_L \ul N\) is
\begin{equation*}
\raisebox{12ex}{$\uparrow^\K_L \ul N \quad \cong$} \quad 
	\begin{tikzpicture}
		\node (MK) at (0,4) {$N_{C_2}$};
		\node (ML) at (-2,2) {$\Z[\K/L] \otimes N_{C_2}$};
		\node (MD) at (0,2) {$N_e$};
		\node (MR) at (2,2) {$N_e$};	
		\node (Me) at (0,0) {$\Z[\K] \otimes_{\Z[L]} N_e$};
		% maps
		\arres{MK}{ML}{0}{0}{$\Delta$}
		\arres{MK}{MD}{-6}{-16}{$\res{C_2}{e}$}
		\arres{MK}{MR}{-3}{-15}{$\res{C_2}{e}$}
		\artr{ML}{MK}{0}{0}{$\nabla$}
		\artr{MD}{MK}{0}{0}{$\tr{C_2}{e}$}
		\artr{MR}{MK}{0}{0}{$\tr{C_2}{e}$}
		\draw[->,\resC,bend right=2ex] (ML) to node[sloped, anchor=center, xshift={0pt},yshift={-5pt}] {\scalebox{0.5}{$\text{id}\otimes \res{C_2}{e}$}} (Me);
		\arres{MD}{Me}{-5}{0}{$\Delta_d$}
		\arres{MR}{Me}{0}{0}{$\Delta_r$}
		\draw[->,\trC,bend right=2ex] (Me) to node[sloped, anchor=center, xshift={0pt},yshift={-5pt}] {\scalebox{0.5}{$\text{id} \otimes \tr{C_2}{e}$}} (ML);
		\artr{Me}{MD}{0}{0}{$\nabla_d$}
		\artr{Me}{MR}{0}{0}{$\nabla_r$}
	\end{tikzpicture}
\end{equation*}

The Weyl group \(W_\K(D)\) acts on \(\uparrow_L^K \ul{N}(D)= N_e\) via the isomorphism \(W_\K(D)\cong C_2\) and the given action of \(C_2 = W_{C_2}(e)\) on \(N_e\). Similarly for the action of \(W_{\K}(R)\) on \(N_e\). As for \(L\), \(W_\K(L)\) acts only on \(\Z[\K/L]\). Finally, \(W_\K(e) = K\) acts on the \(\Z[\K]\) factor of \(\Z[\K]\otimes_{\Z[L]} N_e\).

We are now going to define the unit map \(\ul{M}\rightarrow  \uparrow^\K_{L} \downarrow^\K_L \ul M\). The pullback along \(\K\twoheadrightarrow \K/L\cong C_2\) of \(\Z\hookrightarrow \Z[C_2]\) is an inclusion \(\Z\hookrightarrow \Z[\K/L]\). Tensoring with \(M_e\) gives
\begin{equation*}
\begin{tikzcd}
    M_e \arrow[r, hook, "i_L"] & \Z[\K/L]\otimes M_e.
\end{tikzcd}
\end{equation*}
We will also use \(i_L\) to denote the inclusion of \(\K/L\) fixed points
\begin{equation*}
\begin{tikzcd}
    M_L \arrow[r, hook, "i_L"] & \Z[\K/L]\otimes M_L.
\end{tikzcd}
\end{equation*}

Consider
\begin{align*}
    \nabla^r: \Z[\K/L]\otimes M_e \rightarrow M_e \qquad \text{and} \qquad \Delta^r:M_e \rightarrow \Z[\K/L]\otimes M_e.
\end{align*}

Let \(\nabla^r\) be the action map and define \(\Delta^r\) by \(m\mapsto e\otimes m + r\otimes r\cdot m\). Then, for \(\ul M\in \text{Mack}(\K)\), the map \(\ul M \rightarrow \uparrow^\K_{L} \downarrow^\K_L \ul M\) is
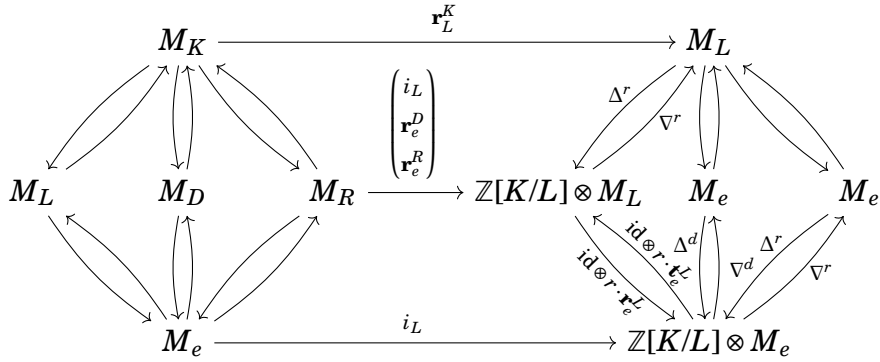
\begin{equation*}
\begin{tikzpicture}
	\node (MK) at (0,4) {$M_\K$};
	\node (ML) at (-2,2) {$M_L$};
	\node (MD) at (0,2) {$M_D$};
	\node (MR) at (2,2) {$M_R$};	
	\node (Me) at (0,0) {$M_e$};
	\node (rMK) at (7,4) {$M_L$};
	\node (rML) at (5,2) {$\Z[\K/L] \otimes M_L$};
	\node (rMD) at (7,2) {$M_e$};
	\node (rMR) at (9,2) {$M_e$};
	\node (rMe) at (7,0) {$\Z[\K/L] \otimes M_e$};
	% maps
	\arres{MK}{ML}{0}{0}{}
	\arres{MK}{MD}{0}{0}{}
	\arres{MK}{MR}{0}{0}{}
	\artr{ML}{MK}{0}{0}{}
	\artr{MD}{MK}{0}{0}{}
	\artr{MR}{MK}{0}{0}{}
	\arres{ML}{Me}{0}{0}{}
	\arres{MD}{Me}{0}{0}{}
	\arres{MR}{Me}{0}{0}{}
	\artr{Me}{ML}{0}{0}{}
	\artr{Me}{MD}{0}{0}{}
	\artr{Me}{MR}{0}{0}{}
	\arres{rMK}{rML}{0}{0}{$\Delta^r$}
	\arres{rMK}{rMD}{0}{0}{}
	\arres{rMK}{rMR}{0}{0}{}
	\artr{rML}{rMK}{0}{0}{$\nabla^r$}
	\artr{rMD}{rMK}{0}{0}{}
	\artr{rMR}{rMK}{0}{0}{}
	\draw[->,\resC,bend right=2ex] (rML) to node[sloped, anchor=center, xshift={0pt},yshift={-5pt}] {\scalebox{0.7}{$\text{id}\otimes r\cdot \res{L}{e}$}} (rMe);
	\arres{rMD}{rMe}{-5}{2}{$\Delta^d$}
	\arres{rMR}{rMe}{0}{0}{$\Delta^r$}		
	\draw[->,\trC,bend right=2ex] (rMe) to node[sloped, anchor=center, xshift={0pt},yshift={5pt}] {\scalebox{0.7}{$\text{id}\otimes r\cdot \tr{L}{e}$}} (rML);
	\artr{rMe}{rMD}{0}{0}{$\nabla^d$}
	\artr{rMe}{rMR}{0}{0}{$\nabla^r$}
	\draw[->] (MK) to node[above, xshift={-1pt}] {\scalebox{0.7}{$\res{\K}{L}$}} (rMK);
	\draw[->] (MR) to node[above, xshift={-1pt}] {\scalebox{0.7}{$\begin{pmatrix}i_L \\ \res{D}{e} \\ \res{R}{e} \end{pmatrix}$}} (rML);
	\draw[->] (Me) to node[above, xshift={-1pt}] {\scalebox{0.7}{$i_L$}} (rMe);
\end{tikzpicture}
\end{equation*}
Again, \(W_\K(D)\) acts on \(M_e\) via the \(C_2\)-action \(W_{C_2}(e)\circlearrowright M_e\). The same applies for \(W_\K(R)\circlearrowright M_e\). Note we have used \cref{shearingiso} to rewrite the bottom group in \(\uparrow^\K_L \downarrow^\K_L \ul M\). We now have a diagonal action \(\K/L\) on \(\Z[\K/L]\otimes M_L\). Similarly, \(\K/e\) acts diagonally on \(\Z[\K/L]\otimes M_e\).

\begin{exmp}
For \(\ul M=\ulZ\), $\uparrow_L^K \ulZ$ and $\ulZ \rightarrow \uparrow_L^K \ulZ$ are as follows.
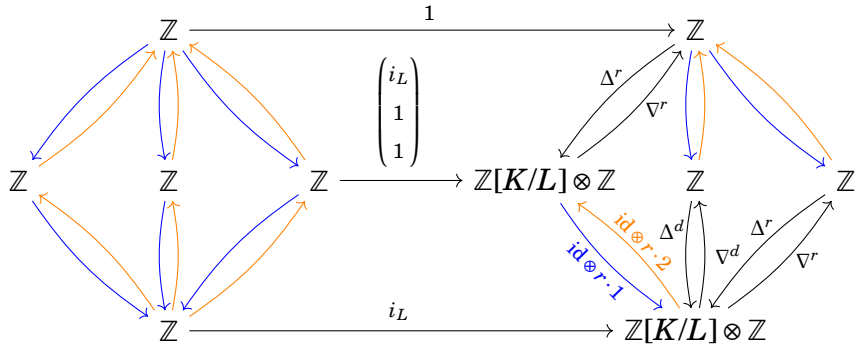
\begin{equation*}
\begin{tikzpicture}
	\node (MK) at (0,4) {$\Z$};
	\node (ML) at (-2,2) {$\Z$};
	\node (MD) at (0,2) {$\Z$};
	\node (MR) at (2,2) {$\Z$};	
	\node (Me) at (0,0) {$\Z$};
	\node (rMK) at (7,4) {$\Z$};
	\node (rML) at (5,2) {$\Z[\K/L] \otimes \Z$};
	\node (rMD) at (7,2) {$\Z$};
	\node (rMR) at (9,2) {$\Z$};
	\node (rMe) at (7,0) {$\Z[\K/L] \otimes \Z$};
	% maps
	\draw[->,blue,bend right=2ex] (MK) to (ML);
	\draw[->,blue,bend right=2ex] (MK) to (MD);
	\draw[->,blue,bend right=2ex] (MK) to (MR);
	\draw[->,orange,bend right=2ex] (ML) to (MK);
	\draw[->,orange,bend right=2ex] (MD) to (MK);
	\draw[->,orange,bend right=2ex] (MR) to (MK);
	\draw[->,blue,bend right=2ex] (ML) to (Me);
	\draw[->,blue,bend right=2ex] (MD) to (Me);
	\draw[->,blue,bend right=2ex] (MR) to (Me);
	\draw[->,orange,bend right=2ex] (Me) to (ML);
	\draw[->,orange,bend right=2ex] (Me) to (MD);
	\draw[->,orange,bend right=2ex] (Me) to (MR);
	\arres{rMK}{rML}{0}{0}{$\Delta^r$}
	\draw[->,blue,bend right=2ex] (rMK) to (rMD);
	\draw[->,blue,bend right=2ex] (rMK) to (rMR);
	\artr{rML}{rMK}{0}{0}{$\nabla^r$}
	\draw[->,orange,bend right=2ex] (rMD) to (rMK);
	\draw[->,orange,bend right=2ex] (rMR) to (rMK);
	\draw[->,blue,bend right=2ex] (rML) to node[sloped, anchor=center, xshift={0pt},yshift={-5pt}] {\scalebox{0.7}{$\text{id}\otimes r\cdot 1$}} (rMe);
	\arres{rMD}{rMe}{-5}{2}{$\Delta^d$}
	\arres{rMR}{rMe}{0}{0}{$\Delta^r$}		
	\draw[->,orange,bend right=2ex] (rMe) to node[sloped, anchor=center, xshift={0pt},yshift={5pt}] {\scalebox{0.7}{$\text{id}\otimes r\cdot 2$}} (rML);
	\artr{rMe}{rMD}{0}{0}{$\nabla^d$}
	\artr{rMe}{rMR}{0}{0}{$\nabla^r$}
	\draw[->] (MK) to node[above, xshift={-1pt}] {\scalebox{0.7}{$1$}} (rMK);
	\draw[->] (MR) to node[above, xshift={-1pt}] {\scalebox{0.7}{$\begin{pmatrix}i_L \\ 1 \\ 1 \end{pmatrix}$}} (rML);
	\draw[->] (Me) to node[above, xshift={-1pt}] {\scalebox{0.7}{$i_L$}} (rMe);
\end{tikzpicture}
\end{equation*}
Again, blue denotes multiplication by one and orange multiplication by two.
\end{exmp}

%%%%%%%%%%%%%%%%%%%%%%%%%%%%%%%%%%%%%
\section{2-slice Characterization \hspace{2in}} \label{sec:2slice}
%%%%%%%%%%%%%%%%%%%%%%%%%%%%%%%%%%%%%

Consider a \(\K\)-spectrum \(X\). Then by \cref{rho:susp:commutes},
\begin{align*}
    P^n_n  X &\simeq \Si{\frac{n-r}{4}\rho} P^r_r \left( \Si{-\frac{n-r}{4}\rho} X \right)
\end{align*}
where \(n\equiv r\pmod 4\) and \(0\leq r\leq 3\). Thus, to know the slices of \(X\), we need only know the 0-, 1-, 2-, and 3-slices of certain suspensions of \(X\). \cref{charac:slice:01-1} and \cref{rho:susp:commutes} characterize the 0-, 1-, and 3-slices. We now characterize the 2-slices.

\begin{prop} \label{Homotopy_Reduction}
Let \(G\) be a finite group, \(H\) an index two subgroup of \(G\), and \(\sigma^H\) the sign representation from \(G\rightarrow G/H\). 
For a \(G\)-spectrum \(X\), if \(\downarrow^G_H \ul \pi_{n+1}(X) = \ul 0 = \downarrow^G_H \ul \pi_n(X)\), the natural map \(\Si{-\sigma^H} X \rightarrow X\) induces an isomorphism on \(\ul\pi_{n}\). In particular, if \(\pi_{n+1}^H(X) = 0 = \pi_n^G(X)\), then \(\pi_n^G(\Si{-\sigma^H} X) = 0\). 
\end{prop}
\begin{proof}
Because \(\sigma^H\) is one-dimensional we may construct \(\Si{-\sigma^H} X\) with the cofiber sequence \((S^{-\sigma^H} \rightarrow S^0 \rightarrow G/H_+) \wedge X\). The resulting long exact sequence in homotopy is
\begin{align} \label{les:reduction}
	\uparrow^G_H \downarrow^G_H \ul\pi_{n+1}(X) \rightarrow \ul\pi_n(\Si{-\sigma^H} X) \rightarrow \ul\pi_n(X) \rightarrow \uparrow^G_H \downarrow^G_H \ul\pi_{n}(X).
\end{align}
As \(\uparrow^G_H \downarrow^G_H \ul\pi_{n+1}(X) = \ul 0 = \uparrow^G_H \downarrow^G_H \ul \pi_n(X)\),
we find that \(\ul\pi_n(\Si{-\sigma^H} X) \cong \ul\pi_n(X)\).

Now suppose that \(\pi_{n+1}^H(X) = 0 = \pi_n^G(X)\). Because \(\pi^H_{n+1}(X)\) is the value of \(\uparrow^G_H \downarrow^G_H \ul\pi_{n+1}(X)\) at the orbit \(G/H\), the left three terms of \cref{les:reduction} prove the exact sequence
\begin{align*}
    0 \rightarrow \pi_n^G(\Si{-\sigma^H} X) \rightarrow 0.
\end{align*}
Consequently, \(\pi_n^G(\Si{-\sigma^H} X) = 0\).
\end{proof}

\begin{cor} \label{EM:homotopy:reduction}
Let \(H\ul M\) be an  Eilenberg-MacLane \(G\)-spectrum  and \(V \cong \mathbf{s} \oplus \bigoplus_{i=1}^r \sigma^{H_i}\) be a real representation with \(s\) copies of the trivial representation and each \(\sigma^{H_i}\) the sign representation from \(G\rightarrow G/H_i\), where \(H_i\) is an index two subgroup of \(G\). Then \(\Si{-V} H\ul M\) does not have nontrivial homotopy above degree \(-s\).
\end{cor}
\begin{proof}
First apply \cref{les:reduction} to \(X_1 := \Si{-\sigma^{H_1}} H\ul M\) with \(n\geq 1\). Then repeat for \mbox{\(X_i:= \Si{-\sigma^{H_i}} X_{i-1}\)} where \(2\leq i\leq r\). Then \(\Si{-V} H\ul M \simeq \Si{-s} X_{r}\) has no homotopy above degree \(-s\).
\end{proof}

Recall from \cref{subsec:signrep} that \(\alpha\), \(\beta\), and \(\gamma\) are the sign representations of \(\K\).

\begin{lem} \label{homotopy:-beta:M}
Let \(\ul M\) be a \(\K\)-Mackey functor. The nontrivial homotopy of \(\Si{-\beta} H\ul M\) is
\begin{center}
\begin{tikzpicture}[scale=0.8, every node/.style={scale=0.75}, bend right=2ex]
	\node (pi0) at (-8,2) {$\ul \pi_0 \quad =$};
	\node (lMK) at (-5,4) {$\ker \res{\K}{L}$};
	\node (lML) at (-7,2) {$0$};
	\node (lMD) at (-5,2) {$\ker \res{D}{e}$};
	\node (lMR) at (-3,2) {$\ker \res{R}{e}$};
	\node (lMe) at (-5,0) {$0$};
	%
	%\node (and) at (0,2) {and};
	%
	\begin{scope}[xshift=-3cm]
	\node (pi-1) at (3,2) {$\ul \pi_{-1} \quad =$};
	\node (MK) at (7,4) {$M_L / \im \res{\K}{L}$};
	\node (ML) at (5,2) {$\Z_\sigma^L\otimes M_L$};
	\node (MD) at (7,2) {$M_e / \im \res{D}{e}$};
	\node (MR) at (9,2) {$M_e / \im \res{R}{e}$};
	\node (Me) at (7,0) {$\Z_\sigma^L\otimes M_e$};
	\end{scope}
	% left arrows
	\draw[->,\resC] (lMK) to (lMR);
	\draw[->,\trC] (lMR) to (lMK);
	\draw[->,\resC] (lMK) to (lMD);
	\draw[->,\trC] (lMD) to (lMK);
	% right arrows
	\arres{MK}{ML}{-2}{0}{$\varphi_r^{\K L}$}
	\draw[->,\resC] (MK) to (MD);
	\draw[->,\resC] (MK) to (MR);
	\artr{ML}{MK}{0}{0}{$\pi$}
	\draw[->,\trC] (MD) to (MK);
	\draw[->,\trC] (MR) to (MK);
	\draw[->,\resC] (ML) to node[sloped,anchor=center,xshift={0pt},yshift={-5pt}] {\scalebox{0.7}{$1\otimes r\cdot \res{L}{e}$}} (Me);
	\draw[->,\trC] (Me) to node[sloped,anchor=center,xshift={0pt},yshift={5pt}] {\scalebox{0.7}{$1\otimes r\cdot \tr{L}{e}$}} (ML);
	\arres{MD}{Me}{-6}{0}{$\varphi_r^{De}$}
	\arres{MR}{Me}{1}{2}{$\varphi_r^{Re}$}
	\artr{Me}{MD}{0}{0}{$\pi$}
	\artr{Me}{MR}{0}{0}{$\pi$}
\end{tikzpicture}	
\end{center}
Here, \(\varphi_r^{\K L}\) is induced by \(\Delta^r\) in the square
\begin{equation*}
\begin{tikzcd}
    M_L \arrow[r, "\pi"] \arrow[d, "\Delta^r"] & M_L/ \im \res{\K}{L} \arrow[d, "\varphi_r^{\K L}"] \\
    \Z[\K/L] \otimes M_L \arrow[r, "q_L"] & \Z_\sigma^L\otimes M_L
\end{tikzcd}	    
\end{equation*}
and is given by \(\varphi_r^{KL}(m) = m-r\cdot m\). The maps \(\varphi_r^{De}:M_e/\im \res{D}{e}\rightarrow \Z_\sigma^L\otimes M_e\) and \(\varphi_r^{Re}:M_e/\im \res{R}{e}\rightarrow \Z_\sigma^L\otimes M_e\) are defined similarly.
\end{lem}
\begin{proof}
We have the cofiber sequence \((S^{-\beta} \rightarrow S^0 \rightarrow \K/L_+)\wedge H\ul M\). The result then follows from the description of the map \(\ul M \rightarrow \uparrow^\K_L\downarrow^\K_L \ul M\) given in \cref{Induction_Formula}.
\end{proof}

\begin{prop} \label{homotopy:-rho:M}
Let \(\ul M\) be a \(\K\)-Mackey functor. The nontrivial homotopy of \(\Si{-\alpha-\beta-\gamma} H\ul M\) is

\begin{tikzpicture}[scale=0.8, every node/.style={scale=0.75}]
    \node (pi0) at (-9,2) {$\ul \pi_0 \quad =$};
    \node (lMK) at (-5,4) {$\ker (\res{\K}{L} + \res{\K}{D} + \res{\K}{R})$};
    \node (lML) at (-7,2) {$0$};
    \node (lMD) at (-5,2) {$0$};
    \node (lMR) at (-3,2) {$0$};
    \node (lMe) at (-5,0) {$0$};
    %
    %\node (and) at (0,2) {,};
    %
    %
    \begin{scope}[xshift=-3cm]
    \node (pi-1) at (3,2) {$\ul \pi_{-1} \quad =$};
    \node (MK) at (7,4) {$E_2$};
    \node (ML) at (5,2) {\scalebox{0.9}{$\begin{aligned}
    		\Z_\sigma^L &\otimes \\ \ker & \res{L}{e}
    	\end{aligned}$}};
    \node (MD) at (7,2) {\scalebox{0.9}{$\begin{aligned}
    		\Z_\sigma^D &\otimes \\ \ker & \res{D}{e}
    	\end{aligned}$}};
    \node (MR) at (9,2) {\scalebox{0.9}{$\begin{aligned}
    		\Z_\sigma^R &\otimes \\ \ker & \res{R}{e}
    	\end{aligned}$}};
    \node (Me) at (7,0) {$0$};	
    \end{scope}
    \arres{MK}{ML}{0}{0}{}
    \arres{MK}{MD}{0}{0}{}
    \arres{MK}{MR}{0}{0}{}
    \artr{ML}{MK}{0}{0}{}
    \artr{MD}{MK}{0}{0}{}
    \artr{MR}{MK}{0}{0}{}
\end{tikzpicture}

\begin{tikzpicture}[scale=0.8, every node/.style={scale=0.75}]
    \node (pi-2) at (-9,2) {$\ul \pi_{-2} \quad =$};
    \node (lMK) at (-5,4) {$E_3$};
    \node (lML) at (-7,2) {\scalebox{0.9}{$\begin{aligned}
    		\Z_\sigma^L \otimes& \\  \ker\varphi_r^{Le}
    	\end{aligned}$}};
    \node (lMD) at (-5,2) {\scalebox{0.9}{$\begin{aligned}
    		\Z_\sigma^D \otimes& \\ \ker \varphi_r^{De}
    	\end{aligned}$}};
    \node (lMR) at (-3,2) {\scalebox{0.9}{$\begin{aligned}
    		\Z_\sigma^R \otimes& \\ \ker\varphi_r^{Re}
    	\end{aligned}$}};
    \node (lMe) at (-5,0) {$0$};
    %
    % left maps
    \arres{lMK}{lML}{0}{0}{}
    \arres{lMK}{lMD}{0}{0}{}
    \arres{lMK}{lMR}{0}{0}{}
    \artr{lML}{lMK}{0}{0}{}
    \artr{lMD}{lMK}{0}{0}{}
    \artr{lMR}{lMK}{0}{0}{}
    %
    %\node (and) at (0,2) {,};
    %
    \begin{scope}[xshift=-3cm]
    \node (pi-3) at (3,2) {$\ul \pi_{-3} \quad =$};
    \node (MK) at (7,4) {$(M_e / \im \varphi_r^{De}) / \im \varphi_r^{R,L}$};
    \node (ML) at (5,2) {$M_e / \im \varphi_d^{Le}$};
    \node (MD) at (7,2) {$M_e / \im \varphi_r^{De}$};
    \node (MR) at (9,2) {$M_e / \im \varphi_r^{Re}$};
    \node (Me) at (7,0) {$M_e$};	
    \end{scope}
    
    %
    % right maps
    \arres{MK}{ML}{0}{0}{$\varphi_d$}
    \arres{MK}{MD}{-5}{-10}{$\varphi_l$}
    \arres{MK}{MR}{0}{-14}{$\varphi_d$}
    \artr{ML}{MK}{0}{0}{$\pi$}
    \artr{MD}{MK}{0}{0}{$\pi$}
    \artr{MR}{MK}{0}{0}{$\pi$}
    \arres{ML}{Me}{-15}{-5}{$\varphi_{l}^{dLe}$}
    \arres{MD}{Me}{-10}{5}{$\varphi_{d}^{rDe}$}
    \arres{MR}{Me}{0}{0}{$\varphi_{l}^{rRe}$}
    \artr{Me}{ML}{0}{0}{$\pi$}
    \artr{Me}{MD}{0}{0}{$\pi$}
    \artr{Me}{MR}{0}{0}{$\pi$}
\end{tikzpicture}

where \(E_1\), \(E_2\), and \(E_3\) are extensions
\begin{equation*}
\begin{tikzcd}[row sep=small]
    M_e/\im \varphi_r^{De} \arrow[r] & E_3 \arrow[r] & (M_e/\im \res{R}{e})/\im \res{L}{e} \\
    \ker\res{D}{e}  \arrow[r] & E_2 \arrow[r] & E_1 \\
    \ker\res{R}{e} \arrow[r] & E_1 \arrow[r] & M_L/\im \res{\K}{L}
\end{tikzcd}
\end{equation*}
Let \(\varphi_h^{*}:\overline{M_e} \rightarrow \overline{M_e}'\) be one of the maps shown. Then \(\varphi_h^{*}(m) = m-h\cdot m\). Additionally, \((M_e/\im \varphi_r^{De})/\im \varphi_r^{R,L}\) is the cokernel of the map
\begin{equation*}
\begin{tikzcd}
    (M_e/\im \res{R}{e})/\im \res{L}{e} \arrow[r, "\varphi_r^{R,L}"] & M_e/\im \varphi_r^{De}.
\end{tikzcd}
\end{equation*}
\end{prop}
\begin{proof}
\cref{homotopy:-beta:M} supplies the homotopy for \(\Si{-\beta} H\ul M\). We continue constructing \(\Si{-\alpha-\beta-\gamma} H\ul M\) iteratively. The cofiber sequence
\begin{align*}
    (S^{-\alpha} \rightarrow S^0 \rightarrow \K/R_+) \wedge \Si{-\beta} H\ul M
\end{align*}
results in the homotopy of \(\Si{-\alpha-\beta} H\ul M\)

\begin{tikzpicture}[scale=0.8, every node/.style={scale=0.75}]
    \node (pi0) at (-9,2) {$\ul \pi_0 \quad =$};
    \node (lMK) at (-5,4) {$\ker (\res{\K}{L} + \res{\K}{R})$};
    \node (lML) at (-7,2) {$0$};
    \node (lMD) at (-5,2) {$\ker \res{\K}{D}$};
    \node (lMR) at (-3,2) {$0$};
    \node (lMe) at (-5,0) {$0$};
    \arres{lMK}{lMD}{0}{0}{}
    \artr{lMD}{lMK}{0}{0}{}
    %
    %
    %\node (and) at (0,2) {,};
    \begin{scope}[xshift=-3cm]
    \node (pi-1) at (3,2) {$\ul \pi_{-1} \quad =$};
    \node (MK) at (7,4) {$E_1$};
    \node (ML) at (5,2) {\scalebox{0.9}{$\begin{aligned}
    			\Z_\sigma^L &\otimes \\ \ker & \res{L}{e}
    		\end{aligned}$}};
    \node (MD) at (7,2) {\scalebox{0.9}{$\ker\varphi_r^{De}$}};
    \node (MR) at (9,2) {\scalebox{0.9}{$\begin{aligned}
    			\Z_\sigma^R &\otimes \\ \ker & \res{R}{e}
    		\end{aligned}$}};
    \node (Me) at (7,0) {$0$};	
    \end{scope}
    \arres{MK}{ML}{0}{0}{}
    \arres{MK}{MD}{0}{0}{}
    \arres{MK}{MR}{0}{0}{}
    \artr{ML}{MK}{0}{0}{}
    \artr{MD}{MK}{0}{0}{}
    \artr{MR}{MK}{0}{0}{}
\end{tikzpicture}

\begin{tikzpicture}[scale=0.8, every node/.style={scale=0.75}]
    \node (pi-2) at (-9,2) {$\ul \pi_{-2} \quad =$};
    \node (lMK) at (-5,4) {$(M_e/\im \res{D}{e})/ \im \res{L}{e}$};
    \node (lML) at (-7,2) {\scalebox{0.9}{$\begin{aligned}
    			\Z_\sigma^L &\otimes \\  M_e / &\im \res{L}{e}
    		\end{aligned}$}};
    \node (lMD) at (-5,2) {\scalebox{0.9}{$\begin{aligned}
    			\Z_\sigma^D &\otimes \\ M_e / &\im \varphi_r^{De}
    		\end{aligned}$}};
    \node (lMR) at (-3,2) {\scalebox{0.9}{$\begin{aligned}
    			\Z_\sigma^R &\otimes \\ M_e / &\im \res{R}{e}
    		\end{aligned}$}};
    \node (lMe) at (-5,0) {$\Z_\sigma^D \otimes M_e$};
    %
    % left maps
    \arres{lMK}{lML}{0}{0}{$\varphi_r^{D,L}$}
    \arres{lMK}{lMD}{-8}{-17}{$\varphi_r^{D,L}$}
    \arres{lMK}{lMR}{0}{-23}{$\varphi_d^{D,L}$}
    \artr{lML}{lMK}{-11}{0}{$\pi$}
    \artr{lMD}{lMK}{0}{0}{$\pi$}
    \artr{lMR}{lMK}{0}{0}{$\pi$}
    \arres{lML}{lMe}{-10}{-8}{$\varphi_d^{Le}$}
    \arres{lMD}{lMe}{-6}{3}{$\varphi_d^{rDe}$}
    \draw[->,\resC,bend right=2ex] (lMR) to node[sloped, anchor=center, xshift={0pt},yshift={-5pt}] {\scalebox{0.7}{$1\otimes \varphi_r^{Re}$}} (lMe);
    \artr{lMe}{lML}{0}{0}{$\pi$}
    \artr{lMe}{lMD}{-2}{0}{$\pi$}		
    \draw[->,\trC,bend right=2ex] (lMe) to node[sloped, anchor=center, xshift={0pt},yshift={-5pt}] {\scalebox{0.7}{$1\otimes \pi$}} (lMR);	
\end{tikzpicture}	

Finally, the cofiber sequence \((S^{-\gamma} \rightarrow S^0 \rightarrow \K/D_+) \wedge \Si{-\alpha-\beta} H\ul M\) provides the result.
\end{proof}

\begin{lem} \label{Kernel_Equivalences}
Let \(\phi:G\rightarrow N\) and \(\psi:G\rightarrow M\) be group homomorphisms. Then 
\begin{align*}
	\ker(\phi, \psi) = \ker(\phi) \cap \ker(\psi) = \ker( \phi\vert_{\ker(\psi)})
\end{align*}
where \((\phi,\psi):G\rightarrow N\oplus M\) is defined by \(g\mapsto \phi(g)\oplus \psi(g)\).
\end{lem}

\begin{prop} \label{2slice:equivalences:si2}
Let \(\ul M\) be a \(\K\) Mackey functor where the restrictions \(\res{L}{e}\), \(\res{D}{e}\), and \(\res{R}{e}\) are injective. The following are equivalent.
\begin{enumerate}[(A)]
    \item The sequence
    \begin{align*}
        M_\K \xrightarrow{\res{\K}{L}+\res{\K}{D}+\res{\K}{R}} M_L \oplus M_D \oplus M_R \xrightarrow{\left(\begin{smallmatrix} -\res{L}{e} & \res{L}{e} & 0 \\ \res{D}{e} & 0 & -\res{D}{e} \\ 0 & -\res{R}{e} & \res{R}{e} \end{smallmatrix}\right)}  M_e^3
    \end{align*} 
    is exact. \label{2slice:si2:exact1}
    
    \item \(\im \res{\K}{H_1} = (\res{H_1}{e})^{-1}(\im \res{H_2}{e}) \cap (\res{H_1}{e})^{-1}(\im \res{H_3}{e})\) where \(H_1\), \(H_2\), \(H_3\) are distinct order two subgroups of \(\K\). \label{2slice:res:injective}
    This equality is represented by the diagram below.
    \begin{equation*}
    \begin{tikzcd}
        \im \res{K}{H_1} \arrow[d, hook] \ar[r] &  \im \res{H_2}{e} \cap \im \res{H_3}{e} \ar[d, hook] \\
        M_{H_1} \ar[r, "\res{H_1}{e}"] & M_e 
    \end{tikzcd}
    \end{equation*}
\end{enumerate}
\end{prop}
\begin{proof}
Without loss of generality, let \(H_1=L\), \(H_2=D\), and \(H_3=R\). For convenience, set \(I = (\res{L}{e})^{-1}(\im \res{D}{e}) \cap (\res{L}{e})^{-1}(\im \res{R}{e})\).

\cref{2slice:si2:exact1} \(\Rightarrow\) \cref{2slice:res:injective}:

Let \(x\in \im \res{\K}{L}\). Then we have \(k\in M_\K\) such that \(\res{\K}{L}(k) = x\). Additionally, \(\res{L}{e}(x) = \res{D}{e} \res{\K}{D} (k) = \res{R}{e} \res{\K}{R} (k)\). It follows that \(x\in I\).

Now suppose \(x\in I\). We then have \(y\in D\) and \(z\in R\) such that
\begin{align*}
    \res{L}{e}(x) = \res{D}{e}(y) = \res{R}{e}(z).
\end{align*}
Thus, \((x,y,z) \in \ker \left(\begin{smallmatrix} -\res{L}{e} & \res{L}{e} & 0 \\ \res{D}{e} & 0 & -\res{D}{e} \\ 0 & -\res{R}{e} & \res{R}{e} \end{smallmatrix}\right)\). Consequently, we have \(x = \res{\K}{L}(k)\) for some \(k\in M_\K\); that is, \(x\in \im \res{\K}{L}\).

\cref{2slice:res:injective} \(\Rightarrow\) \cref{2slice:si2:exact1}:

Let \((x,y,z)\in \ker \left(\begin{smallmatrix} -\res{L}{e} & \res{L}{e} & 0 \\ \res{D}{e} & 0 & -\res{D}{e} \\ 0 & -\res{R}{e} & \res{R}{e} \end{smallmatrix}\right)\). Then \(\res{L}{e}(x) = \res{D}{e}(y) = \res{R}{e}(z)\), so \(x\in \im \res{\K}{L}\). So there is some \(k\in M_\K\) such that \(x = \res{\K}{L}(k)\). Because \(\res{D}{e}\) and \(\res{R}{e}\) are injective, it must be that \(y = (\res{D}{e})^{-1}(\res{L}{e}(x))\) and \(z = (\res{R}{e})^{-1}(\res{L}{e}(x))\). Hence, \cref{2slice:si2:exact1} is exact.
\end{proof}

\begin{prop} \label{2slice:characterization:si1}
The \(\K\)-spectrum \(\Si{1} H\ul N\) is a 2-slice if and only if \(N_e = 0\) and \(N_\K \rightarrow N_L \oplus N_D \oplus N_R\) is injective.
\end{prop}
\begin{proof}
By \cref{thm:atleastn}, we find that \(\Si{1} H\ul N\geq 2\) if and only if \(N_e = 0\). Consequently, going forward we may assume \(N_e = 0\).

Now \(\Si{1} H\ul N\leq 2\) if and only if \([S^{k\rho_H+r}, \Si{2} H\ul N]^H=0\) for all \(k\geq \frac{3}{\abs{H}}\) and \(r\geq 0\). 
Because \(N_e = 0\), \cref{C2:2slice} implies that \(i_H^*(\Si{1} H\ul N)\) is a 2-slice, where \(H\) is an order two subgroup of \(\K\). Thus, to finish the equivalence, we only need consider 
\begin{align*}
    [S^{k\rho_\K+r}, \Si{1} H\ul N]^\K = [S^{k+r-1}, \Si{-k\overline{\rho}_\K} H\ul N]^\K
\end{align*}
for all \(k\geq 1\) and \(r\geq 0\).

From \cref{EM:homotopy:reduction}, we need only concern ourselves with \([S^0, \Si{-k\overline{\rho}_\K} H\ul N]\). From \cref{homotopy:-rho:M},
\begin{align*}
    \ul \pi_0(\Si{-\overline\rho} H\ul N) &= \phi^*_\K( \ker \res{\K}{L} \cap \ker \res{\K}{D} \cap \ker \res{\K}{R}).
\end{align*}
Therefore \cref{Homotopy_Reduction} shows that \(\ul \pi_0(\Si{-\overline\rho} H \ul N)\) vanishes if and only if \(\ul \pi_0(\Si{-k\overline\rho} H \ul N)\) vanishes for all \(k\geq 1\). Hence, \(\Si{1} H\ul N\leq 2\) if and only if \(\ker \res{\K}{L} \cap \ker \res{\K}{D} \cap \ker \res{\K}{R} = \{0\}\). By \cref{Kernel_Equivalences} this is equivalent to \(N_\K \rightarrow N_L \oplus N_D \oplus N_R\) being injective.
\end{proof}

\begin{prop} \label{2slice:characterization:si2}
The spectrum \(\Si{2} H\ul M\) is a 2-slice if and only if all restrictions of \(\ul M\) are injective and the sequence
\begin{align*}
    M_\K \xrightarrow{\res{\K}{L}+\res{\K}{D}+\res{\K}{R}} M_L \oplus M_D \oplus M_R \xrightarrow{\left(\begin{smallmatrix} -\res{L}{e} & \res{L}{e} & 0 \\ \res{D}{e} & 0 & -\res{D}{e} \\ 0 & -\res{R}{e} & \res{R}{e} \end{smallmatrix}\right)}  M_e^3
\end{align*}
is exact.
\end{prop}
\begin{proof}
Note that \(\Si{2} H\ul M \geq 2\), so we just need to show that \(\Si{2} H\ul M\leq 2\). Now \(\Si{2} H\ul M\leq 2\) if and only if \([S^{k\rho_H+r}, \Si{2} H\ul M]^H=0\) for all \(k\geq \frac{3}{\abs{H}}\) and \(r\geq 0\). For \(H\) an order two subgroup of \(\K\), \(i_H^* \Si{2} H\ul M\leq 2\) if and only if \(\res{H}{e}\)
is injective. Consequently, going forward we may assume \(\res{L}{e}\), \(\res{D}{e}\), and \(\res{R}{e}\) are injective.

To finish the equivalence, we only need consider 
\begin{align*}
    [S^{k\rho_\K+r}, \Si{2} H\ul M]^\K = [S^{k+r-2}, \Si{-k\overline{\rho}_\K} H\ul M]^\K
\end{align*}
for all \(k\geq 1\) and \(r\geq 0\). By \cref{EM:homotopy:reduction}, it is enough to examine \([S^0, \Si{-k\overline{\rho}_\K} H\ul M]\) and \([S^{-1}, \Si{-k\overline{\rho}_\K} H\ul M]\). From \cref{homotopy:-rho:M}, 
\begin{align*}
	\ul \pi_0(\Si{-\overline\rho} H\ul M) &= \phi^*_\K( \ker \res{\K}{L}) 
\end{align*}
and 
\begin{align*}
    \ul \pi_{-1}(\Si{-\overline\rho} H\ul M) &= \phi^*_\K \left( \frac{(\res{L}{e})^{-1}(\im\res{D}{e}) \cap (\res{L}{e})^{-1}(\im \res{R}{e})}{\im \res{\K}{L}} \right).
\end{align*}
Note that \cref{homotopy:-rho:M} states
\begin{align*}
    \ul \pi_0(\Si{-\overline\rho} H\ul M) = \phi^*_\K( \ker \res{\K}{L} \cap \ker\res{\K}{D} \cap\ker\res{\K}{R}),
\end{align*}
but because the lower restrictions are injective, these kernels coincide.
\cref{Homotopy_Reduction} then yields that
\begin{align*}
	\ul\pi_0^G(\Si{-k\overline\rho} H\ul M) &\cong \ul\pi_0^G(\Si{-\overline\rho} H\ul M)  = \ker \res{\K}{L} \\
	\ul\pi_{-1}^G(\Si{-k\overline\rho} H\ul M) &\cong \ul\pi_{-1}^G(\Si{-\overline\rho} H\ul M) = \frac{(\res{L}{e})^{-1}(\im\res{D}{e}) \cap (\res{L}{e})^{-1}(\im \res{R}{e})}{\im \res{\K}{L}}
\end{align*}
for all \(k\geq 2\). By \cref{def:lessthan}, we find that \(\Si{2} H\ul M\leq 2\) if and only if these two homotopy groups vanish. Hence, \(\Si{2} H\ul M\leq 2\) if and only if \(\ker \res{\K}{L}=\{0\}\) and \(\im \res{\K}{L} = (\res{L}{e})^{-1}(\im\res{D}{e}) \cap (\res{L}{e})^{-1}(\im \res{R}{e})\). Because the homotopy of \(\Si{-\alpha-\beta-\gamma} H\ul M\) is invariant under the order of construction -- that is, whether \(H\ul M\) is suspended by say \(-\alpha\) or \(-\beta\) first -- we find that all upper restrictions must be injective and that \cref{2slice:res:injective} must hold.
\end{proof}

\begin{thm} \label{2slice:characterization}
Suppose the only nontrivial homotopy Mackey functors of a \(\K\)-spectrum \(X\) are \(\ul\pi_1(X)\) and \(\ul\pi_2(X)\) where
\begin{enumerate}
	\item All restrictions of \(\ul\pi_2(X)\) are injective and the sequence
	\begin{align*}
	\pi_2^\K(X) \xrightarrow{\res{\K}{L}+\res{\K}{D}+\res{\K}{R}} \pi_2^L(X) \oplus \pi_2^D(X) \oplus \pi_2^R(X) \xrightarrow{\left(\begin{smallmatrix} -\res{L}{e} & \res{L}{e} & 0 \\ \res{D}{e} & 0 & -\res{D}{e} \\ 0 & -\res{R}{e} & \res{R}{e} \end{smallmatrix}\right)}  \pi_2^e(X)^3
	\end{align*}
	is exact. \label{2slice:thm:M}
	
	\item \(\pi_1^e(X) = 0\). \label{2slice:thm:N0}
	
	\item \(\pi_1^\K(X) \rightarrow \pi_1^L(X) \oplus \pi_1^D(X) \oplus \pi_1^R(X)\) is injective. \label{2slice:thm:N0:2}
\end{enumerate}
Then \(X\) is a 2-slice. 

Conversely, if a \(\K\)-spectrum \(X\) is a 2-slice, then its only nontrivial homotopy Mackey functors of a \(\K\)-spectrum \(X\) are \(\ul\pi_1(X)\) and \(\ul\pi_2(X)\) where \(\Si{2} H\ul\pi_2(X)\) is a 2-slice and \(\Si{1} H\ul\pi_1(X) \in [2,4]\), i.e., both \cref{2slice:thm:M} and \cref{2slice:thm:N0} hold.
\end{thm}
\begin{proof}
Let \(\ul\pi_2(X)=\ul M\) and \(\ul\pi_1(X)=\ul N\). If these are the only nontrivial homotopy Mackey functors of \(X\), we have a fiber sequence
\begin{align} \label{2slice:fibseq}
    \Si{2} H\ul M\rightarrow X\rightarrow \Si{1} H\ul N.
\end{align}
By \cref{2slice:characterization:si1} and \cref{2slice:characterization:si2}, conditions \eqref{2slice:thm:M} - \eqref{2slice:thm:N0:2} show that \(\Si{2} H\ul M\) and \(\Si{1} H\ul N\) are 2-slices. Now if \(\Si{2} H\ul M\) and \(\Si{1} H\ul N\) are 2-slices, then \(X\) must be a 2-slice as well. 

Conversely, suppose that \(X\) is a 2-slice. We then find that \(X\) has no homotopy above degree two and none below degree one; thus, we have the fiber sequence in \cref{2slice:fibseq}.
Because \(X\geq 2\) and \(\Si{2} H\ul M\geq 2\), it follows that \(\Si{1} H\ul N\geq 2\). So by \cref{thm:atleastn}, \(N_e = 0\). That \(\Si{1} H\ul N\leq 4\) follows from a similar argument as in \cref{2slice:characterization:si1}.
 
Rotating this fiber sequence gives \(H\ul N\rightarrow \Si{2} H\ul M \rightarrow X\). As \(H\ul N\) is a 0-slice and \(X\) is a 2-slice, we have \(\Si{2} H\ul M \in [0,2]\), so it must be that \(\Si{2} H\ul M = 2\). Consequently, by \cref{2slice:characterization:si2}, \cref{2slice:thm:M} holds.
\end{proof}

It is not necessary for condition \eqref{2slice:thm:N0:2} to hold for \(X\) to be a 2-slice as we show in the following example.

\begin{exmp}
Take \(X \simeq \Si{1+\beta} H\ul E\) where
\begin{equation*}
	\ul E = \begin{tikzcd}[ampersand replacement=\&, column sep=scriptsize]
		\& \F \ar[d, color=\resC] \ar[dr,color=\resC] \& \\
		0 \& \F \& \F \\
		\& \F \ar[u, color=\trC] \ar[ur, color=\trC] \&
	\end{tikzcd}
\end{equation*}
Then \(X\) is a 2-slice with \(\ul\pi_2(X) = \ul f\) and \(\ul\pi_1(X) = \ulg\). But \(\Si{1} H\ulg\) is not a 2-slice.
\end{exmp}
\begin{proof}
We can construct \(\Si{1+\beta} H\ul E\) using the cofiber sequence 
\begin{align*}
    (\K/L_+ \rightarrow S^0 \rightarrow S^\beta) \wedge \Si{1} H\ul E.
\end{align*}
The resulting long exact sequence in homotopy is

\begin{flushright}
% pi_2 row
\begin{tikzpicture}[scale=0.8, every node/.style={scale=0.75},xscale=0.9]
    \node (pi2) at (-9,1.5) {$\ul \pi_2$};
    \node (MD) at (0,1.5) {$\ul 0$};
    \node (rMK) at (5.5,3) {$0$};
    \node (rML) at (4,1.5) {$0$};
    \node (rMD) at (5.5,1.5) {$0$};
    \node (rMR) at (7,1.5) {$0$};
    \node (rMe) at (5.5,0) {$\F$};
    \draw[->] (MD) to (rML);
\end{tikzpicture}	

% pi_1 row	
\begin{tikzpicture}[scale=0.8, every node/.style={scale=0.75},xscale=0.9]
    \node (pi1) at (-9,1.5) {$\ul \pi_1$};
    \node (lMK) at (-5.5,3) {$0$};
    \node (lML) at (-7,1.5) {$0$};
    \node (lMD) at (-5.5,1.5) {$\F$};
    \node (lMR) at (-4,1.5) {$\F$};
    \node (lMe) at (-5.5,0) {$\F[\K/L]$};
    \node (MK) at (0,3) {$\F$};
    \node (ML) at (-1.5,1.5) {$0$};
    \node (MD) at (0,1.5) {$\F$};
    \node (MR) at (1.5,1.5) {$\F$};	
    \node (Me) at (0,0) {$\F$};
    \node (rMK) at (5.5,3) {$\F$};
    \node (rML) at (4,1.5) {$0$};
    \node (rMD) at (5.5,1.5) {$0$};
    \node (rMR) at (7,1.5) {$0$};
    \node (rMe) at (5.5,0) {$0$};
    %
    % left maps
    \arres{lMD}{lMe}{-4}{0}{$\Delta$}
    \arres{lMR}{lMe}{0}{0}{$\Delta$}
    \artr{lMe}{lMD}{0}{0}{$\nabla$}
    \artr{lMe}{lMR}{0}{0}{$\nabla$}
    \draw[->, \resC] (MK) to node[right,xshift={2pt}] {\scalebox{0.6}{$1$}} (MR);
    \draw[->, \resC] (MK) to node[right,xshift={2pt}] {\scalebox{0.6}{$1$}} (MD);
    \draw[->, \trC] (Me) to node[right,xshift={2pt}] {\scalebox{0.6}{$1$}}  (MR);
    \draw[->, \trC] (Me) to node[right,xshift={-2pt}] {\scalebox{0.6}{$1$}}  (MD);
    \draw[->] (lMK) to node[above, xshift={-1pt}] {} (MK);
    \draw[->] (lMR) to node[above, xshift={-1pt}] {\scalebox{0.7}{$\begin{pmatrix}0 \\ 1 \\ 1 \end{pmatrix}$}} (ML);
    \draw[->] (lMe) to node[above, xshift={-1pt}] {\scalebox{0.7}{$\nabla$}} (Me);
    \draw[->>] (MK) to (rMK);
    \draw[->>] (MR) to (rML);
    \draw[->>] (Me) to (rMe);
\end{tikzpicture}
\end{flushright}

So \(\ul \pi_2(X) = \ul f\) and \(\ul \pi_1(X) = \ulg\). From this we see that \(X\geq 2\) by \cref{2slice:characterization}. Note that \(\Si{2} H\ul f\) is a 2-slice and \(\Si{1} H\ulg\) is a 4-slice.

To show that \(X\leq 2\), we need \([S^{k\rho_H+r}, \Si{1+\beta} H\ul E]^H=0\) for all \(k\geq \frac{3}{\abs{H}}\) and \(r\geq 0\). As \(i_H^*(X) \simeq \Si{2} H_{C_2} \ul f\) is a 2-slice, where \(H\) is an order two subgroup of \(\K\), we only need to consider 
\begin{align*}
    [S^{k\rho_\K+r}, \Si{1+\beta} H\ul E]^\K = [S^{k+r-1}, \Si{-\alpha-\gamma-(k-1) \overline{\rho}_\K} H\ul E]^\K
\end{align*}
for all \(k\geq 1\) and \(r\geq 0\).

By \cref{EM:homotopy:reduction}, it is sufficient to examine
\([S^0, \Si{-\alpha-\gamma-(k-1)\overline{\rho}_\K} H\ul E]^\K\)
for all \(k\geq 1\).

From \cref{homotopy:-beta:M} we find that \([S^0, \Si{-\alpha-\gamma} H\ul E] = \ul 0\). Consequently, given any \(k\geq 1\), repeated application of \cref{Homotopy_Reduction} shows that 
\begin{align*}
    [S^0, \Si{-\alpha-\gamma-(k-1)\overline{\rho}_\K} H\ul E] = \ul 0.
\end{align*}
That is, \(X\) is a 2-slice.
\end{proof}

%%%%%%%%%%%%%%%%%%%%%%%%%%%%%%%%%%%%%
\section{Cotowers for \texorpdfstring{\(\Si{-n} H\ulZ\)}{HZ}} \label{sec:cotowers:K:Z}
%%%%%%%%%%%%%%%%%%%%%%%%%%%%%%%%%%%%%

We determine the slice towers of \(\Si{-n} H\ulZ\) and \(\Si{-n} H\ul m^*\) for \(1\leq n\leq 5\).

\begin{exmp} \label{slices:K:m*:1-2}
By \cref{charac:slice:01-1} and \cref{m:mg:equiv}, \(\Si{-1} H\ul m^*\) is a \((-2)\)-slice and \(\Si{-2} H\ul m^*\) is a \((-4)\)-slice.

Alternatively, by \cref{m:mg:equiv}, \(\Si{\rho}\Si{-1} H\ul m^* \simeq \Si{1} H\ul {mg}\), which is a 2-slice by \cref{2slice:characterization:si1}. Consequently, \(\Si{-1} H\ul m^*\) is a \((-2)\)-slice.
\end{exmp}

\begin{exmp} \label{tower:-1:K:Z}
By \cite[Theorem 6-4]{U}, the cotower for \(\Si{-1} H\ulZ\) is
\begin{equation*}
\begin{tikzcd}
    \Si{-1} H\ulZ^* \arrow[d] = P^{-1}_{-1} & \\
    \Si{-1} H\ulZ(2,1) \arrow[d] \arrow[r] & \Si{-1} H\ul m^* = P^{-2}_{-2} \\
    \Si{-1} H\ulZ \arrow[r] & \Si{-1} H\ulg = P^{-4}_{-4}
\end{tikzcd}
\end{equation*}
\end{exmp}

\begin{exmp} \label{tower:-2:K:Z}
We suspend the cotower in \cref{tower:-1:K:Z} by \(-1\) to get the cotower for \(\Si{-2} H\ulZ\).
\begin{equation*}
\begin{tikzcd}
    \Si{-2} H\ulZ^* \arrow[d] = P^{-2}_{-2} & \\
    \Si{-2} H\ulZ(2,1) \arrow[d] \arrow[r] & \Si{-2} H\ul m^* = P^{-4}_{-4} \\
    \Si{-2} H\ulZ \arrow[r] & \Si{-2} H\ulg = P^{-8}_{-8}
\end{tikzcd}
\end{equation*}
\end{exmp}

\begin{exmp} \label{tower:-3:K:Z}
We suspend the cotower in \cref{tower:-2:K:Z} by \(-1\) to get a partial cotower for \(\Si{-3} H\ulZ\).
\begin{equation*}
\begin{tikzcd}
	\Si{-3} H\ulZ^* \arrow[d] = P^{-3}_{-3} & \\
	\Si{-3} H\ulZ(2,1) \arrow[d] \arrow[r] & \Si{-3} H\ul m^* \\
	\Si{-3} H\ulZ \arrow[r] & \Si{-3} H\ulg = P^{-12}_{-12}
\end{tikzcd}
\end{equation*}

The issue here is that \(\Si{-3} H\ul m^*\) is not a slice. 

However, the short exact sequence
\begin{align*}
    \ulg^2 \rightarrow \phi^*_{LDR} (\ulF^*) \rightarrow \ul m^*
\end{align*}
provides the tower
\begin{align*}
    P^{-6}_{-6} = \Si{-3} H\phi^*_{LDR} (\ulF^*) \rightarrow \Si{-3} H\ul m^* \rightarrow \Si{-2} H\ulg^2 = P^{-8}_{-8}.
\end{align*}
Consequently, the remaining slices of \(\Si{-3} H\ulZ\) are \(P^{-6}_{-6} = \Si{-2p+1} H\phi^*_{LDR}( \ulF)\) and \(P^{-8}_{-8} = \Si{-2} H\ulg^2\).
\end{exmp}

\begin{exmp} \label{tower:-4:K:Z}
We suspend the cotower in \cref{tower:-3:K:Z} by \(-1\) to get a partial cotower for \(\Si{-4} H\ulZ\).
\begin{equation*}
\begin{tikzcd}
	\Si{-4} H\ulZ^* \arrow[d] = P^{-4}_{-4} & \\
	\Si{-4} H\ulZ(2,1) \arrow[d] \arrow[r] & \Si{-4} H\ul m^* \\
	\Si{-4} H\ulZ \arrow[r] & \Si{-4} H\ulg = P^{-16}_{-16}
\end{tikzcd}
\end{equation*}

Again, \(\Si{-4} H\ul m^*\) is not a slice. But suspending the cotower for \(\Si{-3} H\ul m^*\) by \(-1\) provides the missing slices: \(P^{-8}_{-8} = \Si{-4} H\phi^*_{LDR}( \ulF^*)\) and \(P^{-12}_{-12} = \Si{-3} H\ulg^2\).
\end{exmp}

\begin{exmp} \label{tower:-5:K:Z}
We suspend the cotower in \cref{tower:-4:K:Z} by \(-1\) and augment with the \(-\rho\) suspension of \cref{tower:-1:K:Z} to get a partial cotower for \(\Si{-5} H\ulZ\).
\begin{equation*}
\begin{tikzcd}
	\Si{-\rho-1} H\ulZ^* \arrow[d] = P^{-5}_{-5} & \\
	\Si{-\rho-1} H\ulZ(2,1) \arrow[d] \arrow[r] & \Si{-\rho-1} H\ul m^* = P^{-6}_{-6} \\
	\Si{-\rho-1} H\ulZ \arrow[d] \arrow[r] & \Si{-2} H\ulg = P^{-8}_{-8} \\
	\Si{-5} H\ulZ(2,1) \arrow[d] \arrow[r] & \Si{-5} H\ul m^* \\
	\Si{-5} H\ulZ \arrow[r] & \Si{-5} H\ulg = P^{-20}_{-20}
\end{tikzcd}
\end{equation*}

This time, the cotower for \(\Si{-5} H\ul m^*\) is 
\begin{equation*}
\begin{tikzcd}
    \Si{-3\rho + 1} H\phi^*_{LDR}(\ulF) = P^{-10}_{-10} \arrow[d]  & \\
    \Si{-5} H\phi^*_{LDR}(\ulF^*) \arrow[d] \arrow[r] & \Si{-3} H\ulg^3 = P^{-12}_{-12} \\
    \Si{-5} H\ul m^* \arrow[r] & \Si{-4} H\ulg^2 = P^{-16}_{-16}
\end{tikzcd}
\end{equation*}

The remaining slices of \(\Si{-5} H\ulZ\) are then \(P^{-10}_{-10} = \Si{-3\rho+1} H\phi^*_{LDR}(\ulF)\), \newline 
\(P^{-12}_{-12} = \Si{-3} H\ulg^3\), and \(P^{-16}_{-16} = \Si{-4} H\ulg^2\).
\end{exmp}

\begin{exmp} \label{tower:-7:K:Z}
The partial cotower for \(\Si{-7} H\ulZ\) follows by suspending the partial cotower in \cref{tower:-5:K:Z} by \(-2\).
\begin{equation*}
\begin{tikzcd}
	\Si{-\rho-3} H\ulZ^* \arrow[d] = P^{-7}_{-7} & \\
	\Si{-\rho-3} H\ulZ(2,1) \arrow[d] \arrow[r] & \Si{-\rho-3} H\ul m^* \\
	\Si{-\rho-3} H\ulZ \arrow[d] \arrow[r] & \Si{-4} H\ulg = P^{-16}_{-16} \\
	\Si{-7} H\ulZ(2,1) \arrow[d] \arrow[r] & \Si{-7} H\ul m^* \\
	\Si{-7} H\ulZ \arrow[r] & \Si{-7} H\ulg = P^{-28}_{-28}
\end{tikzcd}
\end{equation*}

We now have a cotower for \(\Si{-\rho-3} H\ul m^*\).
\begin{align*} 
    P^{-10}_{-10} = \Si{-\rho-3} H\phi^*_{LDR} (\ulF^*) \rightarrow \Si{-\rho-3} H\ul m^* \rightarrow \Si{-3} H\ulg^2 = P^{-12}_{-12},
\end{align*}

And a cotower for \(\Si{-7} H\ul m^*\).
\begin{equation*}
\begin{tikzcd}
    \Si{-4\rho + 1} H\phi^*_{LDR}(\ulF) = P^{-14}_{-14} \arrow[d]  & \\
    \Si{-\rho-5} H\phi^*_{LDR}(\ulF^*) \arrow[d] \arrow[r] & \Si{-4} H\ulg^3 = P^{-16}_{-16} \\
    \Si{-7} H\phi^*_{LDR}(\ulF^*) \arrow[d] \arrow[r] & \Si{-5} H\ulg^3 = P^{-20}_{-20} \\
    \Si{-7} H\ul m^* \arrow[r] & \Si{-6} H\ulg^2 = P^{-24}_{-24}
\end{tikzcd}
\end{equation*}

We now see interference from the cotower for \(\Si{-7} H\ul m^*\). Its \((-14)\)-slice appears below the \((-16)\)-slice in the partial cotower for \(\Si{-7} H\ulZ\). Additionally, both of these (partial) cotowers have a \((-16)\)-slice.

\cref{-4kslices:K:Z} and \cref{-4k-2slices:K:Z} tell us that \(P^{-14}_{-14}(\Si{-7} H\ulZ)\) is indeed \(\Si{-4\rho + 1} H\phi^*_{LDR}(\ulF)\) and \(P^{-16}_{-16}(\Si{-7} H\ulZ)\) is \(\Si{-4} H\ulg \vee \Si{-4} H\ulg^3 \simeq \Si{-4} H\ulg^4\).
\end{exmp}

All partial cotowers for \(\Si{-n} H\ulZ\) will follow this pattern of utilizing the homotopy equivalence \(\Si{-n} H\ulZ^* \simeq \Si{-\rho-n+4} H\ulZ\) to augment the bottom of the cotower for \(\Si{-n} H\ulZ\) with the cotower for \(\Si{-\rho-n+4} H\ulZ\).

%%%%%%%%%%%%%%%%%%%%%%%%%%%%%%%%%%%%%
\section{Slices of \texorpdfstring{\(\Si{-n} H_{\K}\ul \Z\)}{HZ}} \label{sec:slices:K:Z}
%%%%%%%%%%%%%%%%%%%%%%%%%%%%%%%%%%%%%

Here we determine the slices of \(\Si{-n} H_{\K}\ulZ\).

\begin{prop} \label{slice:K:Z1-5:Z*1-4}
\(\Si{n} H\ulZ\) is an \(n\)-slice for \(1\leq n\leq 5\) and \(\Si{-n} H\ulZ^*\) is a \((-n)\)-slice for \(1\leq n\leq 4\).
\end{prop}
\begin{proof}
The \(\K\)-spectra \(\Si{1} H\ulZ\), \(\Si{1} H\ulZ^*\), \(\Si{2} H\ulZ\), and \(\Si{-1} H\ulZ^*\) are 1-, 2-, and \((-1)\)-slices by \cref{charac:slice:01-1} and \cref{2slice:characterization}. The result then follows from \cref{equivalence:ZZ*} and the resulting equivalences 
\begin{align*}
    \Si{5} H \ulZ \simeq \Si{\rho+1} H \ul \Z^*, \quad \Si{4} H \ulZ \simeq \Si{\rho} H \ulZ^*, \quad \text{ and }\quad \Si{3} H \ul\Z \simeq \Si{\rho-1} H\ulZ^*.
\end{align*}
\end{proof}

%%%%%%%%%%%%%%%%%%%%%%%%%%%%%%%%%%%%%
\subsection{The \texorpdfstring{\((-n)\)}{(-n)}-slice} \label{sec:-nslice:K:Z}
%%%%%%%%%%%%%%%%%%%%%%%%%%%%%%%%%%%%%

We first establish a comparison of the slices of \(\Si{-n} H\ulZ\) with those of \(\Si{-n+4} H\ulZ\).

\begin{prop} \label{slice:reduction:ZZ^*}
Let \(n\geq 5\). Then
\begin{align*}
    P^{-k}_{-k} ( \Si{-n} H\ulZ) \simeq \Si{-\rho} P^{-k+4}_{-k+4} (\Si{-n+4} H\ulZ)
\end{align*}
for \(k\in[n,2n-1]\).
\end{prop}
\begin{proof}
By \cref{equivalence:ZZ*}, we have
\begin{align*}
    \Si{-\rho} P^{-k+4}_{-k+4} (\Si{-n+4} H\ulZ) \simeq \Si{\rho} \Si{-\rho} P^{-k}_{-k} (\Si{-n+4-\rho} H\ulZ) \simeq P^{-k}_{-k} (\Si{-n} H\ulZ^*).
\end{align*}
Thus, it is sufficient to compare the \((-k)\)-slices of \(\Si{-n} H\ulZ\) and \(\Si{-n} H\ulZ^*\).

Recall the injection \(\ulZ^* \rightarrow \ulZ \rightarrow \ulM\) from \cref{sec:K:Mackey}. We wish to show that
\begin{align} \label{inj:ZZM:-j}
    \Si{-n} H\ulZ^* \rightarrow \Si{-n} H\ulZ \xrightarrow{\mu}{} \Si{-n} \ulM
\end{align}
induces an equivalence on slices strictly above level \(-2n\).

We take the Brown-Comenetz dual of \cref{inj:ZZM:-j} to find the fiber sequence
\begin{align*}
    \Si{n} H\ulM \xrightarrow{\iota}{} \Si{n} I_{\Q/\Z} H\ulZ \rightarrow \Si{n} I_{\Q/\Z} H\ulZ^*
\end{align*}

Now \(\Si{n} H\ulM\) is \(n\)-connective, and as its underlying spectrum is contractible, when \(n\geq 1\), we know \(\Si{n} H\ulM \geq 2n\) by \cref{thm:atleastn}. Then \cite[Lemma 4.28]{HHR} provides that \(\iota\) induces an equivalence on slices strictly below level \(2n\). It then follows from \cref{duality:prop} that \(\mu\) in \cref{inj:ZZM:-j} induces an equivalence on slices strictly above level \(-2n\).
\end{proof}

This is an analogous result to \cite[Proposition 5.3]{GY}. However, the injection \(\ulZ^* \hookrightarrow \ulZ\) allows us to simplify the argument.

\begin{prop} \label{-nslice:K:Z}
Let \(n\geq 1\) and take \(r\equiv n\pmod 4\) with \(1\leq r\leq 4\). Then the top slice of \(\Si{-n} H\ulZ\) is
\begin{align*}
P^{-n}_{-n}( \Si{-n} H\ulZ) = \Si{-\frac{n-r}{4}\rho -r} H\ulZ^*
\end{align*}
\end{prop}
\begin{proof}
\cref{sec:cotowers:K:Z} gives the result for \(1\leq n\leq 4\), and \(n\geq 5\) follows from repeated application of \cref{slice:reduction:ZZ^*}.
\end{proof}

%%%%%%%%%%%%%%%%%%%%%%%%%%%%%%%%%%%%%
\subsection{The \texorpdfstring{\((-4k)\)}{(-4k)}-slices} \label{-4kslice:K:Z}
%%%%%%%%%%%%%%%%%%%%%%%%%%%%%%%%%%%%%
\phantom{a}

Here we determine the \((-4k)\)-slices of \(\Si{-n} H\ulZ\).

\begin{prop} \label{-4kslices:K:Z}
For \(n\geq 1\), 
\begin{align*}
	P^{-4k}_{-4k} ( \Si{-n} H\ulZ ) \simeq \left\lbrace \begin{array}{ll}
		\Si{-k\rho} H\ul {mg}  & 4k=n+2 \\
		\Si{-k} H\ulg^{\frac{1}{2}(4k-n-1)}  & 
		\begin{aligned}
		4k\in [n+1,2n-2] \\ n\equiv 1\pmod 2
		\end{aligned} \\
		\Si{-k\rho} H(\ulg^{\frac{1}{2}(4k-n+2)-3} \oplus \phi^*_{LDR} \ulF ) & 
		\begin{aligned}
		4k\in[n+3,2n], \\ n\equiv 0\pmod 2
		\end{aligned} \\
		\Si{-k} H\ulg^{n-k+1} & 4k\in[2n+1,4n]
	\end{array}\right.
\end{align*}
\end{prop}
\begin{proof}
We have the equivalence
\begin{align*}
    P^{-4k}_{-4k}(\Si{-n} H\ulZ)  \simeq \Si{-k\rho} H\ul\pi_0( \Si{-n+k\rho} H\ulZ).
\end{align*}
The restriction to each cyclic subgroup agrees with the slices found in \cref{tower:C2:Z:-n} and \cref{krho:fixedpts:K:Z} gives the fixed points. All that remains is to verify the result for \(4k\in[n+2,2n]\) with \(n\) even. For \mbox{\(4k\in[n+2,2n)\)}, this follows from \cref{tower:-1:K:Z}, \cref{tower:-2:K:Z}, \cref{tower:-3:K:Z}, \cref{tower:-4:K:Z}, and \cref{slice:reduction:ZZ^*}.

Now let \(4k=2n\). The transfer \(\tr{\K}{L}\) fits in the cofiber sequence
\begin{equation*}
\begin{tikzcd}
    (\Si{-n+k\rho} H\ulZ)^L \arrow[r, "\tr{\K}{L}"] & (\Si{-n+k\rho} H\ulZ)^\K \arrow[r] & (\Si{-n+k\rho+\beta} H\ulZ)^\K.
\end{tikzcd}
\end{equation*}
This gives the exact sequence
\begin{equation} \label{transfer:fixedpts:proof:Z:-4k}
\begin{tikzcd}
	\pi_n^L( \Si{k\rho} H\ulZ ) \arrow[r, "\tr{\K}{L}"] & \pi_n^\K(\Si{k\rho} H\ulZ) \arrow[r] & \pi_n^\K (\Si{k\rho+\beta} H\ulZ).
\end{tikzcd}
\end{equation}

We wish to show this transfer is trivial. By \cref{krho:fixedpts:K:Z} and \cref{krho+beta:fixedpoints:K:Z},  \cref{transfer:fixedpts:proof:Z:-4k} becomes
\begin{equation*}
\begin{tikzcd}
    \F \arrow[r, "\tr{\K}{L}"] & \F^{\frac{1}{2}(n+2)} \arrow[r, two heads] & \F^{\frac{1}{2}(n+2)} \arrow[r] & 0 = \pi_{n-1}^L( \Si{k\rho} H\ulZ).
\end{tikzcd}
\end{equation*}
Thus, \(\tr{\K}{L}\) must be zero. A similar argument shows that \(\tr{\K}{R}\) and \(\tr{\K}{D}\) are trivial as well.

The restriction \(\res{\K}{L}\) fits into the cofiber sequence
\begin{equation*}
\begin{tikzcd}
    ( \Si{-n+k\rho-\beta} H\ulZ)^\K \arrow[r] & ( \Si{-n+k\rho} H\ulZ)^\K \arrow[r, "\res{\K}{L}"] & ( \Si{-n+k\rho} H\ulZ)^L
\end{tikzcd}
\end{equation*}
This gives the exact sequence
\begin{equation} \label{res:fixedpts:proof:Z:-4k}
\begin{tikzcd}
	\pi_n^\K( \Si{k\rho-\beta} H\ulZ ) \arrow[r] & \pi_n^\K(\Si{k\rho} H\ulZ) \arrow[r, "\res{\K}{L}"] & \pi_n^L(\Si{k\rho} H\ulZ).
\end{tikzcd}
\end{equation}

We wish to show this restriction is surjective. By \cref{krho:fixedpts:K:Z} and \cref{krho-beta:fixedpoints:K:Z},  \cref{res:fixedpts:proof:Z:-4k} becomes
\begin{equation*}
\begin{tikzcd}
    \pi_{n+1}^L( \Si{k\rho} H\ulZ ) = 0 \arrow[r] & \F^{\frac{1}{2}(n)} \arrow[r, hook] & \F^{\frac{1}{2}(n)+1} \arrow[r, "\res{\K}{L}"] & \F
\end{tikzcd}
\end{equation*}
Thus, \(\res{\K}{L}\) must be surjective. A similar argument shows that \(\res{\K}{R}\) and \(\res{\K}{D}\) are surjective as well.

All that remains to be shown is that the restrictions have distinct kernels. Consider the cofiber sequence
\begin{equation*}
\begin{tikzcd}
    \Si{-n+k\rho-\beta} H\ulZ \arrow[r] & \Si{-n+k\rho} H\ulZ \arrow[r] & K/L_+ \wedge \Si{-n+k\rho} H\ulZ.
\end{tikzcd}
\end{equation*}

This results in the exact sequence
\begin{equation} \label{exactseq:-4kproof:K:Z}
\begin{tikzcd}
    \ul\pi_{n} ( \Si{k\rho-\beta} H\ulZ ) \arrow[r] & \ul\pi_{n} ( \Si{k\rho} H\ulZ) \arrow[r] & \uparrow^\K_L \downarrow^\K_L \ul\pi_{n} (\Si{k\rho} H\ulZ )
\end{tikzcd}
\end{equation}

which by \cref{ksigma+r:homotopy:C2:Z}, \cref{krho:fixedpts:K:Z}, and \cref{krho-beta:fixedpoints:K:Z} is
\begin{equation*}
\begin{tikzcd}[scale cd = 0.85,column sep = small]
	%%% begin first row
	% column 1
	&&  \F^{\frac{n}{2}} \arrow[rrr, "\varphi"] 
	\arrow[d, \resC] \arrow[dr, \resC]
	&&[+15pt] 
	% column 2
	& \F^{\frac{n}{2}+1} \arrow[rrr, "\res{\K}{L}"]
	\arrow[d, \resC] \arrow[dr, \resC] \arrow[dl, \resC] 
	&& [+15pt]
	% column 3
	& \F
	\arrow[dl, \resC, bend right = 10]
	&&  \\ [-5pt]
	% end first row
	%%% begin second row
	% column 1
	& 0
	& \F	
	& \F \arrow[r, "0 \text{ } 1 \text{ } 1"] 
	% column 2
	& \F 	
	& \F
	& \F \arrow[r, "\Delta \text{ } 0 \text{ } 0"]
	% column 3
	& 
	\F[\K/L]
	\arrow[ur, \trC, bend right = 10]
	& 0
	& 0
	&  \\ [-5pt]
	% end second row
	%%% begin third row
	&& 0 \arrow[rrr, "0"]
	&&
	% column 2
	& 0 \arrow[rrr, "0"] 
	&&
	& 0 
	&& 
\end{tikzcd}
\end{equation*}	

Each restriction is surjective with kernel of rank \(\frac{n}{2}\). As 
\begin{align*}
    \dim \pi_{n}^\K ( \Si{k\rho} H\ulZ) = \frac{n}{2}+1,
\end{align*}
it is sufficient to show that the kernels are pairwise distinct.

Because the diagram on the left commutes, we find that
\begin{align*}
    \im\varphi \cap \ker(\res{\K}{R}) = \{0\} \qquad \text{and} \qquad \im\varphi \cap \ker(\res{\K}{D}) = \{0\}.
\end{align*}
As \(\im\varphi = \ker(\res{\K}{L})\), we have that \(\ker(\res{\K}{L})\) is distinct from \(\ker(\res{\K}{R})\) and \(\ker(\res{\K}{D})\). Replacing \(\beta\) by \(\alpha\) shows that \(\ker(\res{\K}{R})\) and \(\ker(\res{\K}{D})\) are distinct as well.
\end{proof}

\begin{prop} \label{krho:fixedpts:K:Z}
For \(G=\K\) and \(k\geq 1\),
\begin{align*}
	\pi_i^\K( \Si{k\rho} H\ulZ) &= \left\lbrace \begin{array}{ll}
		\Z & i=4k \\
		\F^{\frac{1}{2}(4k-1-i)} & i\in [2k+1, 4k-1] \text{ odd} \\
		\F^{\frac{1}{2}(4k+2-i)} & i\in[2k+2, 4k-2] \text{ even} \\
		\F^{i-k+1} & i\in[k,2k]
	\end{array}\right.
\end{align*}
\end{prop}
\begin{proof}
Note that
\begin{align*}
    P^{-4k}_{-4k}(\Si{-4k} H\ulZ) &\simeq \Si{-k\rho} H\ul\pi_{4k}(\Si{k\rho} H\ulZ).
\end{align*}
Thus, by \cref{-nslice:K:Z}, \(H\ul\pi_{4k}(\Si{k\rho} H\ulZ) \simeq H\ulZ\). So the result holds for \(i=4k\) and we only need consider \(i\leq 4k-1\).

We will use the resulting long exact sequences in homotopy resulting from the cofiber sequences
\begin{align}
    (\K/L_+ &\rightarrow S^0 \rightarrow S^\beta) \wedge \Si{k\rho+1} H\ulZ  \label{cofib:beta} \\
    (\K/R_+ &\rightarrow S^0 \rightarrow S^\alpha) \wedge \Si{k\rho+\beta+1} H\ulZ \label{cofib:alpha} \\
    (\K/D_+ &\rightarrow S^0 \rightarrow S^\gamma) \wedge \Si{k\rho+\alpha+\beta+1} H\ulZ \label{cofib:gamma} \\
    \Si{k\rho} H\ulZ &\xrightarrow{2}{} \Si{k\rho} H\ulZ \rightarrow \Si{k\rho} H\ulF \label{cofib:mult2}
\end{align}
where \cref{cofib:mult2} is induced by the short exact sequence of Mackey functors \(\ulZ \xrightarrow{2}{} \ulZ \rightarrow \ulF\). 

For \(k=0\), \cref{cofib:beta} - \cref{cofib:gamma} provide that
\begin{align*} 
\ul\pi_n(\Si{\rho} H\ulZ) &= \left\lbrace \begin{array}{cc}
   \ulZ & n=4 \\
   \ul 0 & n=3 \\
   \ul{mg} & n=2 \\
   \ulg & n=1
\end{array}\right.
\end{align*}

Consequently, the result holds for \(k=1\). We now proceed by induction on \(k\). Assume the result holds for \(k\). By restriction to \(L\), \(D\), and \(R\), we find that \(\ul\pi_i(\Si{(k+1)\rho} H\ulZ)\) is a pullback over \(\K\) for \(i\leq 4k+3\) odd, and consequently, 2-torsion as in \cite[Remark 2.13]{Z}.

In the long exact sequence of fixed points resulting from \cref{cofib:beta} - \cref{cofib:gamma}, we have at the \(\K/\K\) level,
\begin{equation*}
\begin{tikzcd}[scale cd = 0.9]
    %%% row 1 - first les
    \pi_i^{C_2} (\Si{2k\rho_{C_2}+1} H_{C_2}\ulZ) \arrow[r] &[-10pt] \pi_i^\K( \Si{k\rho+1} H\ulZ) \arrow[r] &[-10pt] \pi_{i}^\K(\Si{k\rho+\beta + 1} H\ulZ) \\
    %%% row 2 - second les
    \pi_i^{C_2} (\Si{(2k+1)\rho_{C_2}} H_{C_2}\ulZ) \arrow[r] & \pi_i^\K( \Si{k\rho+\beta+1} H\ulZ) \arrow[r] & \pi_{i}^\K(\Si{k\rho+\alpha+\beta +1 } H\ulZ) \\
    %%% row 3 - third les
    \pi_i^{C_2} (\Si{(2k+2)\rho_{C_2}-1} H_{C_2}\ulZ) \arrow[r] & \pi_i^\K( \Si{k\rho+\alpha+\beta+ 1} H\ulZ) \arrow[r] & \pi_{i}^\K(\Si{(k+1)\rho} H\ulZ)
\end{tikzcd}
\end{equation*}

When \(2k+1\leq i \leq 4k+1\), by \cref{ksigma+r:homotopy:C2:Z},
\begin{align*}
    \pi_i^{C_2} (\Si{2k\sigma + 2k+1} H_{C_2}\ulZ) &= \left\lbrace \begin{array}{ll}
        \F &  i \text{ odd} \\
        0 & i \text{ even}
    \end{array}\right.
\end{align*}

Now for \(i\) even, we have in our long exact sequence,
\begin{equation*}
\begin{tikzcd}
    0 \ar[r] & \pi_i^\K( \Si{k\rho+1} H\ulZ) \ar[r, hook] & \pi_{i}^\K(\Si{k\rho+\beta + 1} H\ulZ) \ar[r] & \phantom{a} \\
    \F \ar[r] & \pi_{i-1}^\K( \Si{k\rho+1} H\ulZ) \ar[r, twoheadrightarrow] & \pi_{i-1}^\K(\Si{k\rho+\beta + 1} H\ulZ) \ar[r, "0", twoheadrightarrow] & \phantom{a}
\end{tikzcd}
\end{equation*}

Consequently, when \(i\) is even,
\begin{align*}
    \text{2-rk } \pi_{i}^\K(\Si{k\rho+\beta + 1} H\ulZ) \leq \text{2-rk } \pi_i^\K( \Si{k\rho+1} H\ulZ) + 1,
\end{align*}
where equality occurs if \(\pi_i^\K( \Si{k\rho+\beta + 1} H\ulZ) = \pi_i^\K( \Si{(k\rho+1} H\ulZ) \oplus \F\).

And, when \(i\) is odd,
\begin{align*}
    \text{2-rk } \pi_{i}^\K(\Si{k\rho+\beta + 1} H\ulZ) \leq \text{2-rk } \pi_i^\K( \Si{k\rho+1} H\ulZ).
\end{align*}

For \(i\leq 2k\) we have
\begin{equation*}
\begin{tikzcd}
    0 \ar[r] & \pi_{i}^\K( \Si{k\rho+1} H\ulZ) \ar[r, hook, two heads] & \pi_{i}^\K(\Si{k\rho+\beta + 1} H\ulZ) \ar[r, "0", twoheadrightarrow] & \phantom{a}    
\end{tikzcd}
\end{equation*}

Thus, for \(i\leq 2k\),
\begin{align*}
    \pi_{i}^\K(\Si{k\rho+\beta + 1} H\ulZ) \cong  \pi_{i}^\K( \Si{k\rho+1} H\ulZ).
\end{align*}

A similar statement holds for 
\begin{align*}
    \text{2-rk } \pi_{i}^\K(\Si{ k\rho+\alpha+\beta +1} H\ulZ) \qquad \text{ and } \qquad \text{2-rk } \pi_{i}^\K(\Si{(k+1)\rho} H\ulZ).
\end{align*}

From \cite[Corollary 7.2]{GY}, in the long exact sequence resulting from \cref{cofib:mult2}, we have
\begin{align*}
    \F^l \cong \pi_{2i+1}^\K(\Si{(i+1)\rho} H\ulZ) &\xrightarrow{0} \pi_{2i+1}^\K(\Si{(i+1)\rho} H\ulZ) \hookrightarrow \pi_{2i+1}^\K(\Si{(i+1)\rho} H\ulF) \cong \F^{2i+1} \\ 
    &\twoheadrightarrow \pi_{2i}^\K(\Si{(i+1)\rho} H\ulZ) \cong \pi_{2i-1}^\K(\Si{i\rho} H\ulZ) \cong \F^{i}.
\end{align*}

Consequently, \(l=i+1\) and \(\pi_{2i+1}^\K(\Si{(i+1)\rho} H\ulZ)\cong \F^{i+1}\). We then have in our long exact sequence
\begin{align*}
    \pi_{2i+2}^\K(\Si{(i+1)\rho} H\ulZ) \rightarrow \pi_{2i+2}^\K(\Si{(i+1)\rho} H\ulF) \cong \F^{2i+3} \twoheadrightarrow \F^{i+1} \cong \pi_{2i+1}^\K(\Si{(i+1)\rho} H\ulZ).
\end{align*}

Thus,
\begin{align*}
    \text{2-rk } \pi_{2i+2}^\K(\Si{(i+1)\rho} H\ulZ) \geq i+2 = \text{2-rk } \pi_{2i+1}^\K(\Si{i\rho} H\ulZ) + 3.
\end{align*}

We achieve the maximum bound for \(\text{2-rk } \pi_{2i+2}^\K(\Si{(i+1)\rho} H\ulZ)\); thus,
\begin{align*}
    \pi_{2i+2}^\K(\Si{(i+1)\rho} H\ulZ) \cong \pi_{2i+1}^\K(\Si{i\rho} H\ulZ) \oplus \F^3 \cong \F^{i+2}.
\end{align*}

The rest of the result now follows from this long exact sequence in a similar manner.
\end{proof}

\begin{prop} \label{krho-beta:fixedpoints:K:Z}
For \(G=\K\) and \(k\geq 1\),
\begin{align*}
    \pi_i^\K( \Si{k\rho-\beta} H\ulZ) &= \left\lbrace \begin{array}{ll}
    \F^{\frac{1}{2}(4k-i)} & i\in[2k, 4k-2] \text{ even} \\
    \F^{\frac{1}{2}(4k-i-1)} & i\in[2k-1, 4k-3] \text{ odd} \\
    \F^{i-k+1} & i\in[k,2k]
    \end{array}\right.
\end{align*}
\end{prop}
\begin{proof}
This follows from a similar argument as in  \cref{krho:fixedpts:K:Z}.
\end{proof}

\begin{prop} \label{krho+beta:fixedpoints:K:Z}
For \(G=\K\) and \(k\geq 1\),
\begin{align*}
    \pi_i^\K( \Si{k\rho+\beta} H\ulZ) &= \left\lbrace \begin{array}{ll}
    \F^{\frac{1}{2}(4k-i)+1} & i\in[2k+2, 4k] \text{ even} \\
    \F^{\frac{1}{2}(4k-i+1)} & i\in[2k+1, 4k-1] \text{ odd} \\
    \F^{i-k+1} & i\in[k,2k]
    \end{array}\right.
\end{align*}
\end{prop}
\begin{proof}
This follows from a similar argument as in  \cref{krho:fixedpts:K:Z}.
\end{proof}

%%%%%%%%%%%%%%%%%%%%%%%%%%%%%%%%%%%%%
\subsection{The \texorpdfstring{\((-4k-2)\)}{(-4k-2)}-slices} \label{-4k-2slice:K:Z}
%%%%%%%%%%%%%%%%%%%%%%%%%%%%%%%%%%%%%
\phantom{a}

We now determine the \((-4k-2)\)-slices of \(\Si{-n} H\ulZ\).

\begin{prop} \label{-4k-2slices:K:Z}
For even \(n\), \(\Si{-n} H\ulZ\) has no \((-4k-2)\)-slices, except for possibly the \(-n\) slice.
For odd \(n\geq 1\),
\begin{align*}
	P^{-4k-2}_{-4k-2} ( \Si{-n} H\ulZ ) &\simeq \left\lbrace
	\begin{array}{ll}
		\Si{-k\rho-1} H\ul m^* & 4k+2=n+1 \\
		\Si{-k\rho-1} H\phi^*_{LDR} \ulF^* & 4k+2\in (n+1,2n]
	\end{array}\right.
\end{align*}
\end{prop}
\begin{proof}
When \(n\) is even, the slices in \cref{-4kslices:K:Z} restrict to the all of the appropriate slices for \(\Si{-n} H_{C_2}\ulZ\). If, then, \(\Si{-n} H_\K\ulZ\) has a \((-4k-2)\)-slice, it must be a pullback over \(\K\). But this is a contradiction as such slices are \((-4k)\)-slices. Thus, for \(n\) even, \(\Si{-n} H_\K\ulZ\) has no \((-4k-2)\)-slices.

Now let \(n\) be odd. We first handle the case \(4k+2=2n\). By \cref{rho:susp:commutes},
\begin{align*}
    P^{-4k-2}_{-4k-2}(\Si{-n} H_\K\ulZ) \simeq \Si{-(k+1)\rho} P^2_2( \Si{-n+(k+1)\rho} H_\K\ulZ).
\end{align*}
For clarity, let 
\begin{align*}
    X := P^2_2( \Si{-n+(k+1)\rho} H_\K\ulZ).
\end{align*}
Now from \cref{-2kslice:C2:HZ},
\begin{align*}
    i_{H}^* X \simeq \Si{1} H_{C_2}\ulg.
\end{align*}
where \(H\) is \(L\), \(D\), or \(R\). Thus, by \cref{2slice:characterization},
\begin{align*}
    \ul\pi_2( X) = \phi^*_{\K} B \qquad \text{ and } \qquad \ul\pi_1( X ) = \ul A
\end{align*}
where \(B\) is some group and 
\begin{equation*}
    \ul A = \begin{tikzcd}[ampersand replacement=\&, column sep=scriptsize, bend right=15]
	\& A_\K \ar[dl] \ar[d] \ar[dr] \& \\
	\F \ar[ur] \& \F \ar[u] \& \F \ar[ul] \\
	\& 0 \&
    \end{tikzcd}		
\end{equation*}
and \(A_\K \rightarrow \F\oplus \F\oplus \F\) is injective. That is, \(A_\K=\F^n\) with \(0\leq n\leq 3\).

If \(B\neq 0\), then \(X\) cannot be a 2-slice. Consequently, \(X\simeq \Si{1} H\ul A\).

Because \(\ulZ\) is invariant under the automorphisms of \(\K\), the spectrum \(\Si{-n} H\ulZ\) is as well. Therefore, the slices of \(\Si{-n} H\ulZ\) are also invariant under the automorphisms of \(\K\).

Thus, \(\ul A\) must be one of the following:
\begin{align*}
    \phi^*_{LDR} \ul f, \quad \ul m, \quad \ul{mg}, \quad \text{ or } \quad \phi^*_{LDR} \ulF.
\end{align*}

Except for degree \(-k-1\), \(\Si{-(k+1)\rho +1} H\ul A\) has the same homotopy for each choice of \(A\). 

For degree \(-k-1\), \(\ul\pi_{-k-1}(\Si{-(k+1)\rho +1} H\ul A)\) is a pullback over \(\K\) of dimension 3, 2, 1, or 0.

From \cref{phiF:SSS} we find that \(\pi_{-k-1}^\K(\Si{-(k+1)\rho +1} H\ul A)=0\). The only choice of \(\ul A\) that meets this requirement is \(\phi^*_{LDR}\ulF\).

Now let \(4k+2\in[n+1,2n-1]\). The base cases are established in \cref{tower:-1:K:Z}, \cref{tower:-3:K:Z}, and \cref{tower:-5:K:Z}. The result then follows from \cref{slice:reduction:ZZ^*}.
\end{proof}

%%%%%%%%%%%%%%%%%%%%%%%%%%%%%%%%%%%%%
\section{Slices of \texorpdfstring{\(\Si{n} H\ul \Z\)}{HZ}} \label{sec:slices:n:K:Z}
%%%%%%%%%%%%%%%%%%%%%%%%%%%%%%%%%%%%%

Recall from \cref{slice:K:Z1-5:Z*1-4} that \(\Si{n} H\ulZ\) is a slice for \(1\leq n\leq 5\).

\begin{prop} \label{slice:reduction:K:n:Z}
Let \(n\geq 6\). Then
\begin{align*}
	P^k_k(\Si{n} H\ulZ) &\simeq \Si{\rho} P^{k-4}_{k-4} (\Si{n-4} H\ulZ)
\end{align*}
for \(k\in[n,2n-7]\).
\end{prop}
\begin{proof}
We employ a similar argument as in \cref{slice:reduction:ZZ^*}. Note that
\begin{align*}
	P^k_k(\Si{n} H\ulZ) &\simeq \Si{\rho} P^{k-4}_{k-4}(\Si{n-\rho} H\ulZ) \simeq \Si{\rho} P^{k-4}_{k-4}(\Si{n-4} H\ulZ^*).
\end{align*}

Consequently, it is sufficient to compare the \((k-4)\)-slices of \(\Si{n-4} H\ulZ\) and \(\Si{n-4} H\ulZ^*\).

The exact sequence \(\ulZ^* \rightarrow \ulZ \rightarrow \ulM\)  provides the fiber sequence 
\begin{align*}
    \Si{j-1} H\ulM \rightarrow \Si{j} H\ulZ^* \xrightarrow{\iota}{} \Si{j} H\ulZ.
\end{align*}
Then, because \(\Si{j-1} H\ulM\geq 2j-2\), by \cite[Lemma 4.28]{HHR}, \(\iota\) induces an equivalence on slices strictly below level \(2j-2\). Taking \(j=n-4\) gives the result.
\end{proof}

\begin{exmp} \label{tower:6:K:Z}
The tower for \(\Si{6} H\ulZ\) is
\begin{equation*}
	\begin{tikzcd}
		P^8_8 = \Si{2} H\ulg \arrow[r] & \Si{6} H\ulZ \arrow[d] \\
		& P^6_6 = \Si{\rho+2} H\ulZ(2,1)^*
	\end{tikzcd}
\end{equation*}
\end{exmp}
\begin{proof}
We have the short exact sequence \(\ulZ^* \rightarrow \ulZ(2,1)^* \rightarrow \ulg\). This leads to the cofiber sequence
\begin{equation*}
\begin{tikzcd}
	P^8_8 = \Si{2} H\ulg \arrow[r] & \Si{6} H\ulZ \simeq \Si{\rho+2} H\ulZ^* \arrow[r] & \Si{\rho+2} H\ulZ(2,1)^* = P^6_6.
\end{tikzcd}
\end{equation*}
\end{proof}

\begin{exmp} \label{tower:7:K:Z}
The tower for \(\Si{7} H\ulZ\) is
\begin{equation*}
	\begin{tikzcd}
		P^{12}_{12} = \Si{3} H\ulg \arrow[r] & \Si{7} H\ulZ \arrow[d] \\
		P^{8}_8 = \Si{\rho+2} H\ul m \arrow[r] & \Si{\rho+3} H\ulZ(2,1)^* \arrow[d] \\
		& P^7_7 = \Si{\rho+3} H\ulZ
	\end{tikzcd}
\end{equation*}
\end{exmp}
\begin{proof}
We suspend the tower in \cref{tower:6:K:Z} by 1 and augment with the cofiber sequence \(\Si{\rho+2} H\ul m \rightarrow \Si{\rho+3} H\ulZ(2,1)^* \rightarrow \Si{\rho+3} H\ulZ\) which arises from the short exact sequence \(\ulZ(2,1)^* \rightarrow \ulZ \rightarrow \ul m\).
\end{proof}

We now determine the slices of \(\Si{n} H\ulZ\).

\begin{thm} \label{duality:slices:K:Z}
Let \(n\geq 6\). For \(k\geq n+2\), 
\begin{align*}
	P^k_k (\Si{n} H\ulZ) &\simeq \Si{\rho} I_{\mathbb{Q}/\Z} \, P^{-k+4}_{-k+4} \, \Si{-n+5} H\ulZ.
\end{align*}
\end{thm}
\begin{proof}
Take \(r\equiv n-5\pmod 4\) with \(1\leq r\leq 4\). We may map the top slice of \(\Si{-n+5} H\ulZ\) into it to find the cofiber sequence
\begin{align}
	P^{-n+5}_{-n+5} = \Si{-\frac{n-5-r}{4}\rho -r} H\ulZ^* \rightarrow \Si{-n+5} H\ulZ \rightarrow P^{-n+5-1} \, \Si{-n+5} H\ulZ. \label{slices:cofiber}
\end{align}
Note that all slices of \(P^{-n+4} H\ulZ\) are torsion, so then
\begin{align*}
    I_{\mathbb{Q}/\Z} \, P^{-n+4} \, H\ulZ = \Si{1} I_{\Z} \, P^{-n+4} \, \Si{-n+5} H\ulZ.
\end{align*}
Apply \(I_\Z\) to \cref{slices:cofiber} and suspend by one to find
\begin{align*}
	\Si{1} I_{\Z} \, P^{-n+4} \, \Si{-n+5} H\ulZ \rightarrow \Si{n-4} H\ulZ^* \rightarrow \Si{\frac{n-5-r}{4}\rho +r+1} H\ulZ.
\end{align*}
We can rewrite this as
\begin{align*}
    I_{\mathbb{Q}/\Z} \, P^{-n+4} \, \Si{-n+5} H\ulZ \rightarrow \Si{n-\rho} H\ulZ \rightarrow \Si{\frac{n-5-r}{4}\rho +r+1} H\ulZ.
\end{align*}
Finally, suspend by \(\rho\) to obtain
\begin{align}
    \Si{\rho} I_{\mathbb{Q}/\Z} \, P^{-n+4}  \,\Si{-n+5} H\ulZ \rightarrow \Si{n} H\ulZ \rightarrow \Si{\frac{n-1-r}{4}\rho +r+1} H\ulZ. \label{slices:cofiber:1}
\end{align}
Note that \(\Si{\frac{n-1-r}{4}\rho +r+1} H\ulZ\) is an \(n\)-slice and
\begin{align*}
    \Si{\rho} I_{\mathbb{Q}/\Z} \, P^{-n+4} \, \Si{-n+5} H\ulZ\in [n,4(n-4)].
\end{align*}
Furthermore, if \(n\not\equiv 2\pmod 4\), 
\begin{align*}
    \Si{\rho} I_{\mathbb{Q}/\Z} \, P^{-n+4} \, \Si{-n+5} H\ulZ\in [n+1,4(n-4)].
\end{align*}
From \cref{duality:prop},
\begin{align*}
    I_{\mathbb{Q}/\Z} \, P^{-k}_{-k} \, \Si{-n+5} H\ulZ &\simeq P^{k}_{k} I_{\mathbb{Q}/\Z} \, \Si{-n+5} H\ulZ.
\end{align*}
Consequently, \(\Si{\rho} I_{\mathbb{Q}/\Z} \, P^{-n+4}  \,\Si{-n+5} H\ulZ\) provides all slices of \(\Si{n} H\ulZ\).

Now suppose \(n\equiv 2\pmod 4\) so that \(r=1\). Then from \cref{duality:prop},
\begin{align*}
    P^n_n ( \Si{\rho} I_{\mathbb{Q}/\Z} \, P^{-n+4} \, \Si{-n+5} H\ulZ ) &\simeq \Si{\rho} I_{\mathbb{Q}/\Z} \, P^{-n}_{-n} (  P^{-n+4} \, \Si{-n+5} H\ulZ ) \\
    &\simeq \Si{\rho} I_{\mathbb{Q}/\Z} \, \Si{-(\frac{n-6}{4}+1)\rho+1} H\ul {mg} \\
    &\simeq \Si{\frac{n-2}{4}\rho + 1} H\ul m
\end{align*}

Apply \(P^n_n(-)\) to \cref{slices:cofiber:1} to get the extension
\begin{align*}
    \Si{\frac{n-2}{4}\rho + 1} H\ul m \rightarrow P^n_n \Si{n} H\ulZ \rightarrow \Si{\frac{n-2}{4}\rho +2} H\ulZ
\end{align*}

and the fiber sequence
\begin{align*}
    \Si{\rho} I_{\mathbb{Q}/\Z} \, P^{-n+5} H\ulZ \rightarrow \Si{\rho} I_{\mathbb{Q}/\Z} \, P^{-n+4} H\ulZ \rightarrow P^n_n \Si{n} H\ulZ.
\end{align*}

Now \(\Si{\rho} I_{\mathbb{Q}/\Z} \, P^{-n+5} H\ulZ \in [n+1,4(n-4)]\) and thus supplies the remaining slices of \(\Si{n} H\ulZ\).
\end{proof}

\begin{prop} \label{slices:K:n:Z}
Let \(n\geq 6\) and set \(r\equiv n\pmod 4\) with \(2\leq r\leq 5\). The \(n\)-slice of \(\Si{n} H\ulZ\) is
\begin{align*}
	P^n_n \Si{n} H\ulZ &\simeq \left\lbrace \begin{array}{ll}
		\Si{\frac{n-2}{4}\rho + 2} H\ulZ(2,1)^* & n\equiv 2\pmod 4 \\
		\Si{\frac{n-r}{4}\rho + r} H\ulZ & \text{otherwise}
	\end{array}\right.
\end{align*}
\end{prop}
\begin{proof}
For \(n\not\equiv 2\pmod 4\), this follows from \cref{slices:cofiber:1}. When \(n\equiv 2\pmod 4\) it follows from \cref{tower:6:K:Z} and repeated application of \cref{slice:reduction:K:n:Z}.
\end{proof}

\begin{prop} \label{slices:K:4k:Z}
Let \(n\geq 6\). The \(4k\)-slices of \(\Si{n} H\ulZ\) are
	\begin{align*}
	P^{4k}_{4k} ( \Si{n} H\ulZ ) \simeq \left\lbrace \begin{array}{ll}
		\Si{k\rho} H\ul{mg}^*  & 4k=n+1 \\
		\Si{k} H\ulg^{\frac{1}{2}(4k-n)}  & \begin{aligned}
		4k\in [n+2,2n-8], \\ n\equiv 0\pmod 2
		\end{aligned} \\
		\Si{k\rho} H(\ulg^{\frac{1}{2}(4k-n+3)-3} \oplus \phi^*_{LDR} \ulF^* ) & 
		\begin{aligned} 4k\in[n+2,2n-6], \\ n\equiv 1\pmod 2
		\end{aligned} \\
		\Si{k} H\ulg^{n-k-3} & 4k\in[2n-4,4(n-4)]
	\end{array}\right.
\end{align*}
\end{prop}
\begin{proof}
This follows from \cref{-4kslices:K:Z} and \cref{duality:slices:K:Z}.
\end{proof}

\begin{prop} \label{slices:4k+2:K:Z}
For \(n\geq 6\), except for possibly the \(n\) slice,
the nontrivial \((4k+2)\)-slices of \(\Si{n} H\ulZ\) are
\begin{align*}
	P^{4k+2}_{4k+2} (\Si{n} H\ulZ) &\simeq \left\lbrace \begin{array}{ll}
		\Si{k\rho+1} H\phi^*_{LDR} \ulF & 
		\begin{aligned} 4k+2\in[n+2,2n-6], \\ n\equiv 0\pmod 2
		\end{aligned}
	\end{array}\right.
\end{align*}
\end{prop}
\begin{proof}
By \cref{2kslice:C2:HZ}, \(\Si{n} H_{C_2}\ulZ\) has no \((4k+2)\)-slices except in the range \([n+2,2n-6]\) when \(n\) is even. So for \(4k+2\) not in this range, any such slice must be a pullback over \(\K\). But then it is a \(4k\)-slice. For \(4k+2\in[n+2,2n-6]\), the result follows from \cref{-4k-2slices:K:Z} and \cref{duality:slices:K:Z}.
\end{proof}

%%%%%%%%%%%%%%%%%%%%%%%%%%%%%%%%%%%%%
\subsection{Comparison with the Slices of \texorpdfstring{\(\Si{n} H\ulF\)}{HF}} \label{subsec:compGY}
%%%%%%%%%%%%%%%%%%%%%%%%%%%%%%%%%%%%%
\phantom{a}

This work is complementary to \cite{GY}, which calculates the slices of \(\Si{n} H\ulF\) for \(n\geq 1\). One would hope that the exact sequence \(\ulZ \xrightarrow{2}{} \ulZ \rightarrow \ulF\) could play a role in recovering the slices of \(\Si{n} H\ulZ\) from the slices of \(\Si{n} H\ulF\) or vice versa, but this is not always the case.

When \(G=C_2\), \cite[Theorem 3.18]{GY} shows that the slices of \(\Si{n} H_{C_2}\ulF\) contain both even and odd suspensions of \(H_{C_2}\ulg\), whereas \cref{tower:C2:Z:n} shows that \(\Si{n} H_{C_2}\ulZ\) has only even or odd suspensions of \(H_{C_2}\ulg\). This is illustrated in \cref{tab-C2slices}.

\begin{table}[h] 
\caption[Comparison of {$C_2$}-slices]{Comparison of $C_2$-slices}
\label{tab-C2slices}
\begin{center}
	{\renewcommand{\arraystretch}{1.5}\setlength{\tabcolsep}{15pt}
		\begin{tabular}{|c|c|c|}
			\hline 
			  Slices of \(\Si{9} H_{C_2}\ulZ\)	& Slices of \(\Si{9} H_{C_2}\ulF\) & Slices of \(\Si{10} H_{C_2}\ulZ\) \\  \hline
			   &\(P^{14}_{14} = \Si{7} H_{C_2}\ulg\) & \(P^{14}_{14} = \Si{7} H_{C_2}\ulg\) \\
			   \(P^{12}_{12} = \Si{6} H_{C_2}\ulg\) &\(P^{12}_{12} = \Si{6} H_{C_2}\ulg\) & \\
			   &\(P^{10}_{10} = \Si{5} H_{C_2}\ulg\) & \(P^{10}_{10} = \Si{2\rho+2} H_{C_2}\ulZ^*\) \\
			   \(P^{9}_{9} = \Si{2\rho+1} H_{C_2}\ulZ^*\) & \(P^{9}_{9} = \Si{2\rho+1} H_{C_2}\ulF^*\) & \\ \hline
	\end{tabular} }
\end{center}
\end{table}

The \(2k\)-slices of \(\Si{n} H_{C_2}\ulZ\) and \(\Si{n+1} H_{C_2}\ulZ\) only combine to give the slices of \(\Si{n} H_{C_2}\ulF\) when \(n\equiv 3,4\pmod 4\). When \(n\equiv 5,6\pmod 4\), the \(\Si{n} H_{C_2}\ulZ\) and \(\Si{n+1} H_{C_2}\ulZ\) slices miss the \((n+r)\)-slice of \(\Si{n} H_{C_2}\ulF\), where \(r=1,2\), respectively. For example, neither \(\Si{9} H_{C_2}\ulZ\) nor \(\Si{10} H_{C_2}\ulZ\) has a slice equivalent to \(\Si{5} H_{C_2}\ulg\), but \(\Si{9} H_{C_2}\ulF\) does.

We can recover the \((4k)\)-slices of \(\Si{n} H\ulF\) from the \((4k)\)-slices of \(\Si{n} H_\K\ulZ\) and \(\Si{n+1} H_\K\ulZ\). As in \cref{krho:fixedpts:K:Z}, we use the sequence \(\ulZ \xrightarrow{2}{} \ulZ \rightarrow \ulF\) to get the cofiber sequence \(\Si{-k\rho} H_\K\ulZ \rightarrow \Si{-k\rho} H_\K\ulF \rightarrow \Si{1-k\rho} H_\K\ulZ\). We then have the long exact sequence in homotopy
\begin{equation} \label{les:ZZF:4kslices}
\begin{tikzcd}[column sep=small, row sep=small]
    \ul\pi_{-n}( \Si{-k\rho} H_\K\ulZ) \arrow[r, "2"] & 
    \ul\pi_{-n}( \Si{-k\rho} H_\K\ulZ) \arrow[r] & \phantom{\ul\pi_{-n}( \Si{-k\rho} H_\K\ulZ)} \\ 
    \ul\pi_{-n}( \Si{-k\rho} H_\K\ulF) \arrow[r] &
    \ul\pi_{-n-1}( \Si{-k\rho} H_\K\ulZ) \arrow[r, "2"] &
    \ul\pi_{-n-1}( \Si{-k\rho} H_\K\ulZ).
\end{tikzcd}
\end{equation}
When \(n\leq 4k-1\), all groups in \cref{les:ZZF:4kslices} are 2-torsion and the middle three terms become the exact sequence
\begin{equation} \label{ses:ZZF:4kslices}
\begin{tikzcd}[row sep=small]
    \ul\pi_{-n}( \Si{-k\rho} H_\K\ulZ) \arrow[r, hook] &
    \ul\pi_{-n}( \Si{-k\rho} H_\K\ulF) \arrow[r, twoheadrightarrow] &
    \ul\pi_{-n-1}( \Si{-k\rho} H_\K\ulZ).
\end{tikzcd}
\end{equation}

When \(n=4k\), the left four terms of \cref{les:ZZF:4kslices} become the exact sequence
\begin{equation*} 
\begin{tikzcd}[row sep=small]
    \ulZ^* \arrow[r, "2"] &
    \ulZ^* \arrow[r, twoheadrightarrow, "2"] &
    \ul\pi_{-n-1}( \Si{-k\rho} H_\K\ulZ) \arrow[r] & \ul 0.
\end{tikzcd}
\end{equation*}

Consequently, the \((4k)\)-slices of \(\Si{n} H_\K\ulF\) are
\begin{align*}
P^{4k}_{4k}( \Si{n} H_\K\ulF) &\simeq \left\lbrace
\begin{array}{ll}
   \Si{k\rho} H_\K\ulF^*  & n=4k \\
   \Si{k\rho} \ul E_{-n} &  n\leq 4k-1 
\end{array}\right.
\end{align*}

where \(\ul E_{-n}\) is the middle Mackey functor in \cref{ses:ZZF:4kslices}. This recovery is illustrated in \cref{tab-Kslices}.

\begin{table}[h] 
\caption[Comparison of {$\K$}-slices]{Comparison of $\K$-slices}
\label{tab-Kslices}
\begin{center}
	{\renewcommand{\arraystretch}{1.5}\setlength{\tabcolsep}{15pt}
		\begin{tabular}{|c|c|c|}
			\hline 
			  Slices of \(\Si{7} H\ulZ\) & Slices of \(\Si{7} H\ulF\) &  Slices of \(\Si{8} H\ulZ\)	\\ \hline
			  & \(P^{16}_{16} = \Si{4} H\ulg\) & \(P^{16}_{16} = \Si{4} H\ulg\) \\
			  \(P^{12}_{12} = \Si{3} H\ulg\) & \(P^{12}_{12} = \Si{3} H\ulg^3\) & \(P^{12}_{12} = \Si{3} H\ulg^2\) \\
			   &\(P^{10}_{10} = \Si{\rho+3} H\phi^*_{LDR} \ulF\) & \(P^{10}_{10} = \Si{\rho+3} H\phi^*_{LDR} \ulF\) \\
			  \(P^{8}_{8} = \Si{\rho+2} H\ul m\) & \(P^{8}_{8} = \Si{\rho+2} H\ul m\) & \(P^8_8 = \Si{\rho+4} H\ulZ\) \\
			  \(P^7_7 = \Si{\rho+3} H\ulZ\) & \(P^7_7 = \Si{\rho+3} H\ulF\) & \\ \hline
	\end{tabular} }
\end{center}
\end{table}

Except for the \(n\)-slice, all slices of \(\Si{7} H_\K\ulF\) are recovered from the slices of \(\Si{7} H_\K\ulZ\) and \(\Si{8} H_\K\ulZ\). It is not always the case, however, that the \((4k+2)\)-slices are recovered. For example, \(\Si{10} H_\K\ulF\) has a 14-slice (\cite[Example 6.14]{GY}), but neither \(\Si{10} H_\K\ulZ\) nor \(\Si{11} H_\K\ulZ\) have 14-slices.

%%%%%%%%%%%%%%%%%%%%%%%%%%%%%%%%%%%%%
\section{Homotopy Mackey Functor Computations} \label{sec:htpycomp}
%%%%%%%%%%%%%%%%%%%%%%%%%%%%%%%%%%%%%

Here we compute the homotopy Mackey functors of the slices of \(\Si{\pm n}H\ulZ\).

\begin{prop} \label{htpy:krho:Z}
For \(k\geq 1\), the nontrivial homotopy Mackey functors of \(\Si{k\rho} H\ulZ\) are
	\begin{align*}
	\ul\pi_n ( \Si{k\rho} H\ulZ ) = \left\lbrace \begin{array}{ll}
	    \ulZ & n=4k \\
		\ul{mg}  & n=4k-2 \\
		\ulg^{\frac{1}{2}(4k-n-1)}  & \begin{aligned}
		     n\in [2k,4k-3], \\ \text{ } n\equiv 1\pmod 2
		\end{aligned}			 \\
		\ulg^{\frac{1}{2}(4k-n+2)-3} \oplus \phi^*_{LDR} \ulF  & \begin{aligned}
		     n\in[2k,4k-3], \\ \text{ } n\equiv 0\pmod 2
		\end{aligned} \\
		\ulg^{n-k+1} & n\in[k,2k-1]
	\end{array}\right.
\end{align*}
\end{prop}
\begin{proof}
For \(n\in[k,4k-2]\), this is a restatement of \cref{-4kslices:K:Z}. For \(n=4k\), the result follows from \cref{-nslice:K:Z} and repeated application of \cref{slice:reduction:ZZ^*}.
\end{proof}

\begin{prop} \label{htpy:-krho:Z}
For \(k\geq 1\), the nontrivial homotopy Mackey functors of \(\Si{-k\rho} H\ulZ\) are
	\begin{align*}
	\ul\pi_{-n} ( \Si{-k\rho} H\ulZ ) = \left\lbrace \begin{array}{ll}
	    \ulZ^* & n=4k \\
		\ul{mg}^*  & n=4k-1 \\
		\ulg^{\frac{1}{2}(4k-n)}  & \begin{aligned}
		     n\in [2k+4,4k-3], \\ \text{ } n\equiv 0\pmod 2
		\end{aligned} \\
		\ulg^{\frac{1}{2}(4k-n+3)-3} \oplus \phi^*_{LDR} \ulF^* & \begin{aligned}
		     n\in[2k+3,4k-2], \\ \text{ } n\equiv 1\pmod 2
		\end{aligned}	 \\
	    \ulg^{n-k-3} & n\in[k+4,2k+2]
	\end{array}\right.
\end{align*}
\end{prop}
\begin{proof}
For \(n\in[k+4,4k-1]\), this is a restatement of \cref{slices:K:4k:Z}. For \(n=4k\), the result follows from \cref{slices:K:n:Z} and repeated application of \cref{slice:reduction:K:n:Z}.
\end{proof}

\begin{prop} \label{htpy:krho:Z*}
For \(k\geq 1\), the nontrivial homotopy Mackey functors of \(\Si{k\rho} H\ulZ^*\) are
	\begin{align*}
	\ul\pi_n ( \Si{k\rho} H\ulZ^* ) = \left\lbrace \begin{array}{ll}
	    \ulZ & n=4k \\
		\ul{mg}  & n=4k-2 \\
		\ulg^{\frac{1}{2}(4k-n-1)}  & \begin{aligned}
		     n\in [2k+2,4k-3], \\ \text{ } n\equiv 1\pmod 2 
		\end{aligned}	\\
		\ulg^{\frac{1}{2}(4k-n-2)-3} \oplus \phi^*_{LDR} \ulF  & \begin{aligned}
		     n\in[2k+2,4k-3], \\ \text{ } n\equiv 0\pmod 2
		\end{aligned}	 \\
		\ulg^{n-k+2} & n\in[k+3,2k+1]
	\end{array}\right.
\end{align*}
\end{prop}
\begin{proof}
This follows from the equivalence \(\Si{k\rho} H\ulZ^* \simeq \Si{(k-1)\rho+4} H\ulZ\) and \cref{htpy:krho:Z}.
\end{proof}

\begin{prop} \label{htpy:krho:Z(2,1)*}
For \(k\geq 1\), the nontrivial homotopy Mackey functors of \(\Si{k\rho} H\ulZ(2,1)^*\) are
\begin{align*}
	\ul\pi_n ( \Si{k\rho} H\ulZ(2,1)^* ) = \left\lbrace \begin{array}{ll}
	\ul\pi_n(\Si{k\rho} H\ulZ^*) & n\in[k+3,4k] \\
	\ulg & n=k
	\end{array}\right.
\end{align*}
\end{prop}
\begin{proof}
The exact sequence \(\ulZ^*\rightarrow \ulZ(2,1)^* \rightarrow \ulg\) and corresponding cofiber sequence \(\Si{k\rho} H\ulZ^* \rightarrow \Si{k\rho} H\ulZ(2,1)^* \rightarrow \Si{k} H\ulg\) provide us with a long exact sequence in homotopy. We then have that the homotopy of \(\Si{k\rho} H\ulZ(2,1)^*\) is the homotopy of \(\Si{k\rho} H\ulZ^*\) with an additional \(\ulg\) in degree \(k\).
\end{proof}

\begin{prop} \label{htpy:krho:mg*}
For \(k\geq 2\), the nontrivial homotopy Mackey functors of \(\Si{k\rho} H\ul{mg}^*\) are
\begin{align*}
    \ul\pi_n ( \Si{k\rho} H\ul{mg}^* ) = \left\lbrace \begin{array}{ll}
    \phi^*_{LDR} \ulF & n=2k\\
	\ulg^{3}  & n\in [k+2,2k-1] \\
	\ulg & n=k+1
    \end{array}\right.
\end{align*}
\end{prop}
\begin{proof}
We have the equivalence \(\Si{k\rho} H\ul{mg}^* \simeq \Si{(k-1)\rho+2} H\ul m\). The result then follows from \cite[Proposition 7.3]{GY}.
\end{proof}

\begin{prop} \label{htpy:-krho:mg}
For \(k\geq 2\), the nontrivial homotopy Mackey functors of \(\Si{-k\rho} H\ul{mg}\) are
\begin{align*}
	\ul\pi_{-n} ( \Si{-k\rho} H\ul{mg} ) = \left\lbrace \begin{array}{ll}
	\phi^*_{LDR} \ulF^* & n=2k \\
	\ulg^{3}  & n\in [k+2,2k-1] \\
	\ulg & n=k+1
	\end{array}\right.
\end{align*}
\end{prop}
\begin{proof}
The result follows by taking the Brown-Comenetz dual of each Mackey functor in \cref{htpy:krho:mg*}.
\end{proof}

\begin{prop} \label{htpy:-krho:m}
For \(k\geq 2\), the nontrivial homotopy Mackey functors of \(\Si{-k\rho} H\ul m\) are
\begin{align*}
	\ul\pi_{-n} ( \Si{-k\rho} H\ul{m} ) = \left\lbrace \begin{array}{ll}
	\phi^*_{LDR} \ulF^* & n=2k \\
	\ulg^{3}  & n\in [k+2,2k-1] \\
	\ulg^2 & n=k+1
	\end{array}\right.
\end{align*}
\end{prop}
\begin{proof}
First, take the Brown-Comenetz dual of the Mackey functors in \cite[Proposition 7.4]{GY}. The result then follows from the equivalence \(\Si{-\rho} H\ul m \simeq \Si{-2} H\ul{mg}^*\).
\end{proof}

\begin{prop} \label{htpy:krho:phiF*}
We have the equivalences
\begin{align*}
\Si{k\rho} H\phi^*_{LDR}\ulF^* \simeq \left\lbrace \begin{array}{ll}
	    \Si{2} H\phi^*_{LDR}\ul f & k=1 \\
	    \Si{4} H\phi^*_{LDR}\ulF & k=2
	\end{array}\right.
\end{align*}
Then for \(k\geq 3\), the nontrivial homotopy Mackey functors of \(\Si{k\rho} H\phi^*_{LDR} \ulF^*\) are
\begin{align*}
	\ul\pi_{n} ( \Si{k\rho} H\phi^*_{LDR} \ulF^* ) = \left\lbrace \begin{array}{ll}
	\phi^*_{LDR} \ulF & n=2k\\
	\ulg^{3}  & n\in [k+2,2k-1]
	\end{array}\right.
\end{align*}
\end{prop}
\begin{proof}
This is the pullback over \(L\), \(D\), and \(R\) of the Mackey functors in \cite[Proposition 3.6]{GY}.
\end{proof}

\begin{prop} \label{htpy:-krho:phiF}
We have the equivalences
\begin{align*}
\Si{-k\rho} H\phi^*_{LDR}\ulF \simeq \left\lbrace \begin{array}{ll}
	    \Si{-2} H\phi^*_{LDR}\ul f & k=1 \\
	    \Si{-4} H\phi^*_{LDR}\ulF^* & k=2
	\end{array}\right.
\end{align*}
Then for \(k\geq 3\), the nontrivial homotopy Mackey functors of \(\Si{-k\rho} H\phi^*_{LDR} \ulF\) are
\begin{align*}
	\ul\pi_{-n} ( \Si{-k\rho} H\phi^*_{LDR} \ulF ) = \left\lbrace \begin{array}{ll}
	    \phi^*_{LDR} \ulF^* & n=2k\\
		\ulg^{3}  & n\in [k+2,2k-1]
	\end{array}\right.
\end{align*}
\end{prop}
\begin{proof}
This is the pullback over \(L\), \(D\), and \(R\) of the Mackey functors in \cite[Proposition 3.7]{GY}.
\end{proof}

%%%%%%%%%%%%%%%%%%%%%%%%%%%%%%%%%%%%%
\section{Spectral Sequences} \label{sec:SSS}
%%%%%%%%%%%%%%%%%%%%%%%%%%%%%%%%%%%%%

\sseqset{
Fclass/.sseq style = {fill = black, diamond, draw, inner sep = 0.1ex},
classes= {fill, inner sep = 0pt, minimum size = 0.22em},
class labels={below=2pt}, 
differentials=black,
class pattern=linear, 
class placement transform = { rotate = 225, scale = 1.2 },
%classes={ tooltip = {(\xcoord,\ycoord)} }, 
run off differentials = ->, 
grid = go, 
grid color = gray!30,%!white
% \ulZ
Zclass/.sseq style={fill=black, rectangle, draw,inner sep=0.4ex},
% \ulZ*
Zstarclass/.sseq style={fill=none, rectangle, draw,inner sep=0.4ex},
% \ulg^n
Z4class/.sseq style={fill=none, circle, draw,inner sep=0.28ex},
Z4classSource/.sseq style={circle, inner sep = 0.28ex, draw, fill = red!70},
Z4classTarget/.sseq style={circle, inner sep = 0.28ex, draw, fill = black!40},
% \phi^*_{LDR} \ulF^*
FstarLDRclass/.sseq style = {fill = none, regular polygon, sides=5, draw, inner sep = 0.4ex},
% \phi^*_{LDR} \ulF^*
FLDRclass/.sseq style = {fill = black, regular polygon, regular polygon sides=5, draw, inner sep = 0.4ex},
%  \ul mg^*
mgstarclass/.sseq style = {fill = none, trapezium, draw, inner sep = 0.25ex},
%  \ul mg
mgclass/.sseq style = {fill = black, trapezium, draw, inner sep = 0.25ex},
% \ul m^*
mstarclass/.sseq style = {fill = none, regular polygon, regular polygon sides=3, draw, inner sep = 0.4ex},
}

\begin{sseqdata}[name=SI-1K,
y range={-10}{0}, 
x range={-10}{0}, 
y tick step = 2,
x tick step = 2,
struct lines = blue,
%axes type = center,
%axes gap = -0.09cm,
%x tick gap = -0.2cm,
%y tick gap = -0.21cm,
%    Adams grading,
%title=Page \page,
classes= {fill, inner sep = 0pt, minimum size = 0.2em},
class labels={below=2pt}, 
differentials=black,
class pattern=linear, 
class placement transform = { rotate = 135, scale = 2 },
%classes={ tooltip = {(\xcoord,\ycoord)} }, 
xscale=0.6,
yscale = 0.6,
run off differentials = ->, 
grid = go, 
grid color = gray!30%!white
]

\class[Zstarclass, "\phantom{*}" {inside,font=\tiny}](-1,-1)
\class[mstarclass](-1,-2)
\class[Z4class,"1" {inside,font=\tiny}](-1,-3)

\structline[dashed](-1,-3)(-1,-2)
\structline[dashed](-1,-2)(-1,-1)

\draw[fill=white](-9.9,-1.1) rectangle (-6.4,0.0);
\node at (-8.1,-0.6) {\Si{-1}H_{\K}\ulZ};

\end{sseqdata}

\begin{sseqdata}[name=SI-3K,
y range={-10}{0}, 
x range={-10}{0}, 
y tick step = 2,
x tick step = 2,
struct lines = blue,
%    Adams grading,
%title=Page \page,
classes= {fill, inner sep = 0pt, minimum size = 0.2em},
class labels={below=2pt}, 
differentials=black,
class pattern=linear, 
class placement transform = { rotate = 135, scale = 2 },
%classes={ tooltip = {(\xcoord,\ycoord)} }, 
xscale=0.6,
yscale = 0.6,
run off differentials = ->, 
grid = go, 
grid color = gray!30%!white
]

% not g
\class[Zstarclass, "\phantom{*}" {inside,font=\tiny}](-3,0)
\class[FstarLDRclass](-3,-3)	

% g
\class[Z4class,"1" {inside,font=\tiny}](-2,-6)
\class[Z4class,"2" {inside,font=\tiny}](-3,-9)

\d2(-2,-6)(-3,-3)	

\structline[dashed](-3,-9)(-3,-3)
\structline[dashed](-3,-3)(-3,0)

\draw[fill=white](-9.9,-1.1) rectangle (-6.4,0.0);
\node at (-8.1,-0.6) {\Si{-3}H_{\K}\ulZ};

\end{sseqdata}

\begin{sseqdata}[name=SI-5K,
y range={-16}{0}, 
x range={-10}{0}, 
y tick step = 2,
x tick step = 2,
struct lines = blue,
%    Adams grading,
%title=Page \page,
classes= {fill, inner sep = 0pt, minimum size = 0.2em},
class labels={below=2pt}, 
differentials=black,
class pattern=linear, 
class placement transform = { rotate = 135, scale = 2 },
%classes={ tooltip = {(\xcoord,\ycoord)} }, 
xscale=0.6,
yscale = 0.6,
run off differentials = ->, 
grid = go, 
grid color = gray!30%!white
]

% not g
\class[Zstarclass, "\phantom{*}" {inside,font=\tiny}](-5,0)
\class[mgstarclass, "\phantom{*}" {inside,font=\tiny}](-4,-1)
\class[FstarLDRclass](-3,-3)
\class[FstarLDRclass](-5,-5)	

% g
\foreach \x in {3,...,1} {
	\class[Z4class, "\x" {inside,font=\tiny}](-6+\x,-18+3*\x)
}
\class[Z4class,"1" {inside,font=\tiny}](-3,-2)
\class[Z4class,"1" {inside,font=\tiny}](-2,-4)
\class[Z4class,"1" {inside,font=\tiny}](-2,-6)
\class[Z4class,"3" {inside,font=\tiny}](-4,-6)

\d2(-2,-4)(-3,-2)	
\d2(-3,-3)(-4,-1)
\replacesource
\d3(-2,-6)(-3,-3)
\d3(-3,-9)(-4,-6)
\d7(-4,-12)(-5,-5)

\structline[dashed](-5,-15)(-5,-5)
\structline[dashed](-5,-5)(-5,0)

\draw[fill=white](-9.9,-1.1) rectangle (-6.4,0.0);
\node at (-8.1,-0.6) {\Si{-5}H_{\K}\ulZ};

\end{sseqdata}

\begin{sseqdata}[name=SI-7K,
y range={-28}{0}, 
x range={-9}{0}, 
y tick step = 2,
x tick step = 2,
struct lines = blue,
%    Adams grading,
%title=Page \page,
classes= {fill, inner sep = 0pt, minimum size = 0.2em},
class labels={below=2pt}, 
differentials=black,
class pattern=linear, 
class placement transform = { rotate = 135, scale = 2 },
%classes={ tooltip = {(\xcoord,\ycoord)} }, 
xscale=0.6,
yscale = 0.6,
run off differentials = ->, 
grid = go, 
grid color = gray!30%!white
]

% not g
\class[Zstarclass, "\phantom{*}" {inside,font=\tiny}](-7,0)
\class[mgstarclass, "\phantom{*}" {inside,font=\tiny}](-6,-1)
\class[FstarLDRclass](-5,-5)
\class[FstarLDRclass](-7,-7)	

% g
\class[Z4class,"1" {inside,font=\tiny}](-5,-2)

% diagonal g's
\foreach \x in {1,...,4} {
	\class[Z4class, "\x" {inside,font=\tiny}](-8+\x,-24+3*\x)
}
\foreach \x in {2,...,3} {
\class[Z4class, "\x" {inside,font=\tiny}](-1-\x,-15+3*\x)
}
\foreach \x in {0,...,1} {
	\class[Z4class, "3" {inside,font=\tiny}](-6+\x,-8-\x)
}

\d3(-3,-9)(-4,-6)
\replacetarget
\d3(-4,-12)(-5,-9)
\replacesource
\d4(-5,-5)(-6,-1)
\replacesource
\d4(-4,-6)(-5,-2)
\d7(-4,-12)(-5,-5)
\d7(-5,-15)(-6,-8)
\d11(-6,-18)(-7,-7)

\structline[dashed](-7,-21)(-7,-7)
\structline[dashed](-7,-7)(-7,0)

\draw[fill=white](-3.9,-1.1) rectangle (-0.4,0.0);
\node at (-2.1,-0.6) {\Si{-7}H_{\K}\ulZ};

\end{sseqdata}

\begin{sseqdata}[name=SI-9K,
y range={-28}{0}, 
x range={-11}{0}, 
y tick step = 2,
x tick step = 2,
struct lines = blue,
%    Adams grading,
%title=Page \page,
classes= {fill, inner sep = 0pt, minimum size = 0.2em},
class labels={below=2pt}, 
differentials=black,
class pattern=linear, 
class placement transform = { rotate = 135, scale = 2 },
%classes={ tooltip = {(\xcoord,\ycoord)} }, 
xscale=0.6,
yscale = 0.6,
run off differentials = ->, 
grid = go, 
grid color = gray!30%!white
]

\class[Zstarclass, "\phantom{*}" {inside,font=\tiny}](-9,0)
\class[mgstarclass, "\phantom{*}" {inside,font=\tiny}](-8,-1)

% phiF
\class[FstarLDRclass](-6,-3)
\foreach \x in {0,1,2} {
\class[FstarLDRclass](-9+2*\x,-9+2*\x)
}

% miscellaneous g's
\class[Z4class,"1" {inside,font=\tiny}](-3,-7)	
\class[Z4class,"1" {inside,font=\tiny}](-3,-9)
\class[Z4class,"3" {inside,font=\tiny}](-4,-6)	
\class[Z4class,"3" {inside,font=\tiny}](-4,-12)
\class[Z4class,"1" {inside,font=\tiny}](-4,-5)	
\class[Z4class,"2" {inside,font=\tiny}](-5,-4)
\class[Z4class,"1" {inside,font=\tiny}](-7,-2)

% diagonal g's
\foreach \x in {1,...,5} {
	\class[Z4class, "\x" {inside,font=\tiny}](-10+\x,-30+3*\x)
}
\foreach \x in {0,...,2} {
\class[Z4class, "3" {inside,font=\tiny}](-6-\x,-12+\x)
}
\foreach \x in {0,...,1} {
	\class[Z4class, "3" {inside,font=\tiny}](-5-\x,-9+\x)
}

% differentials
\d2(-4,-6)(-5,-4) 
\replacesource
\d3(-4,-9-3*1)(-4-1,-6-3*1)
\d2(-3-2,-7+2)(-4-2,-5+2)
\d3(-3-2,-9-3*2)(-4-2,-6-3*2)
\replacesource
\foreach \x in {0} {
%\ifthenelse{\equal{\x}{0}}{
%\d2(-3-\x,-7+\x)(-4-\x,-5+\x) 
%\replacesource}{
\d2(-3-\x,-7+\x)(-4-\x,-5+\x)
%\replacesource
\d3(-3-\x,-9-3*\x)(-4-\x,-6-3*\x)
%\replacesource
%	}	
}
\foreach \x in {0,1} {
\d6(-6-\x,-8+\x)(-7-\x,-2+\x)
\replacesource
\d7(-5-\x,-15-3*\x)(-6-\x,-8-3*\x)
\replacesource
}
\d11(-6,-18)(-7,-7)
\d11(-7,-21)(-8,-10)
\d15(-8,-24)(-9,-9)

\structline[dashed](-9,-27)(-9,-9)
\structline[dashed](-9,-9)(-9,0)

\draw[fill=white](-3.9,-1.1) rectangle (-0.4,0.0);
\node at (-2.1,-0.6) {\Si{-9}H_{\K}\ulZ};

\end{sseqdata}

\begin{sseqdata}[name=SI6K,
y range={0}{10}, 
x range={0}{10}, 
y tick step = 2,
x tick step = 2,
struct lines = blue,
%    Adams grading,
%title=Page \page,
classes= {fill, inner sep = 0pt, minimum size = 0.2em},
class labels={below=2pt}, 
differentials=black,
class pattern=linear, 
class placement transform = { rotate = 135, scale = 2 },
%classes={ tooltip = {(\xcoord,\ycoord)} }, 
xscale=0.6,
yscale = 0.6,
run off differentials = ->, 
grid = go, 
grid color = gray!30%!white
]

% not g
\class[Zclass, "\phantom{*}" {inside,font=\tiny}](6,0)

% g
\class[Z4class,"1" {inside,font=\tiny}](2,6)
\class[Z4class,"1" {inside,font=\tiny}](3,3)	

\d3(3,3)(2,6)

\draw[fill=white](7.5,8.5) rectangle (10.3,9.6);
\node at (8.9,9.0) {\Si{6}H_{\K}\ulZ};

\end{sseqdata}

\begin{sseqdata}[name=SI7K,
y range={0}{10}, 
x range={0}{10}, 
y tick step = 2,
x tick step = 2,
struct lines = blue,
%    Adams grading,
%title=Page \page,
classes= {fill, inner sep = 0pt, minimum size = 0.2em},
class labels={below=2pt}, 
differentials=black,
class pattern=linear, 
class placement transform = { rotate = 135, scale = 2 },
%classes={ tooltip = {(\xcoord,\ycoord)} }, 
xscale=0.6,
yscale = 0.6,
run off differentials = ->, 
grid = go, 
grid color = gray!30%!white
]

% not g
\class[Zclass, "\phantom{*}" {inside,font=\tiny}](7,0)
\class[mgclass, "\phantom{*}" {inside,font=\tiny}](5,2)
\class[FLDRclass](4,4)	

% g
\class[Z4class,"1" {inside,font=\tiny}](4,3)
\class[Z4class,"1" {inside,font=\tiny}](3,5)
\class[Z4class,"1" {inside,font=\tiny}](3,9)

\d2(5,2)(4,4)
\replacetarget
\d2(4,3)(3,5)
\d5(4,4)(3,9)

\draw[fill=white](7.5,8.5) rectangle (10.3,9.6);
\node at (8.9,9) {\Si{7}H_{\K}\ulZ};

\end{sseqdata}

\begin{sseqdata}[name=SI8K,
y range={0}{13}, 
x range={0}{10}, 
y tick step = 2,
x tick step = 2,
struct lines = blue,
%    Adams grading,
%title=Page \page,
classes= {fill, inner sep = 0pt, minimum size = 0.2em},
class labels={below=2pt}, 
differentials=black,
class pattern=linear, 
class placement transform = { rotate = 135, scale = 2 },
%classes={ tooltip = {(\xcoord,\ycoord)} }, 
xscale=0.6,
yscale = 0.6,
run off differentials = ->, 
grid = go, 
grid color = gray!30%!white
]

% not g
\class[Zclass, "\phantom{*}" {inside,font=\tiny}](8,0)
\class[mgclass, "\phantom{*}" {inside,font=\tiny}](6,2)
\class[FLDRclass](5,5)	

% g
\class[Z4class,"3" {inside,font=\tiny}](4,6)
\class[Z4class,"2" {inside,font=\tiny}](3,9)
\class[Z4class,"1" {inside,font=\tiny}](5,3)	%\replacetarget
\d3(5,3)(4,6)
\d3(4,6)(3,9)
%\class[Z4class,"3" {inside,font=\tiny}](4,6)

\class[Z4class,"1" {inside,font=\tiny}](4,12)

\d3(6,2)(5,5)
\replacetarget
\d7(5,5)(4,12)

\draw[fill=white](7.5,9.5) rectangle (10.3,10.6);
\node at (8.9,10) {\Si{8}H_{\K}\ulZ};

\end{sseqdata}

\begin{sseqdata}[name=SI10K,
y range={0}{19}, 
x range={0}{12}, 
y tick step = 2,
x tick step = 2,
struct lines = blue,
%    Adams grading,
%title=Page \page,
classes= {fill, inner sep = 0pt, minimum size = 0.2em},
class labels={below=2pt}, 
differentials=black,
class pattern=linear, 
class placement transform = { rotate = 135, scale = 2 },
%classes={ tooltip = {(\xcoord,\ycoord)} }, 
xscale=0.6,
yscale = 0.6,
run off differentials = ->, 
grid = go, 
grid color = gray!30%!white
]

% not g
\class[Zclass, "\phantom{*}" {inside,font=\tiny}](10,0)
\class[mgclass, "\phantom{*}" {inside,font=\tiny}](8,2)
\class[FLDRclass](7,7)	

% g
\class[Z4class,"1" {inside,font=\tiny}](6,18)
\class[Z4class,"2" {inside,font=\tiny}](5,15)
\class[Z4class,"1" {inside,font=\tiny}](3,9)
\class[Z4class,"3" {inside,font=\tiny}](6,8)
\class[Z4class,"3" {inside,font=\tiny}](4,12)
\class[Z4class,"3" {inside,font=\tiny}](5,9)
\class[Z4class,"1" {inside,font=\tiny}](4,6)
\class[Z4class,"1" {inside,font=\tiny}](7,3)

\d3(4,6)(3,9)
\d3(5,9)(4,12)
\d5(7,3)(6,8)
\replacetarget

\d7(6,8)(5,15)
%\d5(6,10)(5,15)

\d5(8,2)(7,7)
\replacetarget
\d11(7,7)(6,18)
%\replacetarget

%\d9(7,9)(6,18)
%\d13(8,8)(7,21)

\draw[fill=white](9.5,18.5) rectangle (12.3,19.6);
\node at (10.9,19.0) {\Si{10}H_{\K}\ulZ};

\end{sseqdata}

\begin{sseqdata}[name=SI11K,
y range={0}{21}, 
x range={0}{13}, 
y tick step = 2,
x tick step = 2,
struct lines = blue,
%    Adams grading,
%title=Page \page,
classes= {fill, inner sep = 0pt, minimum size = 0.2em},
class labels={below=2pt}, 
differentials=black,
class pattern=linear, 
class placement transform = { rotate = 135, scale = 2 },
%classes={ tooltip = {(\xcoord,\ycoord)} }, 
xscale=0.6,
yscale = 0.6,
run off differentials = ->, 
grid = go, 
grid color = gray!30%!white
]

% not g
\class[Zclass, "\phantom{*}" {inside,font=\tiny}](11,0)
\class[mgclass, "\phantom{*}" {inside,font=\tiny}](9,2)
\class[FLDRclass](7,4)	
\class[FLDRclass](8,8)	
\class[FLDRclass](6,6)	

% g
\class[Z4class,"1" {inside,font=\tiny}](7,21)
\class[Z4class,"2" {inside,font=\tiny}](6,18)
\class[Z4class,"3" {inside,font=\tiny}](5,15)
\class[Z4class,"3" {inside,font=\tiny}](7,9)
\class[Z4class,"3" {inside,font=\tiny}](6,10)
\class[Z4class,"1" {inside,font=\tiny}](4,12)
\class[Z4class,"3" {inside,font=\tiny}](5,7)
\class[Z4class,"1" {inside,font=\tiny}](4,8)
\class[Z4class,"1" {inside,font=\tiny}](8,3)
\class[Z4class,"2" {inside,font=\tiny}](6,5)
\class[Z4class,"1" {inside,font=\tiny}](5,6)

\d2(5,6)(4,8)
\d2(7,4)(6,6)
\d2(6,5)(5,7)
\replacetarget

\d5(5,7)(4,12)
\d5(6,10)(5,15)

\d6(8,3)(7,9)
\replacetarget
\d6(9,2)(8,8)
\replacetarget

\d9(7,9)(6,18)
\d13(8,8)(7,21)

\draw[fill=white](10.5,20.5) rectangle (13.3,21.6);
\node at (11.9,21.0) {\Si{11}H_{\K}\ulZ};

\end{sseqdata}

\begin{sseqdata}[name=SI12K,
y range={0}{25}, 
x range={0}{12}, 
y tick step = 2,
x tick step = 2,
struct lines = blue,
%    Adams grading,
%title=Page \page,
classes= {fill, inner sep = 0pt, minimum size = 0.2em},
class labels={below=2pt}, 
differentials=black,
class pattern=linear, 
class placement transform = { rotate = 135, scale = 2 },
%classes={ tooltip = {(\xcoord,\ycoord)} }, 
xscale=0.6,
yscale = 0.6,
run off differentials = ->, 
grid = go, 
grid color = gray!30%!white
]

% not g
\class[Zclass, "\phantom{*}" {inside,font=\tiny}](12,0)
\class[mgclass, "\phantom{*}" {inside,font=\tiny}](10,2)
\class[FLDRclass](8,4)	
\class[FLDRclass](9,9)	
\class[FLDRclass](7,7)	

% g
\foreach \x in {1,...,4} {
	\class[Z4class, "\x" {inside,font=\tiny}](9-\x,27-3*\x)
}
\foreach \x in {0,...,2} {
	\class[Z4class, "3" {inside,font=\tiny}](8-\x,10+\x)
}
\class[Z4class,"1" {inside,font=\tiny}](9,3)
\class[Z4class,"2" {inside,font=\tiny}](7,5)
\class[Z4class,"1" {inside,font=\tiny}](6,6)
\class[Z4class,"2" {inside,font=\tiny}](4,12)
\foreach \x in {0,...,1} {
	\class[Z4class, "3" {inside,font=\tiny}](6-\x,8+\x)
}

% differentials
\d3(6,12)(5,15)
\replacetarget
\d6(7,11)(6,18)

\d7(9,3)(8,10)
\replacetarget
\d12(8,10)(7,21)
\d7(10,2)(9,9)
\replacetarget
\d15(9,9)(8,24)

\d3(6,6)(5,9)
\replacesource
\d3(5,9)(4,12)
\foreach \x in {1,...,2} {
    \d3(6+\x,6-\x)(5+\x,9-\x)
    \replacetarget 
}
\d7(6,8)(5,15)

\draw[fill=white](1.5,23.5) rectangle (4.3,24.6);
\node at (2.9,24.0) {\Si{12}H_{\K}\ulZ};

\end{sseqdata}

\begin{sseqdata}[name=SI14K,
y range={0}{31}, 
x range={0}{14}, 
y tick step = 2,
x tick step = 2,
struct lines = blue,
%    Adams grading,
%title=Page \page,
classes= {fill, inner sep = 0pt, minimum size = 0.2em},
class labels={below=2pt}, 
differentials=black,
class pattern=linear, 
class placement transform = { rotate = 135, scale = 2 },
%classes={ tooltip = {(\xcoord,\ycoord)} }, 
xscale=0.6,
yscale = 0.6,
run off differentials = ->, 
grid = go, 
grid color = gray!30%!white
]

% not g
\class[Zclass, "\phantom{*}" {inside,font=\tiny}](14,0)
\class[mgclass, "\phantom{*}" {inside,font=\tiny}](12,2)
\class[FLDRclass](10,4)	
\class[FLDRclass](11,11)	
\class[FLDRclass](9,9)	

% g
\foreach \x in {1,...,5} {
	\class[Z4class, "\x" {inside,font=\tiny}](11-\x,33-3*\x)
}
\foreach \x in {0,...,3} {
	\class[Z4class, "3" {inside,font=\tiny}](10-\x,12+\x)
}
\class[Z4class,"1" {inside,font=\tiny}](11,3)
\class[Z4class,"2" {inside,font=\tiny}](9,5)
\class[Z4class,"1" {inside,font=\tiny}](8,6)
\class[Z4class,"1" {inside,font=\tiny}](4,12)
\class[Z4class,"3" {inside,font=\tiny}](5,15)
\class[Z4class,"1" {inside,font=\tiny}](5,9)
\foreach \x in {0,...,2} {
	\class[Z4class, "3" {inside,font=\tiny}](8-\x,10+\x)
}

% differentials
\d9(12,2)(11,11)
\replacetarget
\d19(11,11)(10,30)
\d9(11,3)(10,12)
\replacetarget
\d16(10,12)(9,27)

\foreach \x in {0,...,1} {
    \d3\x(5+\x,9+3*\x)(4+\x,12+3*\x)
    \replacetarget
}
\d3(7,15)(6,18)
\replacetarget

\d11(9,13)(8,24)
\d7(8,14)(7,21)
\replacetarget

\foreach \x in {2,...,2} {
    \d5\x(8+\x,6-\x)(7+\x,11-\x)
    \replacetarget
}
\d5(9,5)(8,10)
\replacetarget
\d5(8,6)(7,11)
\replacetarget
\d7(7,11)(6,18)
\d11(8,10)(7,21)

\draw[fill=white](1.5,29.5) rectangle (4.3,30.6);
\node at (2.9,30.0) {\Si{14}H_{\K}\ulZ};

\end{sseqdata}

\begin{sseqdata}[name=SI7KF,
y range={0}{13}, 
x range={0}{10}, 
y tick step = 2,
x tick step = 2,
struct lines = blue,
%    Adams grading,
%title=Page \page,
classes= {fill, inner sep = 0pt, minimum size = 0.2em},
class labels={below=2pt}, 
differentials=black,
class pattern=linear, 
class placement transform = { rotate = 135, scale = 2 },
%classes={ tooltip = {(\xcoord,\ycoord)} }, 
xscale=0.6,
yscale = 0.6,
run off differentials = ->, 
grid = go, 
grid color = gray!30%!white
]

% not g
\class[Fclass, "\phantom{*}" {inside,font=\tiny}](7,0)
\class[mgclass, "\phantom{*}" {inside,font=\tiny}](6,1)
\class[FLDRclass](5,2)
\class[FLDRclass](4,4)
\class[FLDRclass](5,5)

% g
\class[Z4class,"1" {inside,font=\tiny}](4,3)
\class[Z4class,"1" {inside,font=\tiny}](3,5)
\class[Z4class,"1" {inside,font=\tiny}](4,12)
\d2(5,2)(4,4)
\d2(4,3)(3,5)

\class[Z4class,"3" {inside,font=\tiny}](4,6)
\class[Z4class,"3" {inside,font=\tiny}](3,9)
\d3(4,6)(3,9)

\d4(6,1)(5,5)
\replacetarget
\d7(5,5)(4,12)

\draw[fill=white](7.5,9.5) rectangle (10.3,10.6);
\node at (8.9,10) {\Si{7}H_{\K}\ulF};

\end{sseqdata}

\begin{sseqdata}[name=SI10KF,
y range={0}{21}, 
x range={0}{12}, 
y tick step = 2,
x tick step = 2,
struct lines = blue,
%    Adams grading,
%title=Page \page,
classes= {fill, inner sep = 0pt, minimum size = 0.2em},
class labels={below=2pt}, 
differentials=black,
class pattern=linear, 
class placement transform = { rotate = 135, scale = 2 },
%classes={ tooltip = {(\xcoord,\ycoord)} }, 
xscale=0.6,
yscale = 0.6,
run off differentials = ->, 
grid = go, 
grid color = gray!30%!white
]

% not g
\class[Fclass, "\phantom{*}" {inside,font=\tiny}](10,0)
\class[mgclass, "\phantom{*}" {inside,font=\tiny}](9,1)
\class[FLDRclass](8,2)	

\foreach \x in {0,...,2} {
	\class[FLDRclass](6+\x,6+\x)
}
\class[FLDRclass](7,3)
\d7(9,1)(8,8)
\replacetarget
\d5(8,2)(7,7)
\replacesource
\d3(7,3)(6,6)
\replacesource

% g
\class[Z4class,"4" {inside,font=\tiny}](4,12)
\class[Z4class,"5" {inside,font=\tiny}](5,15)
\class[Z4class,"3" {inside,font=\tiny}](6,18)
\class[Z4class,"1" {inside,font=\tiny}](7,21)
\foreach \x in {0,...,2} {
	\class[Z4class, "3" {inside,font=\tiny}](5+\x,7+\x)
}
\class[Z4class,"3" {inside,font=\tiny}](5,9)
\class[Z4class,"3" {inside,font=\tiny}](6,10)
\class[Z4class,"2" {inside,font=\tiny}](6,4)
\class[Z4class,"1" {inside,font=\tiny}](8,2)
\d7(8,2)(7,9)
\replacetarget

\d3(6,4)(5,7)
\d3(5,9)(4,12)

\d5(7,3)(6,8)
\d9(7,9)(6,18)
\d5(6,10)(5,15)
\d13(8,8)(7,21)

\draw[fill=white](9.5,20.5) rectangle (12.3,21.6);
\node at (10.9,21.0) {\Si{10}H_{\K}\ulF};

\end{sseqdata}

The slice spectral sequence for \(\Si{-n} H\ulZ\) and \(\Si{n} H\ulZ\) must recover the homotopy Mackey functors of each spectrum, that is, we must be left with \(\ul\pi_{-n}(\Si{-n}H\ulZ) = \ul\pi_{n}(\Si{n}H\ulZ) = \ulZ \) and all other homotopy Mackey functors trivial. For most of the differentials, then, there is only one choice.

We use the indexing convention from \cite[Section 4.4.2]{HHR}. The Mackey functor \(\ul E_2^{t-n,t}\) is \(\ul\pi_n P^t_t(X)\). We also use the Adams convention, so that \(\ul\pi_n P^t_t(X)\) has coordinates \((n,n-t)\) and the differential,
\begin{align*}
d_r:\ul E_r^{s,t} \rightarrow \ul E_r^{s+r,t+r-1},
\end{align*}
points left one and up \(r\).

The symbols in \cref{SSS:mack} denote the Mackey functors in the slice spectral sequences shown.

\begin{table}[h] 
\caption[Symbols for {$\K$}-Mackey functors]{Symbols for $\K$-Mackey functors}
\label{SSS:mack}
\begin{center}
	{\renewcommand{\arraystretch}{1.5}
		\begin{tabular}{|c|c|c|}
			\hline 
			$\blacksquare=\ulZ$
			& $\bpent=\phi^*_{LDR} \ulF$
			& $\btrap=\ul{mg}$
			\\\hline
			$\square=\ulZ^*$
			&  $\pent=\phi_{LDR}^*\ulF^*$
			& $\trap = \ul {mg}^*$ 
			\\ \hline
			 $\ulF = \bdiamond$ & $\gcirc=\ulg^n$ & $\triangle = \ul m^*$\\\hline 
	\end{tabular} }
\end{center}
\end{table}

\begin{exmp} \label{SSS:-1HZ}
The slices for \(\Si{-1} H\ulZ\) are all a one-fold desuspension of
Eilenberg-MacLane spectra (\cref{tower:-1:K:Z}, \cref{fig:SSS:-1..-5:HZ}). Because each of these Mackey functors is in the same column, there are no differentials. Consequently, in the spectral sequence, we find a double extension:
\begin{equation} \label{double:extension}
\begin{tikzcd}[row sep=small]
    \ulZ^* \ar[dr] & & \\
    & \ulZ(2,1) \ar[dr] \ar[dl] & \\ 
    \ul m^* & & \ulZ \ar[dddll] \\
    & & \\ \\
    \ulg & & 
\end{tikzcd}
\end{equation}
\end{exmp}

\begin{exmp} \label{SSS:-5HZ}
In the spectral sequence for \(\Si{-5} H\ulZ\), \cref{fig:SSS:-1..-5:HZ}, because we can only be left with \(\ul\pi_{-5}(P^{-5}_{-5} \Si{-5} H\ulZ) \cong \ulZ\) and all differentials must go left one and up at least two, all differentials are forced. Once we have evaluated each differential, we once again find ourselves with the double extension in \cref{double:extension}.
\end{exmp}

\begin{exmp} \label{SSS:-9HZ}
In \cref{fig:SSS:-7..-9:HZ}, most of the differentials for \(\Si{-9} H\ulZ\) are again forced by the fact that only \(\ul\pi_{-9}(P^{-9}_{-9} \Si{-9} H\ulZ) \cong \ulZ\) can survive the spectral sequence. For example, we have two choices for a differential from \(\pi_{-8}P^{-32}_{-32} \Si{-9} H\ulZ \cong \ulg^2\). We find it must be
\begin{align*}
d_{15}: \ulg^2 \rightarrow \phi^*_{LDR}\ulF^* \cong  \ul\pi_{-9} (P^{-24}_{-24} \Si{-9} H\ulZ) 
\end{align*}
so that we are left with the extension in \cref{double:extension}. Indeed, we will always be left with this extension once all differentials have been evaluated. Similarly, for \(n\equiv 1\pmod 4\), we will always have a
\begin{align*}
d_2:\phi^*_{LDR} \ulF^*  \xrightarrow{\cong}{} \phi^*_{LDR} \ulF^*
\end{align*}
in the upper right corner.
\end{exmp}

In \cref{-4k-2slices:K:Z} we claim that for \(n\) odd,
\begin{align*}
    P^{-2n}_{-2n} (\Si{-n} H\ulZ) \simeq \Si{-(k+1)\rho +1} H\phi^*_{LDR} \ulF.
\end{align*}
We now prove this claim.

\begin{prop} \label{phiF:SSS}
Let \(n\) be odd and take \(\ul A\) to be one of the Mackey functors listed in \cref{tab-SShtpydegrees}. The only choice of \(\ul A\) where the homotopy of \(P^{-2n}_{-2n} (\Si{-n} H\ulZ) \simeq \Si{-(k+1)\rho+1} H\ul A\) fits into the spectral sequence for \(\Si{-n} H\ulZ\) is \(\phi^*_{LDR} \ulF\).
\end{prop}
\begin{proof}
Because \(\Si{-(k+1)\rho+1} H\ul A\) is a \((-4k-2)\)-slice, its \(\ul\pi_{-k-1}\) is located at \((-k-1,-3k-1)\). We argue that we cannot have a nonzero Mackey functor in this location.

\begin{table}[h] 
\caption[Homotopy Comparison]{Homotopy Comparison}
\label{tab-SShtpydegrees}
\begin{center}
	{\renewcommand{\arraystretch}{1.5}\setlength{\tabcolsep}{12pt}
		\begin{tabular}{|c|c|} \hline 
		$\ul A$ & $\ul\pi _{-k-1}(\Si{-(k+1)\rho +1} H\ul A)$ \\ \hline 
		$\phi^*_{LDR} \ul f$ & $\ulg^3$ \\
		$\ul m$ & $\ulg^2$ \\
		$\ul{mg}$ & $\ulg$ \\
		$\phi^*_{LDR} \ulF$ & $\ul 0$ \rule[-2.4ex]{0pt}{2.2ex} \\ \hline
	\end{tabular} }
\end{center}
\end{table}

In the spectral sequence for \(\Si{-n} H\ulZ\), all slices below level \(-2n\) are \(\Si{-k} H\ulg^{n-k+1}\) where \(4k\in[2n+1,4n]\). These Mackey functors, for \(4k\leq 2n+1\), lie on the line \(y=-3k\). Thus, for the Mackey functors in \cref{tab-SShtpydegrees}, the source of a differential hitting it must be \((-k,-3k)\). This is not possible.

We now argue that the Mackey functor located at \((-k-1,-3k-1)\) cannot be the source of a differential.The first value of \(n\) for which we must determine the \((-2n)\)-slice is \(n=7\). The spectral sequence where \(\ul A=\phi^*_{LDR} \ulF\) is shown in \cref{fig:SSS:-7..-9:HZ}. This spectral sequence leaves us with the appropriate homotopy for \(\Si{-7} H\ulZ\).

Note there is a copy of \(\ulg^4\) located at \((-4,-12)\). For any other choice of \(\ul A\), we would have a nontrivial \(\ul\pi_{-4}\) located at \((-4,-10)\). For a differential originating from \((-4,-10)\), there are two possible targets: \(\phi^*_{LDR} \ulF^*\) at \((-5,-5)\) and \(\ulg\) at \((-5,-2)\). However, these two Mackey functors must fit into the exact sequences
\begin{align*}
    \ulg^2 \rightarrow \ulg^3\rightarrow \ulg \qquad \text{ and } \qquad 
    \ulg^4 \rightarrow \phi^*_{LDR} \ulF^* \rightarrow \ul {mg}.
\end{align*}
Thus, we cannot have a nonzero Mackey functor at \((-4,-10)\).

We now consider the spectral sequence for \(\Si{-9} H\ulZ\), located in \cref{fig:SSS:-7..-9:HZ}. We again use \(\Si{-(l+1)\rho + 1} H\phi^*_{LDR} \ulF\) for the \((-2n)\)-slice. The resulting homotopy fits in the spectral sequence. For the other three choices of \(\ul A\) we would have a nontrivial \(\ul\pi_{-5}\) located at \((-5,-13)\). For a differential originating from \((-5,-13)\) there are two possible targets: \(\ulg^3\) at \((-6,-8)\) and \(\phi^*_{LDR}\ulF^*\) at \((-6,-3)\).

However, we have a \(d_2:\phi^*_{LDR}\ulF^* \twoheadrightarrow \phi^*_{LDR}\ulF^*\) and a \(d_7:\ulg^5\twoheadrightarrow \ulg^3\). Thus, there is no target for a differential from \((-5,-13)\). Consequently, the only choice of \(\ul A\) that works is \(\phi^*_{LDR} \ulF\).

There is a similar story for all odd \(n>9\). There will always be a \(d_2\) hitting the \(\phi^*_{LDR} \ulF^*\) located at \((-k-2,-n+k+2)\). The only other possible targets for a differential from \((-k-1,-3k-1)\) are then the \(\ulg^3\)'s resulting from the homotopy of the other \((-4j-2)\)-slices. All of these will be hit by a differential from the \(\ulg^{n-k+1}\) located at \((-k-1,-3k-3)\). Thus, we cannot have a nonzero Mackey functor located at \((-k-1,-3k-1)\). The only choice of \(\ul A\) which satisfies this requirement is \(\phi^*_{LDR} \ulF\).
\end{proof}

\begin{exmp} \label{SSS:6HZ}
For the positive, trivial suspensions of \(H\ulZ\), we find that \(\Si{6} H\ulZ\) has the first nontrivial slice tower. In \cref{fig:SSS:6..8:HZ}, we then see that there is only one possible differential. This \(d_3\) exists because we must be left with only \(\ul\pi_6 (\Si{6} H\ulZ) \cong \ulZ\).
\end{exmp}

\begin{exmp} \label{SSS:7HZ}
The spectral sequence for \(\Si{7} H\ulZ\), in \cref{fig:SSS:6..8:HZ}, is more interesting. Here we find the differentials \(d_2:\phi^*_{LDR}F \rightarrow \ul{mg}\) and \(d_5:\ul{mg}\rightarrow \ulg\). Indeed, we will always see a \(d_{n-7}:\phi^*_{LDR}\ulF \rightarrow \ul{mg}\) and \(d_{2n-9}:\ul{mg} \rightarrow \ulg\) on the right side of the spectral sequence for \(\Si{n} H_\K\ulZ\).
\end{exmp}

\begin{exmp} \label{SSS:7F}
Except for the homotopy of the \(n\)-slice of \(\Si{7} H_\K\ulZ, \quad \) \(\Si{7} H_\K\ulF\), and \(\Si{8} H_\K\ulZ\), the spectral sequences for \(\Si{7} H_\K\ulZ\) and \(\Si{8} H_\K\ulZ\) collapse to give the spectral sequence for \(\Si{7} H_\K\ulF\). We see in \cref{fig:SSS:6..8:HZ} the \(\ulg\) in \((3,9)\) and the \(\ulg^2\) in \((3,9)\) in the spectral sequences for \(\Si{7} H_\K\ulZ\) and \(\Si{8} H_\K\ulZ\), respectively, combine to give the \(\ulg^3\) in \((3,9)\) in the spectral sequence for \(\Si{7} H_\K\ulF\). Off the diagonal for the \(n\)-slice for \(\Si{7} H_\K\ulZ\) and \(\Si{8} H_\K\ulZ\) we have a single copy of \(\phi^*_{LDR}\ulF\). These provide the two copies of \(\phi^*_{LDR}\ulF\) off the diagonal for \(\Si{7} H_\K\ulF\).
\end{exmp}

\begin{exmp} \label{SSS:11HZ}
Now, in \cref{fig:SSS:10..11:HZ}, we have some choice of differentials in the spectral sequence for \(\Si{11} H\ulZ\). Once we consider that only \(\ul\pi_{11}(\Si{11} H\ulZ) \cong \ulZ\) can be left, there is only one choice of each differential that provides the desired result. Analogously to the spectral sequence for \(\Si{-n} H\ulZ\) where \(n\equiv 1\pmod 4\), we will always have a \(d_2:\phi^*_{LDR} \ulF \xrightarrow{\cong}{} \phi^*_{LDR} \ulF^*\) in the bottom left corner when \(n\equiv 3\pmod 4\).
\end{exmp}

\begin{exmp} \label{SSS:10F}
Again, with the exception of the homotopy of the \(n\)-slice, the spectral sequences for \(\Si{10} H_\K\ulZ\) and \(\Si{11} H_\K\ulZ\) collapse to give the spectral sequence for \(\Si{10} H_\K\ulF\) in \cref{fig:SSS:10..11:HZ}. As in \cref{SSS:7F}, the upper left diagonals in in the spectral sequences for \(\Si{10} H_\K\ulZ\) and \(\Si{11} H_\K\ulZ\) combine to even more copies of \(\ulg\) in the upper left diagonal in the spectral sequence for \(\Si{10} H_\K\ulF\). This will always be the case.
\end{exmp}

\clearpage

\begin{fig}\label{fig:SSS:-1..-5:HZ}
\begin{flushleft}
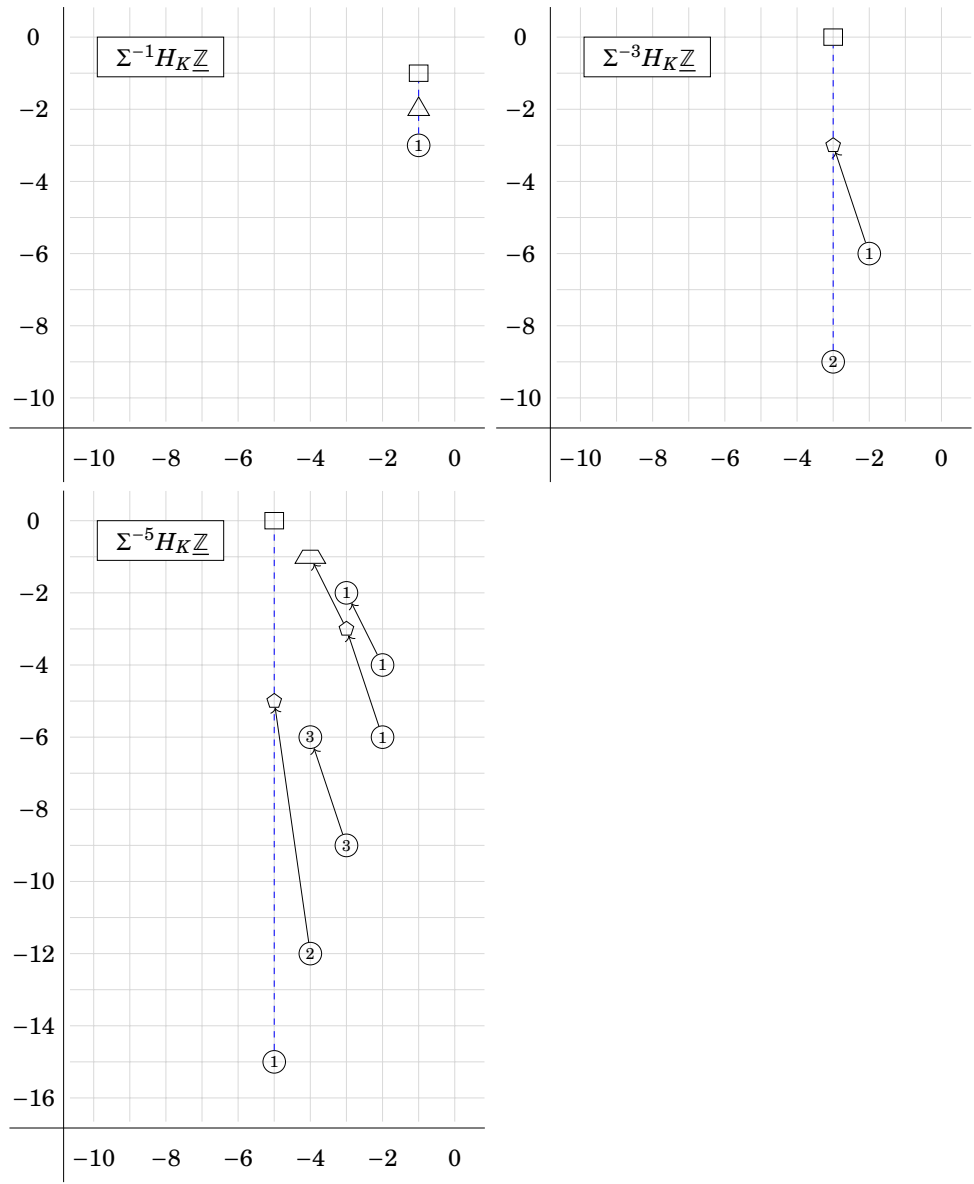
The slice spectral sequence over \(\K\), \(n=-1,-3,-5\).
\end{flushleft}
\end{fig}
\begin{adjustbox}{minipage=[r][\textheight][b]{\paperwidth},scale={0.8}}
\printpage[name=SI-1K]
\printpage[name=SI-3K]\newline 
\printpage[name=SI-5K]
\end{adjustbox}

\clearpage

\begin{fig}\label{fig:SSS:-7..-9:HZ}
\begin{flushleft}
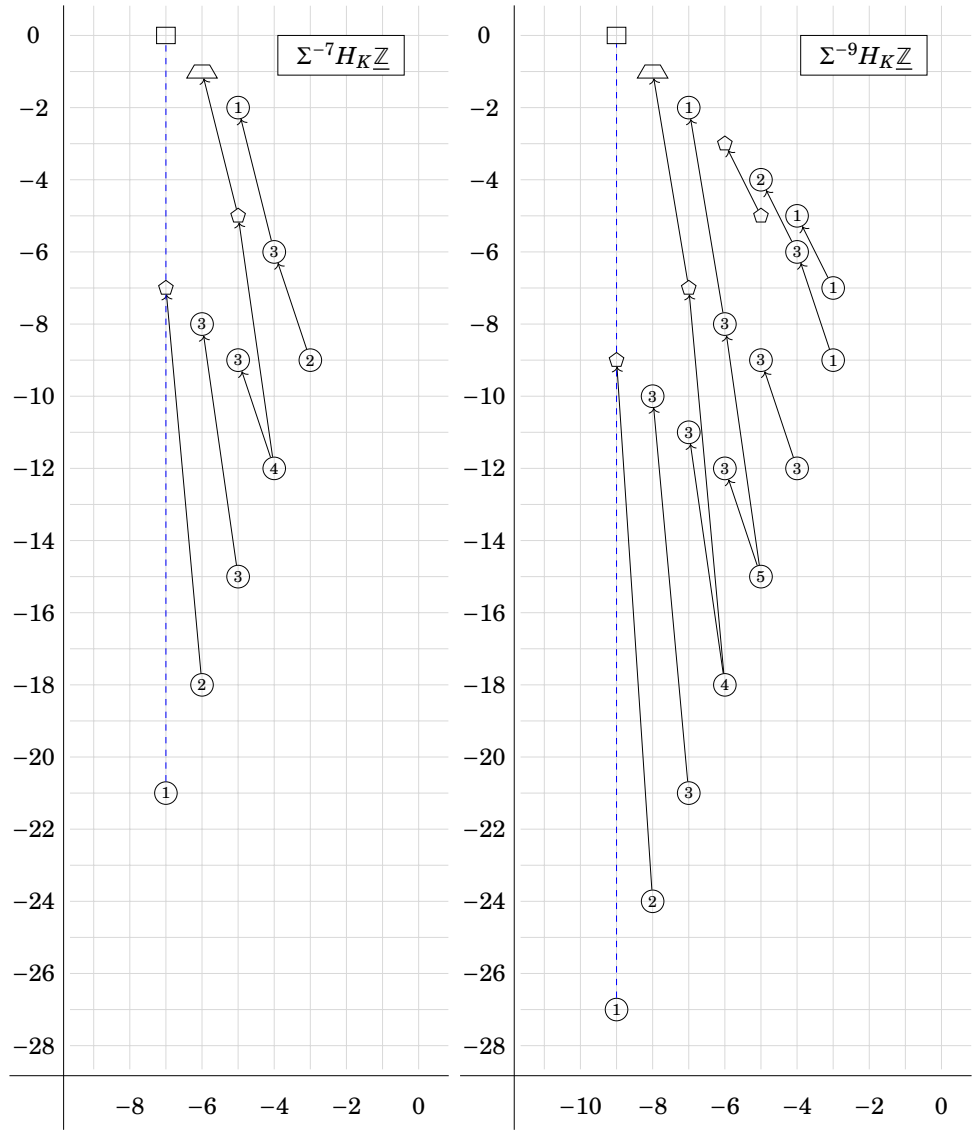
The slice spectral sequence over \(\K\), \(n=-7,-9\).
\end{flushleft}
\end{fig}
\begin{adjustbox}{scale=0.8}
\printpage[name=SI-7K]
\printpage[name=SI-9K]
\end{adjustbox}

\clearpage

\begin{fig}\label{fig:SSS:6..8:HZ}
\begin{flushleft}
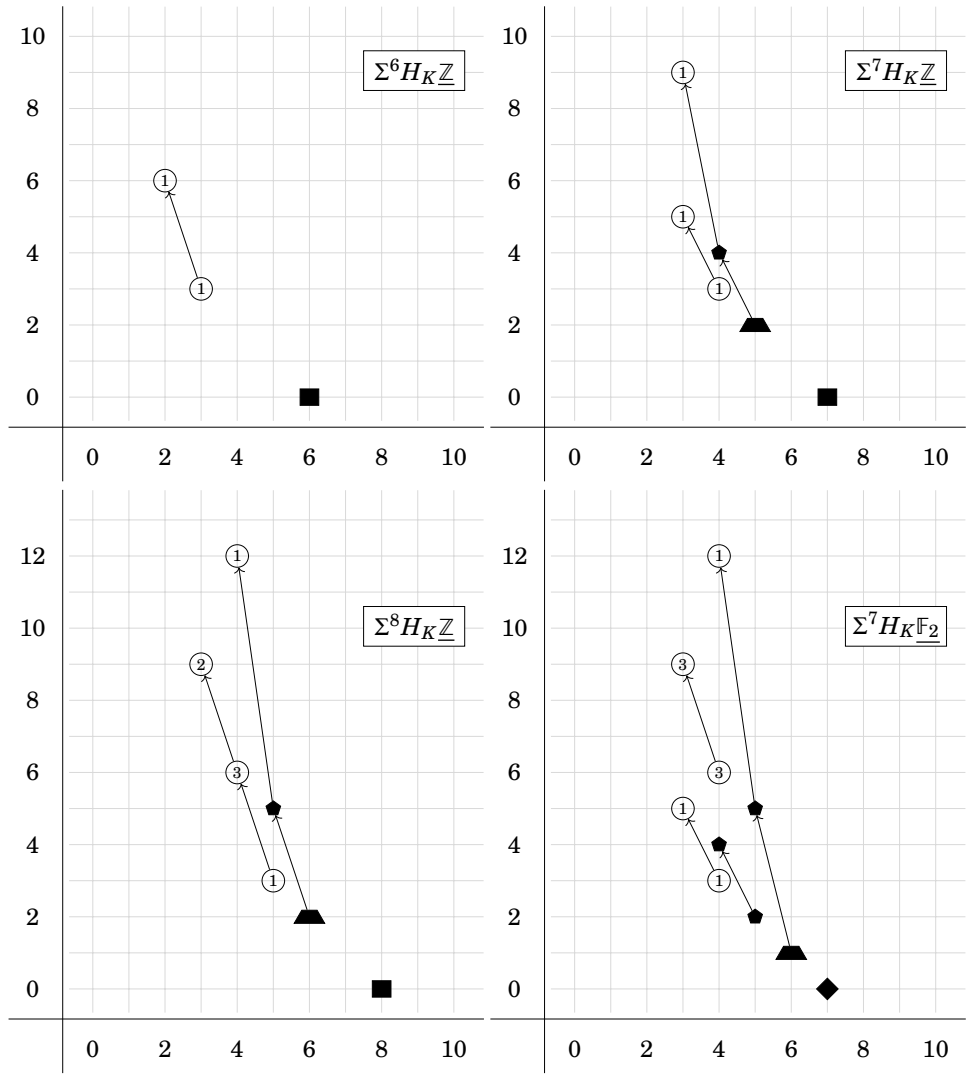
The slice spectral sequence over \(\K\), \(n=6,7,8\).
\end{flushleft}
\end{fig}
\begin{adjustbox}{minipage=[r][\textheight][b]{\paperwidth},scale={0.8}}
\printpage[name=SI6K]
\printpage[name=SI7K] \newline 
\printpage[name=SI8K]
\printpage[name=SI7KF]
\end{adjustbox}

\vfill

\clearpage

\begin{fig}\label{fig:SSS:10..11:HZ}
\begin{flushleft}
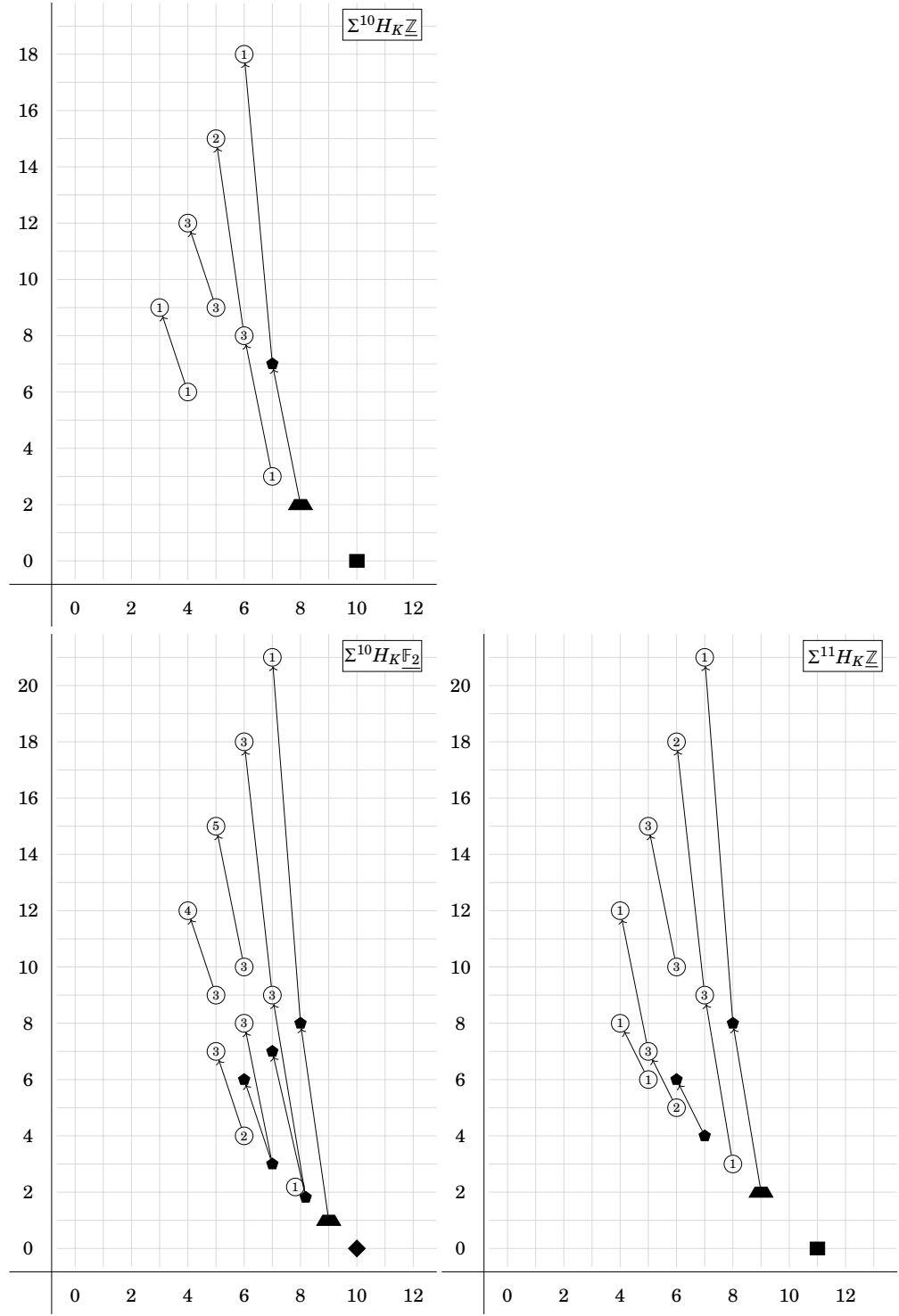
The slice spectral sequence over \(\K\), \(n=10,11\).
\end{flushleft}
\end{fig}
\begin{adjustbox}{scale={0.7}}
\printpage[name=SI10K]
\end{adjustbox} \newline 
\begin{adjustbox}{scale=0.7}
\printpage[name=SI10KF] 
\printpage[name=SI11K]
\end{adjustbox}

\clearpage

\begin{fig}\label{fig:SSS:12:HZ}
\begin{flushleft}
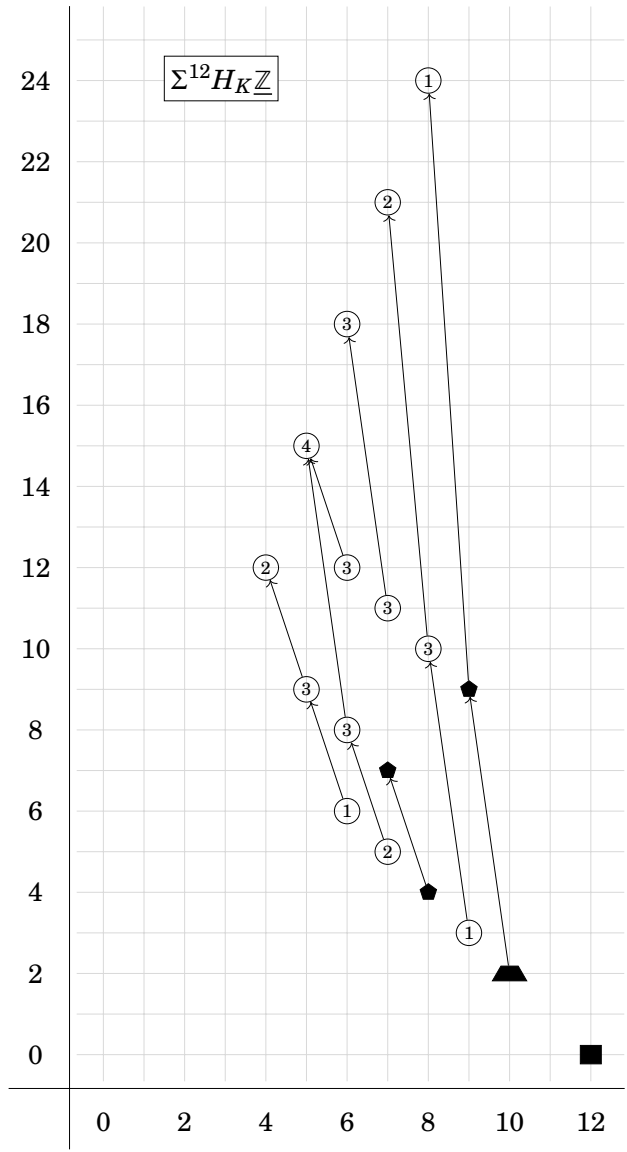
The slice spectral sequence over \(\K\), \(n=12\).
\end{flushleft}
\end{fig}
\begin{adjustbox}{scale=0.9}
\printpage[name=SI12K]
\end{adjustbox}

\clearpage

\begin{fig}\label{fig:SSS:14:HZ}
\begin{flushleft}
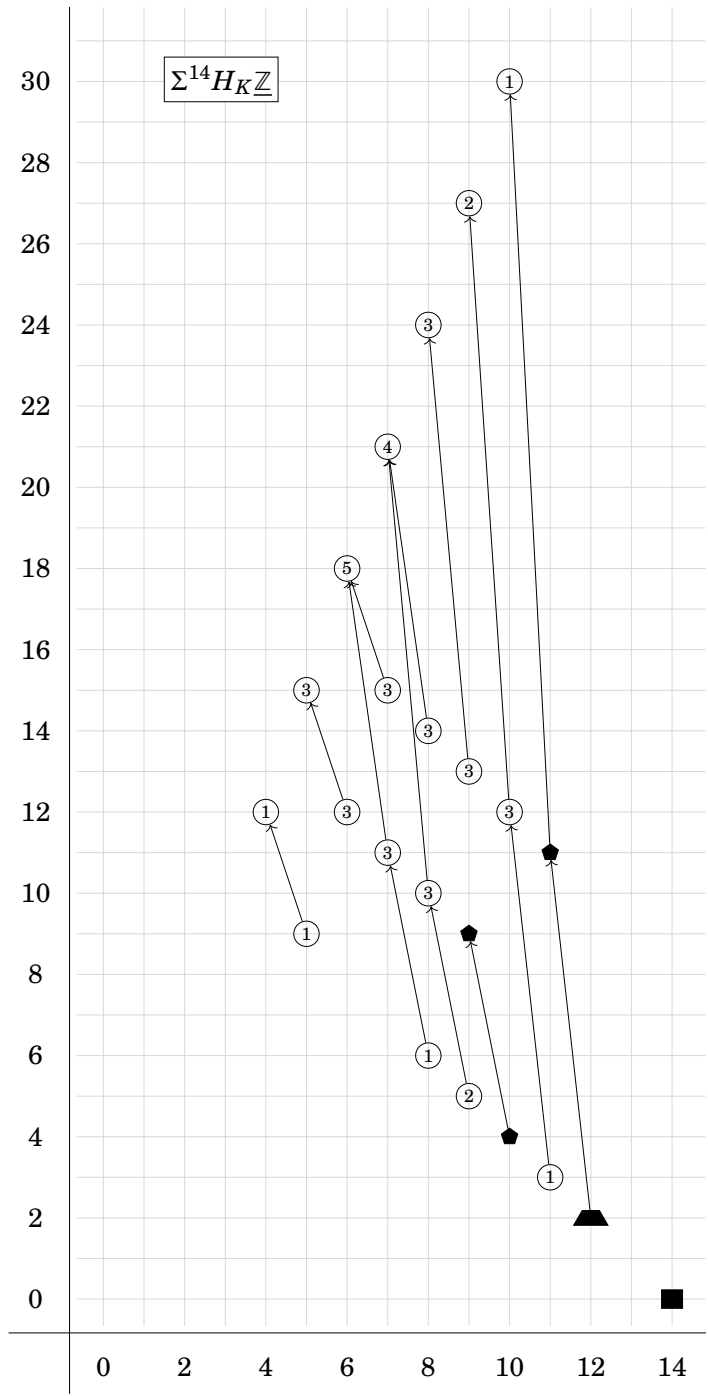
The slice spectral sequence over \(\K\), \(n=14\).
\end{flushleft}
\end{fig}
\begin{adjustbox}{scale=0.9}
\printpage[name=SI14K]
\end{adjustbox}

%%%%%%%%%%%%%%%%%%%%%%%%%%%%%%%%%%%%% Bibliography %%%%%%%%%%%%%%%%%%%%%%%%%%%%%%%%%%%%%

\end{document}